\documentclass[fleqn,reqno,11pt,a4paper,final]{amsart}
\usepackage[a4paper,left=20mm,right=20mm,top=20mm,bottom=20mm,marginpar=20mm]{geometry}
\usepackage{amsaddr}
\usepackage{amsmath}
\usepackage{amssymb}
\usepackage{amsthm}
\usepackage{mathtools}
\usepackage{amscd}
\usepackage[ansinew]{inputenc}
\usepackage{cite}
\usepackage{bbm}
\usepackage{color}
\usepackage[english=american]{csquotes}
\usepackage[final]{graphicx}
\usepackage{hyperref}
\usepackage{calc}
\usepackage{mathptmx}
\usepackage{tikz}
\usepackage{graphicx}
\usepackage{stix}

\graphicspath{{../Pictures/}}

\numberwithin{equation}{section}

\newtheoremstyle{thmlemcorr}{10pt}{10pt}{\itshape}{}{\bfseries}{.}{10pt}{{\thmname{#1}\thmnumber{ #2}\thmnote{ (#3)}}}
\newtheoremstyle{thmlemcorr*}{10pt}{10pt}{\itshape}{}{\bfseries}{.}\newline{{\thmname{#1}\thmnumber{ #2}\thmnote{ (#3)}}}
\newtheoremstyle{remexample}{10pt}{10pt}{}{}{\bfseries}{.}{10pt}{{\thmname{#1}\thmnumber{ #2}\thmnote{ (#3)}}}
\newtheoremstyle{ass}{10pt}{10pt}{}{}{\bfseries}{.}{10pt}{{\thmname{#1}\thmnumber{ A#2}\thmnote{ (#3)}}}

\theoremstyle{thmlemcorr}
\newtheorem{theorem}{Theorem}
\numberwithin{theorem}{section}
\newtheorem{lemma}[theorem]{Lemma}
\newtheorem{corollary}[theorem]{Corollary}
\newtheorem{proposition}[theorem]{Proposition}

\newtheorem{definition}[theorem]{Definition}

\theoremstyle{thmlemcorr*}
\newtheorem{theorem*}{Theorem}
\newtheorem{lemma*}[theorem]{Lemma}
\newtheorem{corollary*}[theorem]{Corollary}
\newtheorem{proposition*}[theorem]{Proposition}
\newtheorem{problem*}[theorem]{Problem}
\newtheorem{conjecture*}[theorem]{Conjecture}
\newtheorem{definition*}[theorem]{Definition}
\newtheorem{assumption*}[theorem]{Assumption}

\theoremstyle{remexample}
\newtheorem{remark}[theorem]{Remark}

\theoremstyle{ass}

\newcommand{\Acal}{\mathcal{A}}

\newcommand{\Dcal}{\mathcal{D}}

\newcommand{\Fcal}{\mathcal{F}}

\newcommand{\Lcal}{\mathcal{L}}

\newcommand{\Scal}{\mathcal{S}}

\newcommand{\N}{\mathbb{N}}
\newcommand{\R}{\mathbb{R}}
\newcommand{\C}{\mathbb{C}}

\newcommand{\Z}{\mathbb{Z}}

\newcommand{\eps}{\epsilon}

\def\XXint#1#2#3{{\setbox0=\hbox{$#1{#2#3}{\int}$}
\vcenter{\hbox{$#2#3$}}\kern-.5\wd0}}

\renewcommand{\eps}{\varepsilon}
\renewcommand{\epsilon}{\varepsilon}
\renewcommand{\phi}{\varphi}

\begin{document}


\title[Long-time asymptotics of solutions to fourth-order quasilinear degenerate-parabolic problems]{Long-time asymptotics and regularity estimates for weak solutions to a doubly degenerate thin-film equation in the Taylor--Couette setting
}

\author{Christina Lienstromberg \& Juan J. L. Vel\'{a}zquez}
\address{Institute of Applied Mathematics, University of Bonn, Endenicher Allee~60, 53115 Bonn, Germany}
\email{lienstromberg@iam.uni-bonn.de} 
\email{velazquez@iam.uni-bonn.de}

\begin{abstract}
We study the dynamic behaviour of solutions to a fourth-order quasilinear degenerate parabolic equation for large times arising in fluid dynamical applications. The degeneracy occurs both with respect to the unknown and with respect to the sum of its first and third spatial derivative.
The modelling equation appears as a thin-film limit for the interface separating two immiscible viscous fluid films confined between two cylinders rotating at small relative angular velocity. More precisely, the fluid occupying the layer next to the outer cylinder is considered to be Newtonian, i.e. it has constant viscosity, while we assume that the layer next to the inner cylinder is filled by a shear-thinning power-law fluid.

Using energy methods, Fourier analysis and suitable regularity estimates for higher-order parabolic equations, we prove global existence of positive weak solutions in the case of low initial energy. Moreover, these global solutions are polynomially stable, in the sense that interfaces which are initially close to a circle, converge at rate $1/t^{1/\beta}$ for some $\beta > 0$ to a circle, as time tends to infinity. 

In addition, we provide regularity estimates for general nonlinear degenerate parabolic equations of fourth order.
\end{abstract}
\vspace{4pt}







\maketitle
\bigskip


\bigskip

\noindent\textsc{Acknowledgement. } The authors have been supported by the Deutsche Forschungsgemeinschaft (DFG, German Research Foundation) through the collaborative research centre `The mathematics of emerging effects' (CRC 1060, Project-ID  211504053) and the Hausdorff Center for Mathematics (GZ 2047/1, Project-ID 390685813).

\bigskip

\section{Introduction}
In this paper, we study regularity properties and the asymptotic behaviour of positive weak solutions to the evolution problem 
\begin{equation} \label{eq:PDE}
    \begin{cases}
        \partial_t h + \partial_\theta \bigl(h^{\alpha+2} |\partial_\theta h + \partial_\theta^3 h|^{\alpha-1} (\partial_\theta h + \partial_\theta^3 h)\bigr) = 0, 
        & 
        t > 0,\ \theta \in S^1 = (0,2\pi)
        \\
        h(0,\cdot) = h_0(\cdot), 
        & 
        \theta \in S^1,
    \end{cases}
\end{equation}
with periodic boundary conditions for exponents $\alpha > 1$. Observe that this is a
quasilinear problem of fourth order which is doubly-degenerate in the sense that we lose uniform parabolicity if either $h=0$ or $\partial_\theta h + \partial_\theta^3 h = 0$. Problem \eqref{eq:PDE} appears for instance as a model for a surface-tension driven flow of a non-Newtonian liquid in a circular geometry. More precisely, it may describe the evolution of an interface separating two immiscible viscous fluid films located between two concentric rotating cylinders as sketched in Figure \ref{figdom}. Here, the liquid layer located next to the outer cylinder is assumed to be filled by a Newtonian fluid with constant viscosity $\mu_+ > 0$, whereas the layer located next to the inner cylinder is assumed to be filled by a non-Newtonian fluid with a strain rate-dependent power-law rheology. 
\begin{center}
   \begin{figure}[h]
    \center
    \includegraphics[width=90mm]{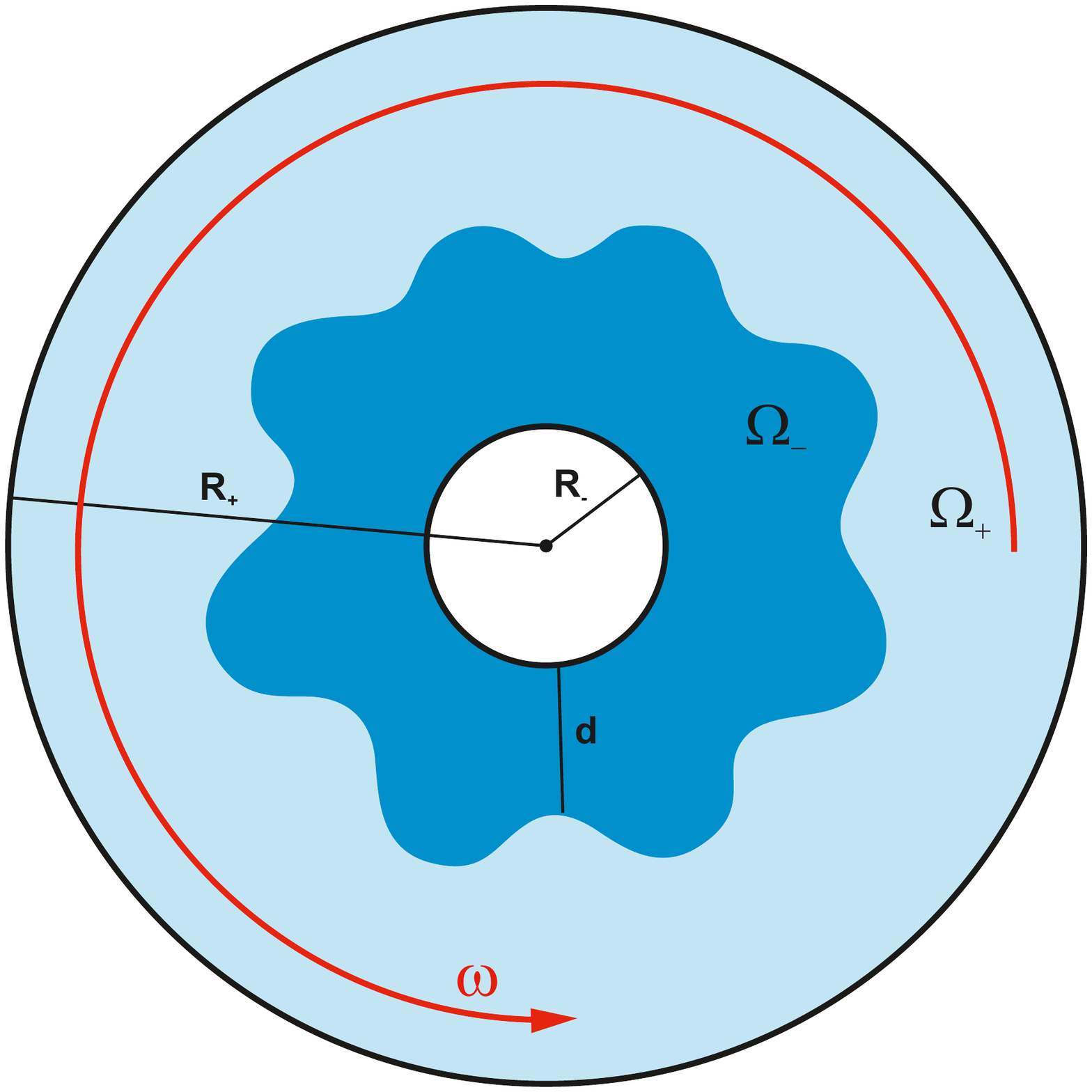}
    \caption{Two-phase flow of liquids confined between two rotating cylinders}
    \label{figdom}
    \end{figure}
\end{center}
That is, the viscosity $\mu_-$ of the non-Newtonian fluid is a function of the strain rate $s \in \R$ and it is given by
\begin{equation}\label{eq:power-law}
    \mu(\left|s\right|) = \mu_0 \left|s\right|^{\frac{1}{\alpha}-1}, \quad s \in \R.
\end{equation}
Fluids whose rheology is defined by \eqref{eq:power-law} are also called \emph{Ostwald--de Waele fluids}. The parameter $\mu_0 > 0$ is the so-called flow-consistency index and the parameter $\alpha > 0$ is the flow behaviour exponent which determines the rheological behaviour of the fluid. More precisely, fluids of the form \eqref{eq:power-law} are called Newtonian if $\alpha=1$. For $\alpha < 1$ the corresponding fluids are called shear-thickening as their viscosity increases with increasing shear rate. Conversely, if $\alpha > 1$, the viscosity decreases with increasing shear rate and the fluids are called shear-thinning fluids. The Newtonian case has been treated in \cite{PV1} and the shear-thickening case has been studied in \cite{LPCV:2022}. More details on these two cases are provided in Subsection \ref{ssec:known_results}. 

The main subject of the present work is to describe rigorously the asymptotic behaviour of positive weak solutions to problem \eqref{eq:PDE} for large times in the shear-thinning regime $\alpha > 1$. However, it turns out that such a precise description of the asymptotic behaviour of solutions requires good (interior) regularity estimates. These regularity estimates constitute a relevant part of this paper and might be useful in a broader context of fourth-order nonlinear degenerate parabolic partial differential equations. In the case of second-order parabolic equations similar regularity estimate have been developed, cf. for instance \cite{DiBenedetto}.

\medskip

\noindent\textbf{\textsc{Notation. }}
We briefly introduce some notation used throughout this paper. We identify $S^1$ with the interval $[0,2\pi]$ and functions $\phi \in L_p(S^1)$ with $2\pi$-periodic functions $\phi \in L_{p,\text{loc}}(\R)$, where $L_p(\R)$ denotes the usual Lebesgue space. 
Moreover, by $W^k_p(S^1)$ and $H^k(S^1)$ we denote the usual Sobolev spaces, defined as the closure of the restriction of $2\pi$-periodic functions in $C^{\infty}(\R)$ to the interval $[0,2\pi]$ with respect to the $W^k_p((0,2\pi))$-norm, respectively the $H^k((0,2\pi))$-norm. 
Finally, we consider $W^k_p(S^1)$ and $H^k(S^1)$ as closed subspaces of the complex spaces $W^{k}_p(S^1;\C)$ and  $H^{k}(S^1;\C)$, respectively. This allows us in particular to write any function $v\in H^k(S^1)$ in terms of its Fourier series
\begin{equation*}
	v(\theta)=\sum_{n\in\Z} v_n e^{in\theta},
	\quad \theta \in S^1, 
\end{equation*}
where $v_n=\bar{v}_{-n}$.

\medskip

\subsection{Known results and the Newtonian case $\alpha = 1$} \label{sec:known_results_Newtonian}
In the case of two Newtonian fluid films, the evolution equation for the separating interface becomes
\begin{equation} \label{eq:intro_Newtonian}
    \begin{cases}
        \partial_t h + \partial_\theta\left( h^3 (\partial_\theta h + \partial_\theta^3 h)\right)
        =
        0,
        &
        t > 0,\ \theta \in S^1
        \\
        h(0,\cdot) = h_0(\cdot),
        &
        \theta \in S^1.
    \end{cases}
\end{equation}
This problem is studied in \cite{PV1}, where the authors prove the following results. 
\begin{itemize}
    \item Global existence and uniqueness of positive solutions 
        $$h \in C\bigl([0,\infty);H^4(S^1)\bigr) \cap C^1\bigl((0,\infty);H^4(S^1)\bigr)
        $$
        for positive initial values $h_0 \in H^4(S^1)$ that are close to circles centered at the origin (i.e. at the common center of the rotating cylinders).
    \item Steady states are given by functions of the form $h(\theta) = \bar{h}_0 + \kappa_1 \cos(\theta) + \kappa_{-1} \sin(\theta)$ with $\kappa_{\pm 1} \in \R$. Thus, for every $\bar{h}_0 > 0$ this defines a two-dimensional manifold of steady states.
    \item The manifold of steady states is locally stable and solutions in a neighbourhood of it approach it exponentially fast.
\end{itemize}
Moreover, in \cite{PV1} the authors also study the more involved model
\begin{equation} \label{eq:intro_Newtonian_Burgers}
    \begin{cases}
        \partial_t h + \partial_\theta\left(\tfrac{h^2}{2}\right)
        +
        \partial_\theta\left( h^3 (\partial_\theta h + \partial_\theta^3 h)\right)
        =
        0,
        &
        t > 0,\ \theta \in S^1
        \\
        h(0,\cdot) = h_0(\cdot),
        &
        \theta \in S^1,
    \end{cases}
\end{equation}
for which they prove the following results.
\begin{itemize}
    \item Global existence and uniqueness of solutions  for positive initial values in $H^4(S^1)$ that are close to circles centered at the origin.
    \item The only steady states $h$ that are close to constants are given by positive constant solutions $h\equiv\bar{h}_0 >0$.
    \item As $t\to \infty$, solutions which are initially close to a circle centered at the origin converge at rate $1/t$ to a circle and, in addition the center of the circle spirals at rate $1/\sqrt{t}$ to the origin.
    \item Non-existence of travelling-wave solutions, where the change of $h$ is a small along the whole wave.
\end{itemize}
Problem \eqref{eq:intro_Newtonian} is also studied, both analytically and numerically, in \cite{renardy}, where the author is concerned with the stability of the two-phase flow for different ranges of viscosities, densities and surface tensions, when gravity is neglected.
Much more literature is available dealing with the dynamics of one single Newtonian fluid between two concentric rotating cylinders. Among other physical and mathematical contributions it is worthwhile to mention \cite{Baumert,Chandra,Chossat,Drazin,Schlichting,renardy}.

\medskip

\subsection{Known results in the shear-thickening case $0 < \alpha < 1$} \label{ssec:known_results}
We briefly summarise what is known about problem \eqref{eq:PDE} in the case of flow-behaviour exponents $0 < \alpha < 1$, i.e. when the non-Newtonian fluid located next to the interior cylinder is shear-thickening. This problem is studied in \cite{LPCV:2022}, where the following results are proved.
\begin{itemize}
    \item Local-in-time existence of positive weak solutions (this has even been proved for all flow-behaviour exponents $\alpha > 0$).
    \item Solutions which are initially close to a circle centered at the origin converge to a circle \emph{in finite time} $0 < t_\ast < \infty$.
    \item These positive weak solutions exists globally in time and keep their circular shape for all times $t \geq t_\ast$. In particular, the center of the limit circle, which does in general not coincide with the origin, does not move for times $t \geq t_\ast$.
\end{itemize}

\medskip

\subsection{The shear-thinning case $\alpha > 1$ -- main results of the paper}
As already mentioned, the contribution of the present paper is twofold. The main issue is to describe the asymptotic behaviour of positive weak solutions to \eqref{eq:PDE} in the shear-thinning regime $\alpha > 1$. However, it turns out that this asymptotic analysis requires good interior regularity estimates for the corresponding solutions. 

Concerning the asymptotic behaviour of positive weak solutions to \eqref{eq:PDE} with $\alpha > 1$ the main results of this paper are the following.
\begin{itemize}
    \item Steady states of \eqref{eq:PDE} are given by functions of the form $h(\theta) = \bar{h}_0 + \kappa_{-1} \cos(\theta) + \kappa_1 \sin(\theta),\ \kappa_{\pm 1} \in \R$, which, in the limit of a vanishing height of the thin non-Newtonian fluid layer, correspond to circular interfaces, not necessarily centered at the origin.
    \item For positive initial data which is $\eps$-close to the manifold of steady states in $H^1(S^1)$, problem \eqref{eq:PDE} possesses globally in time defined positive weak solutions.
    \item The shapes of these global solutions stay, for all times $t \geq 0$, close to a circle. The center of this circle remains for all times $t \geq 0$ close to the origin, i.e. to the common center of the rotating cylinders. The distance of the circle's center to the origin is bounded by $C \eps \bigl(1 + \log(1/\eps^{1-\alpha})\bigr)$.
    \item As $t \to \infty$, solutions converge with a power-law decay $1/t^\frac{1}{\alpha-1}$ to a circle (not necessarily centered at the origin).
\end{itemize}
It is worthwhile to mention that the convergence to a circle relies on Fourier methods and energy estimates only. However, in order to get information on the position of the circle's center, one needs to control the first Fourier modes $n = \pm 1$. Using the differential equation, one observes that this requires to control the $L_1$ in time and $L_{\alpha}$ in space norm of the solution's third derivative. This lead us to prove the following interior regularity result for general fourth-order degenerate parabolic equations of the form
\begin{equation} \label{eq:parabolic_intro}
    \partial_t v + \partial_x\bigl(\Phi(v,\partial_x v,\partial_x^2 v, \partial_x^3 v)\bigr) + S(t,x) = 0,
    \quad
    t \in (1/4,1),\, x \in (-1,1),
\end{equation}
with a source term $S \in L_1\bigl((1/4,1);H^2([-1,1])\bigr)$, cf. Theorem \ref{thm:reg_estimate_general}. This regularity result states the following.
\begin{itemize}
    \item Under some structural conditions on the nonlinearity of the function $\Phi$ bounded weak solutions to \eqref{eq:parabolic_intro} satisfy the following interior regularity estimate 
    \begin{equation*}
        \int_{1/2}^1 \int_{-1/8}^{1/8} |\partial_x^3 v|^{\alpha+1}\, dx\, dt
        \leq
        C + C \|S\|_{L_1((1/4,1);H^2([-1,1]))},
    \end{equation*}
    where $C > 0$ is a generic constant related to the norm $\|v\|_{L_\infty((1/4,1)\times[-1,1])}$.
\end{itemize}
Moreover, based on this regularity estimates, an iterative argument allows us to derive an $L_1((1/2,1);L_\alpha(S^1))$-estimate for the solution's third derivative (instead of $\alpha+1$), cf. Theorem \ref{thm:improved_estimate_alpha}. Observe that this is an improvement for regions in which the third derivative is small. The corresponding result is the following (cf. Theorem \ref{thm:improved_estimate_alpha}).
\begin{itemize}
    \item For positive initial data which is $\eps$-close to a circle $h_c$ in $H^1(S^1)$, solutions satisfy
    \begin{equation*}
        \int_{t/2}^{t} \int_{S^1} |\partial_\theta^3(h - h_c)|^{\alpha}\, d\theta\, dt
        \leq
        C 
        \frac{\eps}{\bigl(1 + \bar{C} \eps^{\alpha-1}\bigr)^\frac{1}{\alpha-1}}
    \end{equation*}
    as long as they exists and do not touch the inner cylinder. 
\end{itemize}


\medskip
\noindent\textbf{\textsc{Related results.}}
The evolution equation \eqref{eq:PDE} belongs to a class of shear-thinning non-Newtonian thin-film equations that has been studied in different settings in the literature, see for instance \cite{AG:2002,AG:2004,K:2001a,King:2001b,LM:2020,MBB:1965,SW:1994}.

For multi-phase thin-film problems we refer the reader to the articles \cite{Belinchon,Escher,Laurencot}, all of them treating the Newtonian case. Moreover, a large number of two-fluid viscous flows in various geometrical settings is described in the textbook \cite{renardybook}.

A more detailed review on the above mentioned works is provided in the introduction of our previous paper \cite{LPCV:2022}.
\medskip

\subsection{Derivation of the equation from the two-phase Taylor--Couette setting}
We briefly recall from our previous work \cite{LPCV:2022}  how problem \eqref{eq:PDE} can be derived as a model equation for the interface separating two viscous, non-Newtonian immiscible fluids that are confined between two concentric cylinders rotating at a small relative velocity. For a detailed derivation of the model we refer the reader to \cite[Section 2]{LPCV:2022}.

\medskip

\noindent\textbf{\textsc{Dimensionless Navier--Stokes system for one Newtonian and one non-Newtonian fluid.}}

We consider two immiscible fluid films confined between two concentric cylinders of radii $\R_{\pm} > 0$, centered at the origin and rotating at different angular velocities. A sketch of the setting is given in Figure \ref{figdom}. 
We assume that the hydrodynamic behaviour of the two fluids is determined by the Navier--Stokes equations. After a suitable  rescaling, cf. \cite[Eq. (2.2)]{LPCV:2022}, such that the innerl cylinder is at rest and has radius $1$ and the outer cylinder rotates at angular velocity $1$ and has radius $\eta =R_+/R_- > 1$, we obtain the dimensionless Navier--Stokes system
\begin{equation} \label{eq:Navier-Stokes}
	\begin{cases}
		\rho \left(\textbf{u}^-_t + (\textbf{u}^-\cdot \nabla) \textbf{u}^-\right)
		=
		- \nabla p^- 
		+\frac{2\mu}{\text{Re}} \nabla\cdot \bigl( \mu_-\bigl(\tau \left\|\textbf{D}^-\textbf{u}^-\right\|\bigr) \textbf{D}^-\textbf{u}^-\bigr)
		&
		\text{in } \Omega_-(t)
		\\
		\bigl(\textbf{u}^+_t + (\textbf{u}^+\cdot \nabla) \textbf{u}^+\bigr)
		=
		- \nabla p^+ 
		+\frac{1}{\text{Re}} \Delta \textbf{u}^+
		&
		\text{in } \Omega_+(t)
		\\
		\nabla\cdot \textbf{u}^\pm
		=
		0
		&
		\text{in } \Omega_\pm(t).
	\end{cases}
\end{equation}
See Figure \ref{figdomthin} for a sketch of the problem in dimensionless variables. In the dimensionless Navier--Stokes system \eqref{eq:Navier-Stokes} we use the notation
$\textbf{u}^\pm(t,x) = (u^\pm(t,x),v^\pm(t,x))$ for the nondimensional velocity field at time $t > 0$ and position $x \in \R^2$ and $p^\pm$ for the pressure. Moreover, $\textbf{D}^- \textbf{u}^-=\frac{1}{2}\bigl(\nabla \textbf{u}^- + (\nabla \textbf{u}^-)^T\bigr)$ denotes the symmetric gradient of the velocity field $\textbf{u}^-$ of the inner fluid and $\left\|\textbf{D}^- \textbf{u}^-\right\| = \sqrt{\text{tr}(|\textbf{D}^- \textbf{u}^-|^2)}$. 
Denoting by $\rho_{\pm} > 0$ the densities of the inner (-), respectively outer (+) fluid, we have that $\rho = \rho_{-}/ \rho_{+}$. 
Furthermore, the fluid (+) next to the outer cylinder is assumed to be Newtonian with constant viscosity $\mu_+ > 0$, while the fluid (-) next to the inner cylinder is assumed to be a shear-thinning power-law fluid. More precisely, the constitutive law for its viscosity is $\mu_-(s) = \mu_0 |s|^{\frac{1}{\alpha}-1},\, s \in \R$, with a \emph{flow-behaviour exponent} $\alpha > 0$ and a \emph{characteristic viscosity} $\mu_0 > 0$ which is the fluid's viscosity for 
$\tau \|\textbf{D}^- \textbf{u}^-\|$ 
of order one. Here, $\tau > 0$ is the outer characteristic time induced by the relative angular velocity of the rotating cylinders. Note that having a flow-behaviour exponent $\alpha > 1$ reflects the shear-thinning behaviour of the fluid as it guarantees the viscosity to decrease with increasing strain rate. Finally, $\text{Re} = \rho_{+}\omega R_{-}^2/\mu_+$ denotes the Reynolds number, reflecting the ratio of inertia and viscous forces.

We introduce polar coordinates $\textbf{x} = (x_1,x_2) = r(\cos \theta, \sin \theta) \in \R^2$ in order to describe spatial positions in the region between the two concentric cylinders.
Therewith, the dimensionless regions $\Omega_-(t)$ and $\Omega_+(t)$, filled by the inner, respectively the outer fluid, are given by
\begin{equation} \label{eq:def_Omega_rescaled}
    \begin{cases}
        \Omega_-(t) = \left\{\textbf{x} \in \R^2;\ 1 < r < 1 + \eps h(t,\theta)\right\} &
	    \\
	    \Omega_+(t) = \left\{\textbf{x} \in \R^2;\ 1 + \eps h(t,\theta) < r < \eta\right\}, &
    \end{cases}
\end{equation}
where $\eps > 0$ denotes the positive dimensionless average height of the non-Newtonian fluid film next to the inner cylinder and $h(t,\theta) > 0$ is the function describing the interface separating the two immiscible liquid films. Throughout the whole paper we assume that the function $h(t,\theta)$ is strictly positive, i.e. the interface of the two fluid films does never touch the inner cylinder. 
The system is complemented with the dimensionless boundary conditions
\begin{equation} \label{eq:NS_boundary_cond}
    \begin{cases}
	    \textbf{u}^-
	    =
	    0,
	    &
	    \textbf{x} \in \partial B_1(0)
	    \\
    	\textbf{u}^+
    	=
    	(-x_2,x_1),
	    &
	    \textbf{x} \in \partial B_{\eta}(0)
	    \\
	    
	    \textbf{u}^-\cdot \textbf{t}
	    =
	    \textbf{u}^+\cdot \textbf{t},
	    &
	    \textbf{x} \in \partial\Omega_{-} \cap \partial\Omega_+
	    \\
	    \textbf{u}^-\cdot \textbf{n}
	    =
	    \textbf{u}^+\cdot \textbf{n}
	    =
	    V_n,
	    &
	    \textbf{x} \in \partial\Omega_{-} \cap \partial\Omega_+
	    \\
	    \textbf{t} \left(\Sigma^+ - \Sigma^-\right) \cdot \textbf{n} = 0,
	    &
	    \textbf{x} \in \partial\Omega_{-} \cap \partial\Omega_+
	    \\
	    \textbf{n} \left(\Sigma^+ - \Sigma^-\right) \cdot \textbf{n} = \gamma \kappa,
	    &
	    \textbf{x} \in \partial\Omega_{-} \cap \partial\Omega_+.
	\end{cases}
\end{equation}
\begin{center}
  \begin{figure}[h]
    \center
    \includegraphics[width=90mm]{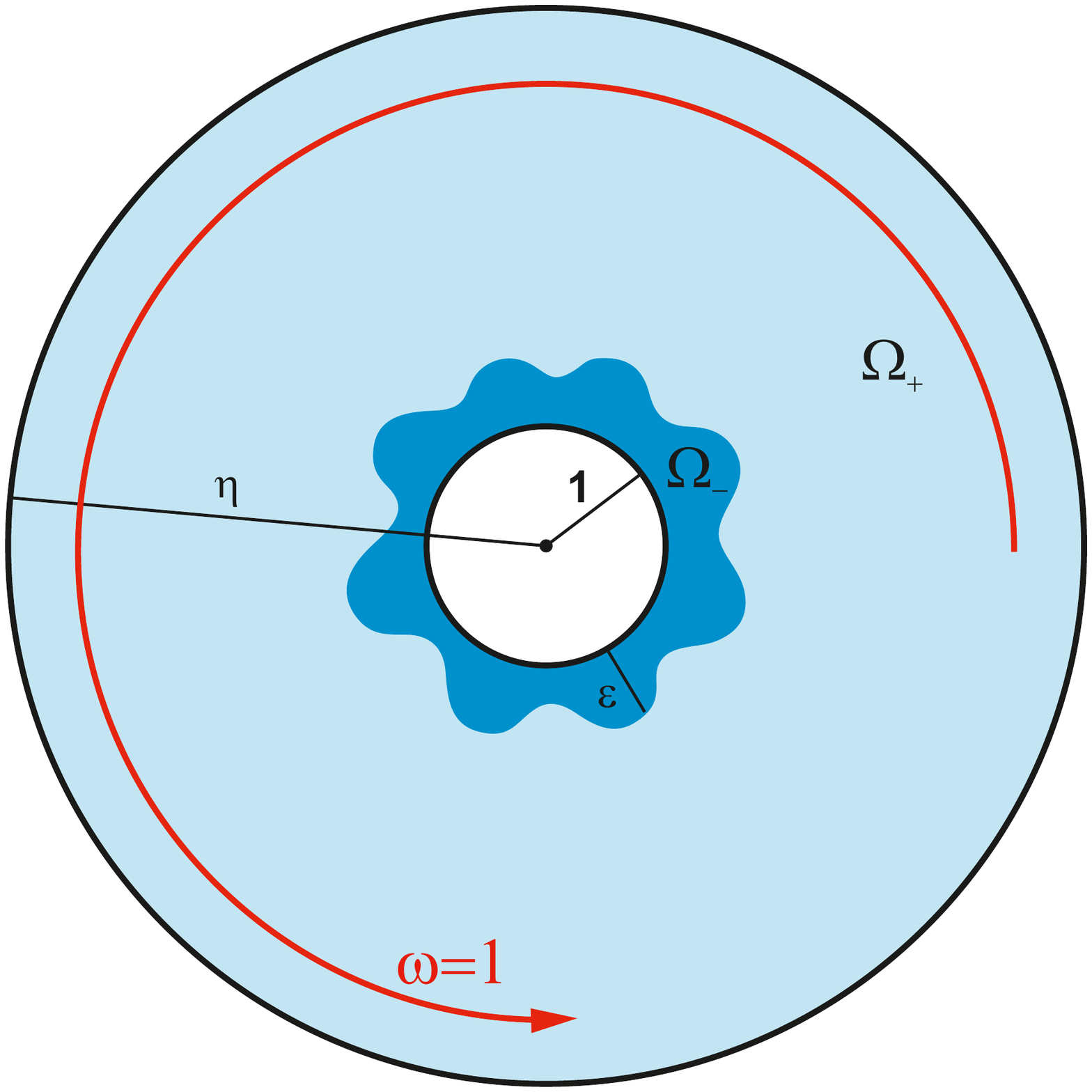}\caption{Problem setting with thin-film next to the inner cylinder in non-dimensional variables}
    \label{figdomthin}
\end{figure}
\end{center}
In the list \eqref{eq:NS_boundary_cond} of boundary conditions, $\Sigma^\pm(\textbf{u},p)$ denotes the stress tensor of the respective fluid $\pm$. Moreover, we use the notation $\textbf{n}$ and $\textbf{t}$ for the normal vector pointing from the region $\Omega_-$ occupied by the inner fluid to the region $\Omega_+$ occupied by the outer fluid and the tangential vector at the interface, respectively. Furthermore, $\kappa$ denotes the dimensionless mean curvature of the interface and $\gamma$ is the dimensionless constant surface tension. We require the tangential stress balance condition and the normal stress balance condition $\eqref{eq:NS_boundary_cond}_5$--$\eqref{eq:NS_boundary_cond}_6$ to be satisfied at the interface $\partial\Omega_{-} \cap \partial\Omega_+$. 

\medskip

\noindent\textbf{\textsc{Lubrication approximation and Taylor--Couette flows -- the non-Newtonian thin-film equation \eqref{eq:PDE}. }}
Taylor--Couette flows describe the dynamics of viscous fluids confined between two concentric cylinders.
It was experimentally observed by Maurice Couette at the end of the 19th century that the flow of a single Newtonian fluid confined between rotating cylinders is steady when the relative angular velocity of the cylinders is small and the distance between the cylinders is small compared to their radii. In his paper \cite{Couette} from 1890 Maurice provided an explicit formula for this so-called Couette flow. It has been proved mathematically in 1923 by G. I. Taylor \cite{Taylor} that the Couette flow becomes unstable as soon as the relative angular velocity of the cylinders exceeds a certain critical value and that the flow becomes more turbulent with increasing relative velocity of the rotating cylinders.

In the present paper, we are interested in the situation in which the volume of the fluid film $\Omega_-(t)$ is much smaller than the volume of $\Omega_+(t)$ such that the dynamics of the two-fluid flow is described by a small perturbation of the Taylor--Couette flow for one single fluid confined between two cylinders. Taking the formal limit $\eps \to 0$ and using formal matched asymptotics with the Taylor--Couette flow for the Newtonian fluid film, one may derive explicit expressions for the pressure $p^\pm$ as well as for the velocity fields $\textbf{u}^\pm$ of each of the fluids. The underlying Navier--Stokes system is thus reduced to a single evolution equation for the interface $h$ separating the two immiscible fluids.
Mathematically this corresponds to restricting to the leading-order system in $\eps$, i.e. to applying the so-called 
lubrication approximation \cite{Ockendon,giacomelli_otto,Gunther}.

When deriving thin-film equations from the Navier--Stokes system in the Taylor--Couette setting it turns out that there is a subtle interplay between different physical quantities. In particular, the effects of surface tension, the shear forces induced by the rotation of the cylinder and the characteristic stresses of the non-Newtonian have a strong impact on the form of the resulting evolution equation for the interface. In \cite{LPCV:2022} a more evolved equation is derived for the situation in which all three effects -- surface tension, shear induced by the rotation of the cylinders and the characteristic stress of the non-Newtonian fluid -- are comparable. 
In the present paper, we study the equation \eqref{eq:PDE} which is obtained in the following  asymptotic regime
\begin{itemize}
    \item The effects of surface tension dominate the effects of the shear forces induced by the rotation of the cylinder;
    \item the surface tension dominates the characteristic stresses of the non-Newtonian fluid or vice versa.
\end{itemize}
Moreover, it turns out that the ratio of the surface tension and the dimensionless height $\eps$ of the non-Newtonian fluid film strongly affects the nature of the resulting evolution for the separating interface.
In order to be able to observe nontrivial dynamics, we thus assume that the surface tension scales with the dimensionless average thickness $\eps$ as
\begin{equation}\label{eq:surface_tension}
    \gamma\approx\frac{b}{\varepsilon^{2}}\quad\textit{as}\quad\varepsilon\to 0, 
\end{equation} 
where $b > 0$ is a constant of order one.

\medskip

\subsection{Outline of the paper}
We close the introduction by a brief outline of the paper. In Section \ref{sec:shear-thinning} we first introduce some frequently used notation, definitions and results concerning solutions to \eqref{eq:PDE}. Moreover, we discuss steady state solutions to \eqref{eq:PDE} and their circular shape, when lubrication approximation is applied. At the end of Section \ref{sec:shear-thinning} we give a precise formulation of the main result concerning the asymptotic behaviour of positive weak solutions to \eqref{eq:PDE} which are initially close to a circle centered at the origin, cf. Theorem \ref{thm:main_results}. 
In Section \ref{sec:energy_estimates} we introduce a suitable regularised version of \eqref{eq:PDE}. For this regularised problem we prove local existence of positive (strong) solutions for positive initial data. Moreover, we derive a differential inequality for the corresponding energy functional which is then used to show that, if the shape of the interface is initially close to a circle (centered at the common center of the rotating cylinders), then, as long as the solution is strictly bounded away from zero and bounded above, it remains in an $\eps$-neighbourhood of a circle centered at the origin and decays like $1/t^\frac{1}{\alpha-1}$. However, at this point we do not have any information on the position of the circle's center. 
Section \ref{sec:regularity_estimates} provides interior regularity estimates for general nonlinear degenerate parabolic problems of fourth-order as they often appear in the context of non-Newtonian thin-film problems. More precisely, for general nonlinear degenerate parabolic problems of fourth-order we estimate (under some structural condition on the nonlinearity) the third derivative in $L_1$ in time and $L_{\alpha+1}$ in space in terms of the solution's $L_\infty$-norm and the norm of the source term.
In Section \ref{ssec:Fourier_modes} we apply the interior regularity estimates to Equation \eqref{eq:PDE_alpha>1}. After a suitable rescaling of the regularised problem, we obtain this estimate in terms of the solution's $L_\infty$-norm and the Fourier coefficients $h_{\pm 1}$.
It turns out that we need to control these Fourier modes $n = \pm 1$ in order to control the position of the circle's center. To this end, we use an iterative scheme to obtain decay estimates for $h_{\pm 1}$. This allows us at the end of Section \ref{ssec:Fourier_modes} to prove an existence result stating the following. Given initial values that are in $H^1(S^1)$ close to a circle centered at the origin, for any given finite time, there exists a positive weak solution to the regularised problem that stays bounded away from zero (i.e. does not touch the cylinder), conserves its mass and satisfies for almost all times a energy-dissipation identity.

In Section \ref{sec:sigma_to_zero} we consider the limit $\sigma \to 0^+$ of a vanishing regularisation parameter and prove global existence of positive weak solutions to the original problem \eqref{eq:PDE} for initial data that are close to circles centered at the common center of the two cylinders. The conservation of mass property and the energy-dissipation equality (for almost all positive times) are shown to be conserved in the limit of a vanishing regularisation parameter.

Finally, in Section \ref{sec:polynomial_stability} we prove polynomial stability of solutions to the original problem \eqref{eq:PDE} in the sense that solutions that are $\eps$-close to circles centered at the origin in $H^1(S^1)$, remain $\eps$-close and converge to circle at rate $1/t^\frac{1}{\alpha-1}$. Moreover, the position of the circle's center is controlled by $C\eps \bigl(1 + \log(\eps^{1-\alpha})\bigr)$ for some positive constant $C>0$.

\bigskip

\section{Global weak solutions and asymptotic behaviour for $\alpha > 1$ -- The main results.} \label{sec:shear-thinning}

In this work we study the case $\alpha > 1$ in which the thin inner layer is occupied by a shear-thinning power-law fluid. That is, for $\alpha > 1$ we consider the problem
\begin{equation}\label{eq:PDE_alpha>1}
    \begin{cases}
        \partial_t h + \partial_\theta \left(h^{\alpha+2} |\partial_\theta h+\partial_\theta^3 h|^{\alpha-1} \bigl(\partial_\theta h+\partial_\theta^3 h\bigr)\right)
        =
        0,
        &
        t > 0,\, \theta \in S^1
        \\
        h(0,\cdot)
        =
        h_0(\cdot), &
        \theta \in S^1.
    \end{cases}
\end{equation}
Note that since we are solving the problem in the unit circle $S^1$ we implicitly require periodic boundary conditions $\partial_\theta^k h(t,0) = \partial_\theta^k h(t,2\pi)$ for $0 \leq k \leq 3$. 
The partial differential equation in \eqref{eq:PDE_alpha>1} is quasilinear, of fourth order and (doubly) degenerate parabolic in the sense that we loose parabolicity if either $h=0$ or $\partial_\theta h + \partial_\theta^3 h = 0$. Moreover, since we study the case of flow behaviour exponents $\alpha > 1$ we have $(\alpha-1) > 0$ and the coefficients of the highest-order terms depend $(\alpha-1)$-H\"older-continuously on the lower-order terms. 


In this section we prove that, if the solution's shape is initially close to a circle in $H^1(S^1)$, then it stays close to a circle for all times and, moreover, it converges at decay rate $1/t^{\frac{1}{\alpha-1}}$ for large times.

\bigskip

\noindent\textbf{\textsc{Notation, definitions and frequently used facts. }}
We start by introducing the notion of a weak solutions to \eqref{eq:PDE_alpha>1} used throughout the paper.


\begin{definition} \label{def:weak_solution}
Let $T > 0$. By a weak solution to \eqref{eq:PDE_alpha>1} on $[0,T)$ we mean a function 
\begin{equation*}
    h \in
    C\bigl([0,T);H^1(S^1)\bigr)
    \cap 
    L_{\alpha+1}\bigl((0,T);W^3_{\alpha+1}(S^1)\bigr)
    \quad  \text{and} \quad
    \partial_t h \in  L_\frac{\alpha+1}{\alpha}\bigl((0,T);(W^1_{\alpha+1}(S^1))'\bigr),
\end{equation*}
that satisfies the initial and boundary condition $\eqref{eq:PDE_alpha>1}_2$--$\eqref{eq:PDE_alpha>1}_3$ pointwise and the quasilinear parabolic differential equation $\eqref{eq:PDE_alpha>1}_1$ in the weak form
\begin{equation} \label{eq:weak_form}
    \int_0^{T} \langle \partial_t h(t),\phi(t)\rangle_{W^1_{\alpha+1}(S^1)}
    =
    \int_0^{T} \int_{S^1} h^{\alpha+2} \psi\bigl(\partial_\theta h + \partial_\theta^3 h\bigr)\, \partial_\theta \phi\, d\theta\, dt
\end{equation}
for all $\phi \in L_\frac{\alpha+1}{\alpha}\bigl((0,T_\sigma);(W^3_{\alpha+1}(S^1))^\prime\bigr)$.
\end{definition}


Recall that such a solution exists, at least for short times, in view of \cite[Theorem 3.1]{LPCV:2022}. Moreover, such a solution conserves its mass in the sense of the following lemma.


\begin{lemma}[Conservation of mass] \label{lem:cons_mass}
Given $h_0 \in H^1(S^1)$, let 
\begin{equation*}
    h \in L_{\alpha+1}\bigl((0,T);W^3_{\alpha+1}(S^1)\bigr) \cap C\bigl([0,T];H^1(S^1)\bigr)
\end{equation*}
be a weak solution to \eqref{eq:PDE_alpha>1} on $[0,T]$. Then $h$ satisfies
\begin{equation*}
    \|h(t)\|_{L_1(S^1)} = \|h_0\|_{L_1(S^1)}, \quad t \in [0,T].
\end{equation*}
\end{lemma}

\begin{proof}
This follows immediately by testing the equation \eqref{eq:weak_form} with the function $\phi \equiv 1$.
\end{proof}


In virtue of Lemma \ref{lem:cons_mass} it makes sense to introduce the notation
\begin{equation*}
    \bar{h} 
    = 
    \frac{1}{2\pi} \int_{S^1} h(t,\theta)\, d\theta
    =
    \frac{1}{2\pi} \int_{S^1} h_0(\theta)\, d\theta
    =
    \bar{h}_0, 
    \quad t \in [0,T],
\end{equation*}
for the \textbf{spatial average $\bar{h} = \bar{h}_0$ of $h$}, which does not depend on time. Furthermore, we use the notation
\begin{equation*}
    h(t,\theta) 
    =
    \sum_{n\in \Z} h_n(t) e^{in\theta}
    =
    \bar{h}_0 + \sum_{n\in \Z, n \neq 0} h_n(t) e^{in\theta}, 
    \quad t \in (0,T),\, \theta \in S^1,
\end{equation*}
for the corresponding Fourier series with Fourier coefficients
\begin{equation*}
    h_n(t) = \frac{1}{2\pi} \int_{S^1} h(t,\theta)\, e^{-in\theta}\, d\theta,
    \quad t \in [0,T].
\end{equation*}
Finally, we fix $\alpha > 1$ and introduce the notation 
\begin{equation*}
    \psi(s) = |s|^{\alpha-1} s,
    \quad s \in \R,
\end{equation*}
such that the partial differential equation in \eqref{eq:PDE_alpha>1} can be rewritten as
\begin{equation} \label{eq:PDE_psi}
    \partial_t h + \partial_\theta \bigl(h^{\alpha+2} \psi\bigl(\partial_\theta h + \partial_\theta^3 h\bigr)\bigr) = 0, 
    \quad 
    t > 0,\ \theta \in S^1.
\end{equation}

\bigskip

\noindent\textbf{\textsc{Steady states of \eqref{eq:PDE_alpha>1}. }}
Positive steady states of \eqref{eq:PDE_alpha>1} in $S^1$ are solutions to the ordinary differential equation
\begin{equation} \label{eq:flux}
    h^{\alpha+2} |h^\prime + h^{\prime\prime\prime}|^{\alpha-1} \bigl(h^\prime + h^{\prime\prime\prime}\bigr) = C,
    \quad \theta \in S^1,
\end{equation}
with some integration constant $C \in \R$ that may physically be interpreted as the flux of the non-Newtonian fluid through the radius of the outer cylinder. 
Multiplying \eqref{eq:flux} by $(h^\prime + h^{\prime\prime\prime)}$ and integrating over $S^1$ yields
\begin{equation*}
    \int_{S^1} h^{\alpha+2} |h^\prime + h^{\prime\prime\prime}|^{\alpha+1}\, d\theta = 0.
\end{equation*}
If $h > 0$ in $S^1$, we find $(h^\prime + h^{\prime\prime\prime})=0$ and thus, the positive steady states of \eqref{eq:PDE_alpha>1} are given by functions of the form
\begin{equation*}
    h(\theta) = \bar{h}_0 + h_{-1} e^{-i\theta} + h_1 e^{i\theta},
    \quad
    \text{where} \quad
    \bar{h}_0 - 2|h_1| > 0,
    \quad
    \bar{h}_0 \in \R, h_{1}, h_{-1} \in \C
    \quad
    \text{and}
    \quad
    h_1 = \bar{h}_{-1}.
\end{equation*}
In terms of the corresponding real Fourier series this can be rewritten as
\begin{equation} \label{eq:circle_real}
    h(\theta) = \bar{h}_0 + \kappa_{-1} \cos(\theta) + \kappa_1 \sin(\theta),
    \quad
    \text{where}
    \quad
    \kappa_{-1} = \frac{1}{2}(h_{-1} + h_1)
    \quad \text{and} \quad
    \kappa_1 = \frac{i}{2}(h_{-1} - h_1).
\end{equation}


\begin{remark}
Note that solutions of the form \eqref{eq:circle_real} 
represent circles centered at the origin, when lubrication approximation has been applied. In other words, functions of the form \eqref{eq:circle_real} are $\eps^2$-close to circles centered at $\eps(\kappa_{-1},\kappa_1)$ in the $H^1(S^1)$-norm. Moreover, constant functions $h = \bar{h}_0$ correspond to circles centered at the origin. To see this, recall that in dimensionless variables the interface separating the two fluids is described by the curve
\begin{equation*}
    r = 1 + \eps h(t,\theta),
    \quad t > 0,\, \theta \in S^1.
\end{equation*}
We show that the $H^1(S^1)$-distance between the interface $r = 1 + \eps h(t,\theta)$ and the a circle $\tilde{h}_\eps$ with radius $1 + \eps \bar{h}_0$ centered at $\eps(\kappa_{-1},\kappa_1)$ can be estimated by
\begin{equation*}
   \|\tilde{h}_\eps - (1+\eps h)\|_{H^1(S^1)}
   \leq
   C \eps^2.
\end{equation*}
To this end, we consider w.l.o.g. $\kappa_1 = 0$. By the Poincar\'e-inequality (note that $h$ and $\tilde{h}_\eps$ have the same mass) it suffices to prove the estimate
\begin{equation*}
    \|\partial_\theta\tilde{h}_\eps - \partial_\theta (1+\eps h)\|_{L_2(S^1)}
   \leq
   C \eps^2.
\end{equation*}
First, observe that for $h(\theta) = \bar{h}_0 + \kappa_{-1} \cos(\theta),\ \theta \in S^1$, we have
\begin{equation*}
    \partial_\theta (1+\eps h) 
    =
    \eps \partial_\theta h
    =
    -\eps \kappa_{-1} \sin(\theta),
    \quad
    \theta \in S^1.
\end{equation*}
In order to calculate $\partial_\theta \tilde{h}_\eps$, we parametrise the circle with respect to the origin. By the law of cosines we find that
\begin{equation*}
    (1 + \eps \bar{h}_0)^2
    =
    \eps^2 \kappa_{-1}^2 + \tilde{h}_\eps^2 - 2 \eps \kappa_{-1} \cos(\theta),
    \quad \theta \in S^1,
\end{equation*}
and hence, using the positivity of $\tilde{h}_\eps$,
\begin{equation*}
    \tilde{h}_\eps(\theta)
    =
    \eps \kappa_{-1} \cos(\theta) 
    +
    \sqrt{\eps^2 \kappa_{-1}^2 (\cos(\theta))^2 + (1 + \eps \bar{h}_0)^2 - \eps^2 \kappa_{-1}^2},
    \quad
    \theta \in S^1,
\end{equation*}
whence
\begin{equation*}
    \partial_\theta \tilde{h}_\eps(\theta)
    =
    -
    \eps \kappa_{-1} \sin(\theta) 
    -
    \frac{\eps^2 \kappa_{-1}^2 \sin(\theta)\cos(\theta)}{\sqrt{\eps^2 \kappa_{-1}^2 (\cos(\theta))^2 + (1 + \eps \bar{h}_0)^2 - \eps^2 \kappa_{-1}^2}},
    \quad
    \theta \in S^1.
\end{equation*}
Choosing $\eps < 1/2$ small enough and using that $\kappa_{-1} < \bar{h}_0$, we obtain
\begin{equation*}
    \int_{S^1} |\partial_\theta\tilde{h}_\eps - \partial_\theta (1+\eps h)|^2\, d\theta
    \leq
    \int_{S^1} \Biggl|\frac{\eps^2 \kappa_{-1}^2 \sin(\theta)\cos(\theta)}{\sqrt{\eps^2 \kappa_{-1}^2 (\cos(\theta))^2 + (1 + \eps \bar{h}_0)^2 - \eps^2 \kappa_{-1}^2}}\Biggr|^2 d\theta
    \leq
    2\pi \kappa_{-1}^4 \eps^4,
\end{equation*}
by a pointwise estimate of the integrand. This is the desired estimate and we have shown that functions of the form \eqref{eq:circle_real} are circles up to order $\eps^2$ in $H^1(S^1)$.
\end{remark}


The main results of this paper concern the global existence of weak solutions to \eqref{eq:PDE_alpha>1} for initial data with small initial energy and their asymptotic behaviour for large times. It is summarised in the following theorem.


\begin{theorem}\label{thm:main_results}
Let $\alpha > 1$ be fixed. There exists an $\eps_0 > 0$ small such that for all $\eps \in (0,\eps_0)$ and all initial values $h_0\in H^1(S^1)$ with 
\begin{equation*}
    \frac{1}{2\pi}\int_{S^1} h_0\, d\theta = \bar{h}_0 
    \quad
    \text{and}
    \quad
    \|h_0 - \bar{h}_0\|_{H^1(S^1)} \leq \eps,
\end{equation*}
problem \eqref{eq:PDE_alpha>1} possesses a globally in time defined positive weak solution
\begin{equation*}
    h \in C\bigl([0,\infty);H^1(S^1)\bigr)
    \cap
    L_{\alpha+1}\bigl((0,\infty);W^3_{\alpha+1}(S^1)\bigr)
\end{equation*}
in the sense of Definition \ref{def:weak_solution}. Moreover, the solution has the following properties:
\begin{itemize}
    \item[(i)] There exist $\xi_{\pm 1} \in \R$ and hence a function $\bar{h}_\infty(\cdot) =  \bar{h}_0 + \xi_{\pm 1} e^{\pm i\cdot} \in H^1(S^1)$ such that 
    \begin{equation*}
        h(t,\cdot) 
        \longrightarrow 
        \bar{h}_\infty(\cdot)
        \quad 
        \text{in } H^1(S^1) \quad \text{as } t \to \infty.
    \end{equation*}
    Moreover, the numbers $\xi_{\pm 1}\in \R$ are given by $\xi_{\pm 1}=\lim_{t\to \infty} h_{\pm 1}(t)$ and they satisfy the estimate
    \begin{equation*}
        |\xi_{\pm 1}| 
        \leq 
        C_1 \eps \bigl(1 + \log\bigl(\eps^{1-\alpha}\bigr)\bigr),
    \end{equation*}
    for some positive constant $C_1>0$ that depends only on $\alpha > 1$.
    \item[(ii)] The convergence happens in the following way:
    \begin{equation*}
        \|h(t,\cdot) - \bar{h}_\infty(\cdot)\|_{H^1(S^1)}
        \leq
        \frac{\eps}{\bigl(1 + C_2 \eps^{\alpha-1} t\bigr)^\frac{1}{\alpha-1}}, 
        \quad
        t > 0,
    \end{equation*}
    where $C_2 > 0$ is a positive constant that depends only on $\alpha > 1$.
\end{itemize}
\end{theorem}


The strategy of the proof of Theorem \ref{thm:main_results} is the following. We introduce a regularised version of \eqref{eq:PDE_alpha>1} that removes the two degeneracies in the mobility $h^{\alpha+2}$ as well as in the nonlinear term $\psi(\partial_\theta h + \partial_\theta^3 h)$. This allows us to deduce from standard parabolic theory the existence of a locally in time existing positive (strong) solution for positive initial data, cf. Theorem \ref{thm:existence_regularised}.
Moreover, we use Fourier analysis and energy estimates to prove that, as long as the solution is bounded away from zero and bounded above in the sense that $\tfrac{1}{2}\bar{h}_0 \leq h(t,\theta) \leq 2 \bar{h}_0$, solutions corresponding to initial interfaces that do not touch the boundary and are close to a circle (mathematically this means that the initial value $h_0 >0$ is positive and that the initial energy is `small', cf. \eqref{eq:definition_E}), stay close to a circle for small times and converge at rate $t^{-\frac{1}{\alpha-1}}$ to a circle in $H^1(S^1)$, cf. Theorem \ref{thm:power-law_decay}.

\bigskip

\section{The regularised problem -- Local existence of solutions and polynomial stability}\label{sec:energy_estimates}
The main difficulties when studying weak solutions to \eqref{eq:PDE_alpha>1} are caused by the doubly nonlinear and doubly degenerate nature of the equation. In this chapter we introduce a suitable regularised version of \eqref{eq:PDE_alpha>1} and prove local existence of positive weak solutions as well as polynomial $1/t^\frac{1}{\alpha-1}$-stability of steady states. This means that solutions to the regularised problem which do initially not touch the cylinder and which are close to a circle in $H^1(S^1)$, stay close to a circle for small times and converge at rate $1/t^{\frac{1}{\alpha-1}}$ to a circle in $H^1(S^1)$. However, at this point we obtain the corresponding estimates only as long as the solution satisfies $\tfrac{1}{2}\bar{h}_0 \leq h(t,\theta) \leq 2 \bar{h}_0$. 

\bigskip

\subsection{Definition of the regularised problem and main result on local existence and asymptotic behaviour}


In order to handle the difficulties caused by the doubly nonlinear and doubly degenerate structure of the evolution problem \eqref{eq:PDE_alpha>1}, we introduc,e for a fixed regularisation parameter $\sigma \in (0,1)$ and all $s\in\R$, the function 
\begin{equation*}
    \psi_\sigma(s) = (s^2 + \sigma^2)^\frac{\alpha-1}{2} s,
    \quad 
    s \in \R,
\end{equation*}
and substitute the nonlinear term $\psi\bigl(\partial_\theta h + \partial_\theta^3 h\bigr)$ in \eqref{eq:PDE_alpha>1} accordingly. The \textbf{regularised problem} corresponding to \eqref{eq:PDE_alpha>1} thus reads
\begin{equation}\label{eq:PDE_regularised}\tag{$P_\sigma$}
    \begin{cases}
        \partial_t h^\sigma + \partial_\theta\left((h^\sigma)^{\alpha+2} \psi_\sigma\bigl(\partial_\theta h^\sigma + \partial_\theta^3 h^\sigma\bigr)\right)
        =
        0,
        &
        t > 0,\, \theta \in S^1
        \\
        h^\sigma(0,\cdot)
        =
        h_0(\cdot), &
        \theta \in S^1,
    \end{cases}
\end{equation}
with periodic boundary conditions. It follows from standard parabolic theory \cite{A:1993,Eidelman} that the regularised problem possesses (for each fixed $\sigma \in (0,1)$) a local weak solution $h^\sigma$ in the sense of the following theorem. 

\medskip

\begin{theorem}[Local existence for \eqref{eq:PDE_regularised}]\label{thm:existence_regularised}
Let $\alpha > 1$ and $\sigma \in (0,1)$ be fixed. Then the following holds true.
\begin{itemize}
    \item[(i)] For each initial film height         $h^\sigma_0 \in C^\infty(S^1)$ with $h_0^\sigma(\theta) > 0$ for all $\theta \in S^1$, the regularised problem \eqref{eq:PDE_regularised} possesses a unique solution $h^\sigma$ on some time interval $(0,T_\sigma)$ in the sense that $h^\sigma(t,\theta) > 0$ for $t\in [0,T_\sigma), \theta \in S^1,$
    \begin{equation*}
        h^\sigma \in
        C^1\bigl((0,T_\sigma);H^4(S^1)\bigr)
        \cap 
        C\bigl([0,T_\sigma);H^1(S^1)\bigr)
        \cap 
        L_{\alpha+1}\bigl((0,T_\sigma);W^3_{\alpha+1}(S^1)\bigr)
    \end{equation*}
    and $h^\sigma$ satisfies the quasilinear parabolic differential equation $\eqref{eq:PDE_regularised}_1$ in the weak form
    \begin{equation} \label{eq:weak_sol_reg}
        \int_0^{T_\sigma} \langle \partial_t h^\sigma(t),\phi(t)\rangle_{W^1_{\alpha+1}(S^1)}
        =
        \int_0^{T_\sigma} \int_{S^1} (h^\sigma)^{\alpha+2} \psi_\sigma\bigl(\partial_\theta h^\sigma + \partial_\theta^3 h^\sigma\bigr)\, \partial_\theta \phi\, d\theta\, dt
    \end{equation}
    for all $\phi \in L_\frac{\alpha+1}{\alpha}\bigl((0,T_\sigma);(W^1_{\alpha+1}(S^1))^\prime\bigr)$.
    \item[(ii)] If for a given $\tau > 0$ and all $t \in [0,T_\sigma) \cap [0,\tau]$ the solution $h^\sigma$ obtained in (i) satisfies
    \begin{equation*}
        h^\sigma(t,\theta) > 0 \quad \forall t \in [0,T_\sigma) \cap [0,\tau],\, \theta \in S^1
        \quad \text{and} \quad
        \|h^\sigma(t)\|_{L_\infty(([0,T_\sigma) \cap [0,\tau])\times S^1)} < \infty,
    \end{equation*}
    then we can choose $T_\sigma$ such that $T_\sigma > \tau$.
\end{itemize}
\end{theorem}


\begin{remark}
\begin{itemize}
    \item[(i)] Note that we do not assume the time $T_\sigma$ in Theorem \ref{thm:existence_regularised}  to be the maximal time of existence.
    \item[(ii)] In Theorem \ref{thm:existence_regularised} we choose $h^\sigma_0 \in C^\infty(S^1)$ for simplicity. In order to obtain strong solutions by standard analytic semigroup theory for quasilinear parabolic problems (cf.\cite[Theorem 12.1]{A:1993} or \cite[Theorem 4.2]{LM:2020}) it is enough to choose $h_0^\sigma \in W^{3+s}_p(S^1)$ with $s>1/p$. In order to obtain suitable a-priori estimates and to be able to pass to the limit $\sigma \to 0^+$ we only need 
    \begin{equation*}
        \frac{1}{2\pi}\int_{S^1} h^\sigma_0(\theta)\, d\theta = \bar{h}_0^\sigma,
        \quad
        \|h_0^\sigma - \bar{h}_0^\sigma\|_{H^1(S^1)} \leq \eps
        \quad \text{and} \quad
        h_0^\sigma \longrightarrow h_0 \quad \text{in } H^1(S^1) \quad \text{as } \sigma \to 0^+.
    \end{equation*}
\end{itemize}
\end{remark}


Since the main part of the paper is concerned with the regularised problem and we want to avoid overloading the notation, we drop the index $\sigma$ and denote the solution $h^\sigma$ to \eqref{eq:PDE_regularised} by $h=h^\sigma$ as long as no ambiguity is to be feared.


\begin{proof}
\noindent(i) The proof of existence is a simple application of \cite[Theorem 12.1]{A:1993}. Indeed, for $s > 1/2$ we define the open set 
\begin{equation*}
    \Scal = \left\{v \in H^{3+s}(S^1);\, v(\theta) > 0\ \forall \theta \in S^1\right\}
\end{equation*}
and define for $v(t) \in \Scal$ the linear differential operator
\begin{equation*}
    \Acal(v(t)) \in \Lcal\bigl(H^4(S^1);L_2(S^1)\bigr), 
    \quad
    \Acal(v(t)) u
    =
    A(v(t)) \partial_\theta^4 u
\end{equation*}
of fourth order, where
\begin{equation*}
    A(v(t)) = (\alpha-1) v^{\alpha+2} \bigl(|\partial_\theta v + \partial_\theta^3 v|^2 + \sigma^2\bigr)^\frac{\alpha-3}{2} |\partial_\theta v + \partial_\theta^3 v|^2
    +
    v^{\alpha+2} \bigl(|\partial_\theta v + \partial_\theta^3 v|^2 + \sigma^2\bigr)^\frac{\alpha-1}{2}.
\end{equation*}
Moreover, we introduce
\begin{align*}
    \Fcal(v(t)) 
    &=
    (\alpha-1) v^{\alpha+2} \bigl(|\partial_\theta v + \partial_\theta^3 v|^2 + \sigma^2\bigr)^\frac{\alpha-3}{2} |\partial_\theta v + \partial_\theta^3 v|^2 \partial_\theta^2 v
    +
    v^{\alpha+2} \bigl(|\partial_\theta v + \partial_\theta^3 v|^2 + \sigma^2\bigr)^\frac{\alpha-1}{2} \partial_\theta^2 v
    \\
    &\quad
    +
    (\alpha+2) v^{\alpha+1} \partial_\theta v\, \psi_\sigma\bigl(|\partial_\theta v + \partial_\theta^3 v|\bigr)
\end{align*}
and perceive \eqref{eq:PDE_regularised} as an abstract quasilinear Cauchy problem
\begin{equation*}
    \begin{cases}
        \dot{u} + \Acal(u)u = \Fcal(u), \quad t > 0
        \\
        u(0) = h_0.
    \end{cases}
\end{equation*}
Note that the maps
\begin{equation*}
    \Acal\colon H^{3+s}(S^1) \longrightarrow
    \Lcal\bigl(H^4(S^1);L_2(S^1)\bigr)
    \quad \text{and} \quad
    \Fcal\colon H^{3+s}(S^1) \longrightarrow L_2(S^1)
\end{equation*}
are locally Lipschitz continuous. In addition, the problem \eqref{eq:PDE_regularised} is parabolic in the sense that $-\Acal(v(t))$ generates an analytic semigroup on $L_2(S^1)$. Indeed, due to the embedding $H^{3+s}(S^1) \hookrightarrow C(S^1)$ and the fact that $\sigma > 0$, we have that $A(v(t,\cdot)) \in C(S^1)$.
Moreover, the principal symbol $a_{\pi}(\theta,\xi)$ satisfies
\begin{equation*}
    \text{Re}(a_\pi(\theta,\xi)\eta | \eta)
    =
    A(v(t,\theta))(i\xi)^4 \eta^2 > 0,
    \quad
    (\theta,\xi) \in S^1\times \{-1,1\},\, \eta \in \R\setminus\{0\}.
\end{equation*}
Consequently, $A(v(t))$, together with the periodic boundary conditions, is normally elliptic in the sense of \cite[Example 4.3(d)]{A:1993} and we can apply \cite[Theorem  4.1 and Remark 4.2(b)]{A:1993} to conclude that $-\Acal(v(t))$ generates an analytic semigroup on $L_2(S^1)$. Thus, we are in the abstract setting of \cite[
Theorem 12.1]{A:1993} which yields the existence of a maximal solution that satisfies the partial differential equation pointwise in $L_2(S^1)$. Clearly this solution is also a weak solution in the sense of equation \eqref{eq:weak_sol_reg}.

\noindent(ii) We prove the contraposition and assume  that $T_\sigma \leq \tau$. Then we have
\begin{equation*}
    \liminf_{t\to T_\sigma^-} \min_{\theta \in S^1} h^\sigma(t,\theta) = 0
    \quad \text{or} \quad
    \limsup_{t\to T_\sigma^-}\, \max_{\theta\in S^1} h^\sigma(t,\theta) = \infty,
\end{equation*}
which contradicts the assumption on $\tau$.
\end{proof}


One important property of solutions to the regularised problem \eqref{eq:PDE_regularised} is that they conserve their mass in the following sense.


\begin{lemma}[Conservation of mass for \eqref{eq:PDE_regularised}] \label{lem:cons_mass_reg}
Let $h$ be the solution to the regularised problem \eqref{eq:PDE_regularised} on $[0,T_\sigma)$, corresponding to the initial value $h_0$. Then $h$ conserves its mass in the $L_1(S^1)$-sense, i.e. we have
\begin{equation*}
    \|h(t)\|_{L_1(S^1)} = \|h_0\|_{L_1(S^1)},
    \quad
    t \in [0,T_\sigma).
\end{equation*}
\end{lemma}


\begin{proof}
This follows from testing the partial differential equation $\eqref{eq:PDE_regularised}_1$ with the constant function $\phi=1$.
\end{proof}


According to the mass conservation property of solutions to the regularised problem \ref{lem:cons_mass_reg} we also use the notation
\begin{equation*}
    \bar{h} 
    = 
    \frac{1}{2\pi} \int_{S^1} h(t,\theta)\, d\theta
    =
    \frac{1}{2\pi} \int_{S^1} h_0(\theta)\, d\theta
    =
    \bar{h}_0, 
    \quad t \in [0,T),
\end{equation*}
for the \textbf{spatial average $\bar{h} = \bar{h}_0$ of $h$}, which does not depend on time. Note that in this paper we consider only the case $\bar{h}_0 > 0$.
Furthermore, we use the notation
\begin{equation*}
    h(t,\theta) 
    =
    \sum_{n\in \Z} h_n(t) e^{in\theta}
    =
    \bar{h}_0 + \sum_{n\in \Z, n \neq 0} h_n(t) e^{in\theta}, 
    \quad t \in [0,T),\, \theta \in S^1,
\end{equation*}
for the corresponding Fourier series with Fourier coefficients
\begin{equation*}
    h_n(t) = \frac{1}{2\pi} \int_{S^1} h(t,\theta)\, e^{-in\theta}\, d\theta,
    \quad t \in [0,T].
\end{equation*}


In the sequel we study the asymptotic behaviour of solutions to \eqref{eq:PDE_regularised} for large times. 
For this purpose, we introduce the time $t_\ast$ of existence of weak solutions $h$ to \eqref{eq:PDE_regularised} that is bounded away from zero and bounded above in terms of the average. This number $t_\ast$ will be used frequently throughout the paper.


\begin{definition}
We define
\begin{equation*}
    t_\ast = \sup\left\{T > 0; \text{ $\exists$ solution $h$ to \eqref{eq:PDE_regularised} in the sense of Thm. \ref{thm:existence_regularised} s.t. $\tfrac{1}{2}\bar{h}_0 \leq h(t,\theta) \leq 2\bar{h}_0$ on $[0,T]\times S^1$}\right\}.
\end{equation*}
\end{definition}

Observe that, as a consequence of Theorem \ref{thm:existence_regularised} (ii), we have that $T_\sigma > t_\ast$.

\bigskip

\subsection{The energy functional and convergence to a circle} \label{sec:conv_to_circle}
This subsection is dedicated to the proof of part (ii) of Theorem \ref{thm:main_results}. We prove that if the shape of the interface is initially close to a circle and does not touch the cylinder, it remains close to a circle for small times. Moreover, for large times, the solution converges with a polynomial decay to a circle. However, we do not yet have any information on the circle's center. In particular, the center of the limit circle does not necessarily coincide with the common center of the cylinders. 
The proof relies solely on Fourier analysis and suitable estimates for a certain energy functional. 

\begin{theorem}\label{thm:power-law_decay}
Let $\alpha > 1$ and $\sigma \in (0,1)$ be fixed. Let $h$ be the local solution to \eqref{eq:PDE_regularised} on $[0,T]$, corresponding to an initial value $h_0\in H^1(S^1)$ with $h_0(\theta) > 0$ for all $\theta \in S^1$. Then there exists an $\eps_0 > 0$ such that for all $\eps \in (0,\eps_0)$ the additional conditions
\begin{equation*}
    \frac{1}{2\pi}\int_{S^1} h_0\, d\theta = \bar{h}_0 > 0
    \quad \text{and} \quad
    \|h_0 - \bar{h}_0\|_{H^1(S^1)} \leq \eps
\end{equation*}
imply that the solution satisfies
\begin{equation}\label{eq:decay_est}
    \|h(t,\theta) - (\bar{h}_0 + h_{-1}(t) e^{-i\theta} + h_1(t) e^{i\theta})\|_{C([\bar{t}/4,\bar{t}];H^1(S^1))} 
    \leq 
    C \Lambda_\eps(\bar{t}),
    \quad \bar{t} \in [0,t_\ast],
\end{equation}
for some positive constant $C > 0$,
where the function $\Lambda_\eps$ is defined by
\begin{equation}\label{eq:def_Lambda}
    \Lambda_\eps(t)
    =
    \frac{\eps}{\left(1 + C \eps^{\alpha-1} t\right)^\frac{1}{\alpha-1}},
    \quad
    t \geq 0.
\end{equation}
for suitable times $t$.
\end{theorem}


We now introduce the energy functional that provides the crucial estimates for the proof of Theorem \ref{thm:power-law_decay}. 


\begin{definition}
Let $v \in H^1(S^1)$ such that $v(\theta) \geq 0$ for all $\theta \in S^1$.
\begin{itemize}
    \item[(i)] We define the \textbf{energy functional}
        \begin{equation} \label{eq:definition_Energy}
            E[v]
            =
            \pi \bar{v}^2 +
            \frac{1}{2} \int_{S^1} \bigl(|\partial_\theta v|^2 - v^2\bigr)\, d\theta.
        \end{equation}
    \item[(ii)] Moreover, the so-called \textbf{modified energy functional} is given by
        \begin{equation} \label{eq:definition_E}
            e[v]
            =
            \frac{1}{2} \int_{S^1} \bigl(|\partial_\theta v|^2 - v^2\bigr)\, d\theta.
        \end{equation}
    \item[(iii)] If in addition $v \in W^3_{\alpha+1}(S^1)$, then we can in addition define the \textbf{dissipation functional} 
        \begin{equation*}
            J[v] = \int_{S^1} v^{\alpha+2} \left(|\partial_\theta v + \partial_\theta^3 v|^2 + \sigma^2\right)^\frac{\alpha-1}{2} |\partial_\theta v + \partial_\theta^3 v|^2\, d\theta.
        \end{equation*}
\end{itemize}
\end{definition}


As the following two lemmas show, the functional $E$ is an energy functional in the sense that it is a non-negative quantity which decreases along weak solutions to \eqref{eq:PDE_regularised}.


\begin{lemma}
The functional $E$ satisfies $E[v] \geq 0$ for all $v \in H^1(S^1)$.
\end{lemma}


\begin{proof}
The statement follows from the Fourier series representation of $v$ and Plancherel's theorem. Indeed, we find that
\begin{align*}
	E[v](t)
	&=
	\pi \bar{v}^2 + \frac{1}{2}\int_{S^1} \left(|\partial_\theta v|^2 - |v|^2\right)\, d\theta
	\\
	&=
	\pi \bar{v}^2 +
	\frac{1}{2} \|\partial_\theta v\|_{L_2(S^1)}^2 - 
	\frac{1}{2} \|v\|_{L_2(S^1)}^2
	\\
	&=
	\pi \bar{v}^2 +
	\pi 
	\left(\sum_{n\in \Z, n \neq 0} (n^2-1) |v_n|^2
	-
	|\bar{v}|^2\right)
	\\
	&=
	\pi 
	\sum_{n\in\Z, n \neq 0} (n^2-1) |v_n|^2 \geq 0.
\end{align*}
\end{proof}


In addition to its non-negativity property the functional $E$ dissipates energy along solutions to \eqref{eq:PDE_regularised} in the following sense.


\begin{lemma}[Energy-dissipation inequality] \label{lem:energy_dissipation}
Given $h_0 \in H^1(S^1)$, let $\sigma \in (0,1)$ be fixed and let 
\begin{equation*}
    h \in L_{\alpha+1}\bigl((0,T);W^3_{\alpha+1}(S^1)\bigr) \cap C\bigl([0,T);H^1(S^1)\bigr)
\end{equation*}
be a corresponding weak solution to \eqref{eq:PDE_regularised} on $[0,T]$. Then the functional $E$ defined in \eqref{eq:definition_E} satisfies the \textbf{energy (dissipation) equality}
\begin{equation*}
    E[h](t) + 
    D^\sigma_t[h]
    = 
    E[h_0], \quad t \in [0,T].
\end{equation*}
where $D_t^\sigma[h]$ is a non-negative quantity given by
\begin{equation*}
    D^\sigma_t[h]
    =
    \int_0^t \int_{S^1} |h|^{\alpha+2} \left(|\partial_\theta h + \partial_\theta^3 h|^2 + \sigma^2\right)^\frac{\alpha-1}{2}
    |\partial_\theta h + \partial_\theta^3 h|^2
    \, d\theta\, ds.
\end{equation*}
\end{lemma}


Note that $D^\sigma_t[h]$ is the time integral of the dissipation functional $J$. Thus, in the following we often refer to $D^\sigma_t[h]$ as the \textbf{dissipation}.


\begin{proof}
According to the regularity properties of weak solutions obtained in Theorem \ref{thm:existence_regularised}, we may use $(h +\partial_\theta^2 h)$ as a test function for \eqref{eq:PDE_regularised}. This proves the lemma.
\end{proof}


\begin{remark}
We briefly comment on the definition of the energy functional $E$ in \eqref{eq:definition_Energy}. Let us assume that we are given a smooth solution $h$ of \eqref{eq:PDE_regularised}. A natural test function for \eqref{eq:PDE_regularised} in order to observe energy dissipation is $(h +\partial_\theta^2 h)$. Testing \eqref{eq:PDE_regularised} with $(h+\partial_\theta^2 h)$ corresponds to calculating the time derivative of the functional
\begin{equation*}
    e[h] = \frac{1}{2} \int_{S^1} \bigl(|\partial_\theta h|^2 - |h|^2\bigr)\, d\theta.
\end{equation*}
However, this functional is not necessarily a non-negative quantity. In particular, if the solution has a circular shape, i.e. if it is given by
\begin{equation*}
    v(t,\theta) = \bar{v} + v_{-1}(t) e^{-i\theta} + v_1(t) e^{i\theta}, \quad t \in [0,T],\, \theta \in S^1,
\end{equation*}
then the corresponding energy is given by $e[v] = -\pi\bar{v}^2$, i.e. the only remaining term is the one coming from the constant zeroth Fourier mode. This motivates us to define the functional
\begin{equation*} 
    E[v](t) 
    = 
    \pi \bar{v}^2 + e[v](t)
    =
    \pi \bar{v}^2 +
    \frac{1}{2} \int_{S^1} \bigl(|\partial_\theta v|^2 - |v|^2\bigr)\, d\theta, \quad t \in [0,T].
\end{equation*}
\end{remark}


In addition to the fact that weak solutions to \eqref{eq:PDE_regularised} dissipate energy, we observe that also the dissipation functional decreases in time if it is small enough initially. To see this, we start by proving several auxiliary results on the dynamic behaviour of $J$.
First, clearly Lemma \ref{lem:energy_dissipation} may be reformulated as
\begin{equation} \label{eq:E_prime=J}
    \frac{d}{dt}E[h](t) =  -J[h](t), \quad t \in [0,T].
\end{equation}


We are now in a position to prove that, if the energy $E[v]$ of a certain function $v$ with positive constant average $\bar{v}>0$ is small, then this function is close to a circle with respect to the $H^1(S^1)$-norm. 


\begin{proposition}\label{lem:H^1-bound_phi}
Let $\delta > 0$ be a small given number. Then, for every function $v \in H^1(S^1)$ satisfying
\begin{equation*}
    \frac{1}{2 \pi}\int_{S^1} v\, d\theta = \bar{v} > 0
    \quad
    \text{and}
    \quad
    E[v] \leq \frac{\delta}{4}
\end{equation*}
it follows that
\begin{equation*}
    \left\|v(\theta) - \bigl(\bar{v} + v_{-1} e^{-i\theta} + v_1 e^{i\theta}\bigr)\right\|_{H^1(S^1)} 
    \leq 
    \sqrt{\delta}.
\end{equation*}
\end{proposition}


\begin{proof}
Assume that $E[v] \leq \delta/4$. In terms of the energy's Fourier series, this reads
\begin{equation*}
    E[v]
    =
    \pi \sum_{n\in\Z} (n^2 - 1) |v_n|^2 + \pi \bar{v}^2
    =
    \pi \sum_{n\in\Z, n\neq 0,\pm 1} (n^2 - 1) |v_n|^2 \leq \frac{\delta}{4},
\end{equation*}
Now, for $\theta \in S^1$, we define the function
\begin{equation*}
    \phi(\theta) = v(\theta) - \bigl(\bar{v} + v_{-1} e^{-i\theta} + v_1 e^{i\theta}\bigr) 
    = 
    \sum_{n\in\Z} \phi_n e^{in\theta}
\end{equation*}
as the pointwise difference of the function $v$ to a circle.
Then, we clearly have
\begin{equation*}
    \phi_0 = \phi_{\pm 1} = 0 \quad \text{and} \quad \phi_n = v_n \quad \text{for}\quad 
    n \neq 0, \pm 1.
\end{equation*}
Moreover, using Plancherel's theorem, we obtain
\begin{equation*}
    \begin{split}
        \left\|\phi\right\|_{H^1(S^1)}^2
        &=
        2\pi \sum_{n\in\Z, n \neq 0,\pm 1} (n^2+1) |v_n|^2
        \\
        &=
        2\pi \sum_{n\in\Z, n \neq 0,\pm 1} \bigl[(n^2-1)+2\bigr] |v_n|^2
        \\
        &\leq
        4\pi \sum_{n\in\Z, n \neq 0,\pm 1} (n^2-1) |v_n|^2
        \\
        &=
        4E[v] \leq \delta.
    \end{split}
\end{equation*}
This is the desired estimate.
\end{proof}


Note that if $J[h](t)=0$, then also $\frac{d}{dt}E[h](t)=-J[h](t)=0$, i.e. the system does not dissipate energy and the solution is already contained in the manifold of steady states.

In the next lemma we derive an estimate for the energy $E$ in terms of the dissipation $J$. Note that this in turn gives a sense to the smallness assumptions in the previous lemma.


\begin{lemma} \label{lem:E-E_min_J}
There exists an $\eps_0 > 0$ such that for all $\eps \in (0,\eps_0)$ and all initial values $h_0 \in H^1(S^1)$ with
\begin{equation*}
    \frac{1}{2\pi} \int_{S^1} h_0\,d\theta = \bar{h}_0 > 0
    \quad \text{and} \quad
    \|h_0 - \bar{h}_0\|_{H^1(S^1)} \leq \eps
\end{equation*}
the corresponding local solution $h$ to \eqref{eq:PDE_regularised} satisfies 
\begin{equation*}
    E[h](t) \leq C \bigl(J[h](t)\bigr)^\frac{2}{\alpha+1},
    \quad
    t \in [0,t_\ast].
\end{equation*}
\end{lemma}


\begin{proof}
We apply again Fourier analysis and write the energy $E[h](t)$ for $t \in [0,t_\ast]$ as
\begin{equation*}
    E[h](t)
    =
    \pi \sum_{n\in\Z} (n^2 - 1) |h_n|^2 + \pi \bar{h}_0^2
    =
    \pi \sum_{n\in\Z, n\neq 0,\pm 1} (n^2 - 1) |h_n|^2
    =
    \pi \sum_{n\in\Z, n\neq 0,\pm 1} (n^2 - 1) |(h - \bar{h})_n|^2,
\end{equation*}
where we use that $h_n = (h - \bar{h})_n,\ n \neq 0$. Setting $w = h + \partial_\theta^2 h$, the Fourier coefficients of $(w-\bar{w})$ are given by $(w - \bar{w})_n = (1 - n^2) (h-\bar{h})_n$ for $n\in\Z$. Thus, we find that
\begin{equation*}
    E[h](t)
    =
    \pi \sum_{n\in\Z, n\neq 0,\pm 1} \frac{1}{n^2 - 1} |(w - \bar{w})_n|^2,
    \quad t \in [0,t_\ast].
\end{equation*}
Using that $\frac{1}{n^2-1} < 1$ for all $n \neq \pm 1$ and invoking Plancherel's theorem, we obtain
\begin{equation*}
    E[h](t)
    \leq
    \pi \sum_{n\in\Z, n\neq 0,\pm 1} |(w - \bar{w})_n|^2
    \leq
    \frac{1}{2} \left\|w - \bar{w}\right\|_{L_2(S^1)}^2.
\end{equation*}
Finally, Jensen's inequality (for concave functions) and the assumption $\frac{\bar{h}}{2} \leq h(t,\theta),\, t \in [0,t_\ast],\, \theta \in S^1$, imply that
\begin{equation*}
        E[h](t)
        \leq
        \frac{1}{2} \left\|w - \bar{w}\right\|_{L_2(S^1)}^2
        \leq
        C \left(\int_{S^1} |\partial_\theta (w - \bar{w})|^{\alpha+1}\right)^\frac{2}{\alpha+1}
        \leq
        C \bigl(J[h](t)\bigr)^\frac{2}{\alpha+1},
        \quad t \in [0,t_\ast].
\end{equation*}
This is the desired inequality.
\end{proof}


Using Lemma \ref{lem:energy_dissipation} and the previous Lemma \ref{lem:E-E_min_J} we are able to prove the following decay estimate for the energy functional $E$. Recall that $E[h](0) \leq \|h_0 - \bar{h}_0\|_{H^1(S^1)}^2$.


\begin{corollary} \label{cor:energy_decay}
There exists an $\eps_0 > 0$ such that for all $\eps \in (0,\eps_0)$ and all initial values $h_0 \in H^1(S^1)$ with
\begin{equation*}
    \frac{1}{2\pi} \int_{S^1} h_0\,d\theta = \bar{h}_0 > 0
    \quad \text{and} \quad
    \|h_0 - \bar{h}_0\|_{H^1(S^1)} \leq \eps
\end{equation*}
the corresponding local solution $h$ to \eqref{eq:PDE_regularised} satisfies 
\begin{equation*}
    E[h](t) 
    \leq 
    (\Lambda_\eps(t))^2,
    \quad
    0\leq t \leq t_\ast, 
\end{equation*}
with some positive constant $C > 0$ in the definition of $\Lambda_\eps$, cf. \eqref{eq:def_Lambda}.
\end{corollary}


\begin{proof}
In view of Lemma \ref{lem:E-E_min_J} we have the estimate 
\begin{equation*}
    E[h](t) \leq C \bigl(J[h](t)\bigr)^\frac{2}{\alpha+1}, \quad 0 \leq t \leq t_\ast.
\end{equation*}
This implies in particular that
\begin{equation*}
    \frac{d}{dt}E[h](t) 
    =
    - J[h](t) 
    \leq 
    - C \bigl(E[h](t)\bigr)^\frac{\alpha+1}{2}, 
    \quad 0\leq t \leq t_\ast.
\end{equation*}
and consequently,
\begin{equation*}
    \frac{d}{dt}E[h](t) \bigl(E[h](t)\bigr)^{-\frac{\alpha+1}{2}}
    = \frac{2}{1-\alpha} \frac{d}{dt}\bigl(E[h](t)\bigr)^\frac{1-\alpha}{2} \leq -C,
    \quad 0\leq t \leq t_\ast.
\end{equation*}
Integration with respect to time yields (recall that $\alpha > 1$)
\begin{equation*}
    \frac{2}{1-\alpha} \left[\bigl(E[h](t)\bigr)^{-\frac{(\alpha-1)}{2}} - E[h_0]^{-\frac{(\alpha-1)}{2}}\right] \leq - Ct, \quad
    0\leq t \leq t_\ast,
\end{equation*}
and thus,
\begin{equation*}
    \bigl(E[h](t)\bigr)^\frac{-(\alpha-1)}{2} \geq
    E[h_0]^{-\frac{(\alpha-1)}{2}} + \frac{C (\alpha-1)}{2} t, 
    \quad 
    0\leq t \leq t_\ast.
\end{equation*}
We eventually arrive at the estimate 
\begin{equation*}
    E[h](t)
    \leq
    \left(E[h_0]^{-\frac{(\alpha-1)}{2}} + \frac{C (\alpha-1)}{2} t\right)^{-\frac{2}{\alpha-1}}
    \leq
    \left(\eps^{1-\alpha} + \frac{C (\alpha-1)}{2} t\right)^{-\frac{2}{\alpha-1}}
    =
    \eps \left(1 + \frac{C (\alpha-1)}{2} \eps^{\alpha-1}  t \right)^{-\frac{2}{\alpha-1}}
\end{equation*}
for all $0\leq t \leq t_\ast$.
\end{proof}
\medskip


\begin{proof}[\textbf{Proof of Theorem \ref{thm:power-law_decay}}]
The proof of Theorem \ref{thm:power-law_decay} is now a simple consequence of the energy decay estimate in Corollary \ref{cor:energy_decay}. Indeed, as above we write
\begin{equation*}
    \phi(t,\theta) = h(t,\theta) - \bigl(\bar{h}_0 + h_{- 1}(t) e^{-i \theta} + h_1(t) e^{i\theta}\bigr),
    \quad
    t \in [0,t_\ast], \theta \in S^1,
\end{equation*}
for the solution $h$ to \eqref{eq:PDE_regularised}, corresponding to an initial value $h_0 \in H^1(S^1)$ with
\begin{equation*}
    \frac{1}{2\pi} \int_{S^1} h_0\,d\theta = \bar{h}_0 > 0
    \quad \text{and} \quad
    \|h_0 - \bar{h}_0\|_{H^1(S^1)} \leq \eps.
\end{equation*}
In the proof of Lemma \ref{lem:H^1-bound_phi} we have seen that $\|\phi(t)\|_{H^1(S^1)} \leq 4 E[h](t)$ for $t \in [0,t_\ast]$. Thus, Corollary \ref{cor:energy_decay} yields
\begin{equation*}
    \|\phi(t)\|_{H^1(S^1)}^2
    \leq 
    4
    E[h](t) 
    \leq 
    C
    (\Lambda_\eps(t))^2,
    \quad
    t \in [0,t_\ast],
\end{equation*}
with a constant $C > 0$ that does not depend on $\sigma$. 
\end{proof}


It is worthwhile to briefly comment on the estimate obtained in Theorem \ref{thm:power-law_decay}. The estimate implies in particular that, for small times, the function
\begin{equation*}
    \phi(t,\theta) = h(t,\theta) - \bigl(\bar{h}_0 + h_{- 1}(t) e^{-i \theta} + h_1(t) e^{i\theta}\bigr),
    \quad
    t \in [0,t_\ast], \theta \in S^1,
\end{equation*}
satisfies
\begin{equation*}
    \begin{cases}
        \|\phi(t)\|_{C([0,1];H^1(S^1))} \leq \sqrt{E[h_0]}
        \leq C \eps
        &
        \\
        \|\phi(t)\|_{C([\bar{t}/4,\bar{t}];H^1(S^1))} \leq C
        \sqrt{E[h](\bar{t}/4)} \leq C \eps, 
        \quad 1 \leq \bar{t} \leq \eps^{1-\alpha},
    \end{cases}
\end{equation*}
and, for large times $\bar{t} \geq \eps^{1-\alpha}$,
\begin{equation*}
    \|\phi(t)\|_{C([\bar{t}/4,\bar{t}];H^1(S^1))} \leq C \sqrt{E[h](\bar{t}/4)} \leq C (\bar{t})^{-\frac{1}{\alpha-1}}, \quad \bar{t} \geq \eps^{1-\alpha}.
\end{equation*}
That is, if the shape of the interface is initially close to a circle w.r.t. the $H^1(S^1)$-norm, then it remains close to a circle for small times. Moreover, for large times, the solution converges with a power-law decay to the manifold of steady states. 

However, so far we do not have any control on the position of the circle's center. In order to obtain more information, we have to prove that the Fourier modes $h_{\pm 1}(t)$ are bounded in the sense that $|h_{\pm 1}(t)| \leq C t^{-\frac{1}{\alpha-1}}$ for large times, since this implies that the center of the circle does not move too much.

\bigskip

\section{Regularity estimates for quasilinear problems of fourth order} \label{sec:regularity_estimates}
In this section we derive suitable regularity estimates for general nonlinear fourth-order degenerate parabolic equations.
More precisely, we consider problems of the form
\begin{equation} \label{eq:general_parabolic}
    \partial_\tau v + \partial_x\bigl( \Phi(v,\partial_x v,\partial_x^2 v, \partial_x^3 v)\bigr) + S(\tau,x) = 0,
    \quad
    \tau \in (1/4,1),\ x\in [a,b],
\end{equation}
where $[a,b] \subset \R$ is some compact interval, $S\in L_1\bigl((1/4,1);H^2([a,b])\bigr)$ and the function $\Phi$ satisfies the inequalities
\begin{equation} \label{eq:conditions_Phi}
    \begin{cases}
        |\Phi(v,\partial_x v,\partial_x^2 v, \partial_x^3 v)| \leq
        \sum_{k=0}^3 C_k |\partial_x^k v|^\alpha + C, 
        \\
        \Phi(v,\partial_x v,\partial_x^2 v, \partial_x^3 v) \partial_x^3 v \geq C_3 |\partial_x^3 v|^{\alpha+1} - \sum_{k=0}^2 C_k |\Phi|\cdot |\partial_x^k v|
    \end{cases}
\end{equation}
for positive constants $C, C_k > 0,\, k \in \{0,1,2,3\}$.


\begin{definition}\label{def:sol_general_parabolic}
By a weak solution to \eqref{eq:general_parabolic} we mean a function
\begin{equation*}
    v \in L_{\alpha+1}\bigl([1/4,1];W^3_{\alpha+1,\text{loc}}([a,b])\bigr)
    \quad \text{with} \quad
    \partial_\tau v \in L_\frac{\alpha+1}{\alpha}\bigl([1/4,1];(W^1_{\alpha+1,\text{loc}}([a,b]))'\bigr)
\end{equation*}
that satisfies the weak formulation
\begin{equation*}
    \int_{1/4}^{1} \langle \partial_\tau v,\xi\rangle_{W^1_{\alpha+1}([a,b])}\, d\tau
    =
    \int_{1/4}^{1} \int_a^b \Phi(v,\partial_x v,\partial_x^2 v, \partial_x^3 v) \xi\, dx\, d\tau
    +
    \int_{1/4}^{1} \int_a^b S\, \xi\, dx\, d\tau
\end{equation*}
for all test functions $\xi \in L_{\alpha+1}\bigl([1/4,1];W^1_{\alpha+1}([a,b])\bigr)$ with $\text{supp}\, \xi \subset [1/4,1]\times K$ for some compact subinterval $K \subset (a,b)$.
\end{definition}

\bigskip

\subsection{Weighted interpolation inequalities.}
In this paragraph we derive a weighted interpolation inequality for $L_p(\R)$-spaces. More precisely, we prove local estimates for the $L_p(\R)$-norm of functions $v \in L_\infty(\R) \cap W^3_{p,\text{loc}}(\R)$ and its derivatives up to order two in terms of the $L_p(\R)$-norm of its third derivative and the $L_\infty(\R)$-norm of the function itself.
\medskip


\begin{proposition} \label{prop:weighted_Sobolev}
Let $1 \leq p,\ m > n > 0,\ j \in \{0,1,2\}$ and let $n < m-jp+1$ as well as $n \leq (3-j)p$. Then, for all $\delta \in (0,1)$ there exists a constant $C_\delta = C \delta^{-\frac{j}{3-j}} > 0$, $C >0$ depending only on $\alpha$, such that the weighted interpolation inequality 
\begin{equation} \label{eq:weigthed_int_est}
    \int_{\R} \zeta^{m-n} |\partial_x^j v|^{p}\, dx
    \leq 
    \delta \int_{\R} \zeta^m |\partial_x^3 v|^p\, dx
    + 
    C_\delta \|v\|_{L_\infty(\R)}^{p}
\end{equation}
is satisfied for all $v \in L_\infty(\R) \cap W^3_{p,\text{loc}}(\R)$, where $\zeta \in C^2\bigl((x_0-1/4,x_0+1/4)\bigr)\cap C(\R)$ is a cut-off function such that 
\begin{equation*}
    \zeta(x) =
    \begin{cases}
        8(x-x_0) + 2, & x \in [x_0-1/4,x_0 - 3/16]
        \\
        2 - 8(x-x_0), & x \in [x_0+3/16,x_0+1/4]
        \\
        1, & x \in [x_0-1/8,x_0+1/8]
        \\
        0, & x \in \R\setminus [x_0-1/4,x_0+1/4]
    \end{cases}
\end{equation*}
and $1/2 \leq \zeta(x) \leq 1$ for $x \in [x_0-3/16,x_0-1/8]$ and $[x_0+1/8,x_0+3/16]$.
Moreover, if $1\leq q < p$, then for all $\delta \in (0,1)$ there exists $\tilde{C}_\delta = C \delta^{-\frac{qj}{3p-qj}} > 0$ such that
\begin{equation} \label{eq:weigthed_int_est_q}
    \int_\R \zeta^{m-n} |\partial_x^jv|^q\, dx
    \leq
    \delta \int_\R \zeta^{m} |\partial_x^3 v|^p\, dx
    +
    \tilde{C}_\delta \|v\|_{L_\infty(\R)}^\frac{qp(3-j)}{3p-qj}.
\end{equation}
\end{proposition}


\begin{proof}
Due to the invariance of the result under translations we can assume without loss of generality that $x_0=1/4$ in order to simplify notation.
We consider the different subintervals of definition of $\zeta$.

\noindent(i) On $\R\setminus [0,1/2]$ we have $\zeta=0$ and there is nothing to show.

\noindent(ii) On $[1/8,3/8]$ we have $\zeta=1$. In this case the inequality \eqref{eq:weigthed_int_est} may be derived as 
\begin{equation*}
        \int_{1/8}^{3/8} |\partial_x^j v|^p\, dx 
        \leq
        \left(\int_{1/8}^{3/8} |\partial_x^3 v|^p\, dx\right)^{j/3} \left(\int_{1/8}^{3/8} |v|^p\, dx\right)^\frac{3-j}{3} 
        \leq
        \delta \int_{1/8}^{3/8} |\partial_x^3 v|^p\, dx + C_\delta \|v\|_{L_\infty(\R)}^p,
\end{equation*}
where $C_\delta = C \delta^{-\frac{j}{3-j}}$.
Here we used the Gagliardo--Nirenberg inequality, Young's inequality and the fact that $L_\infty([1/8,3/8]) \hookrightarrow L_p([1/8,3/8])$ for $0 < p < \infty$.

\noindent(iii) On $(0,1/16)$ we have $\zeta(x) = 8x \in (0,1/2)$. In this case we decompose the interval $[0,1/16]$ as follows. For $k \geq 5$ let $I_k = [2^{-k},2^{-k+1})$. Then the length of each subinterval $I_k$ is given by $|I_k| = 2^{-k}$. Moreover, we have that
\begin{equation*}
    8\cdot 2^{-k} \leq \zeta(x) \leq 8\cdot 2^{-k+1}, \quad x \in I_k,\ k \geq 5.
\end{equation*}
Thus, with $v_k(x) = v(2^k x)$ we obtain
\begin{equation} \label{eq:alternative_1}
    \begin{split}
        \int_{I_k} \zeta^{m-n} |\partial_x^j v|^p\, dx
        &\leq
        (2^{-k+1})^{m-n} \int_{I_k} |\partial_x^j v|^p\, dx
        \leq
        C_{m,n}(2^{-k})^{m-n} |I_k|^{-jp+1} \int_{0}^{1/16} |\partial_x^j v_k|^p\, dx.
    \end{split}
\end{equation}
As in part (ii) we may now conclude that for all $\delta > 0$ there exists a constant $C_\delta = C \delta^{-\frac{j}{3-j}} > 0$ such that
\begin{equation} \label{eq:alternative_2}
        \int_{0}^{1/16} |\partial_x^j v_k|^p\, dx
        \leq
        \delta \int_{0}^{1/16} |\partial_x^3 v_k|^p\, dx + C_\delta \|v\|_{L_\infty(\R)}^p
        \leq
        \delta |I_k|^{3p-1} \int_{I_k} |\partial_x^3 v|^p\, dx + C_\delta \|v\|_{L_\infty(\R)}^p.
\end{equation}
Combining \eqref{eq:alternative_1} and \eqref{eq:alternative_2} and recalling that $|I_k|=2^{-k}$, we find that
\begin{equation*}
    \begin{split}
        \int_{I_k} \zeta^{m-n} |\partial_x^j v|^p\, dx
        &\leq
        C_{m,n}\ \delta\ (2^{-k})^{m-n} |I_k|^{(3-j)p} \int_{I_k} |\partial_x^3 v|^p\, dx
        +
        C_{m,n}\ C_\delta\ (2^{-k})^{m-n} |I_k|^{-jp+1} \|v\|_{L_\infty(\R)}^p
        \\
        &\leq
        C_{m,n}\ \delta\ \int_{I_k} \zeta^{m-n+(3-j)p} |\partial_x^3 v|^p\, dx
        +
        C_{m,n}\ C_\delta\ (2^{-k})^{m-n-jp+1} \|v\|_{L_\infty(\R)}^p.
    \end{split}
\end{equation*}
Since $\zeta(x) \in (0,1)$ on $I_k$ and $n \leq (3-j)p$ by assumption, we deduce that
\begin{equation*}
    \int_{I_k} \zeta^{m-n+(3-j)p} |\partial_x^3 v|^p\, dx
    \leq
    \int_{I_k} \zeta^{m} |\partial_x^3 v|^p\, dx.
\end{equation*}
Moreover, using that also $n+jp \leq m+1$ by assumption, we find that
\begin{equation*}
    \sum_{k \in \N} (2^{-k})^{m-n-jp+1} \leq C < \infty,
\end{equation*}
and, consequently, summing up over all $k \in \N$ yields the desired estimate
\begin{equation*}
    \int_{0}^{1/6}  \zeta^{m-n} |\partial_x^j v|^p\, dx
    =
    \sum_{k \in \N} \int_{I_k} \zeta^{m-n} |\partial_x^j v|^p\, dx
    \leq
    \delta \int_{0}^{1/16} \zeta^m |\partial_x^3 v|^p\, dx 
    +
    C_\delta \|v\|_{L_\infty(\R)}^p.
\end{equation*}

\noindent(iv) The same argument as in (iii) applies on the subinterval $[7/16,1/2]$, where $\zeta$ is also linear.

\noindent(v) It remains to prove \eqref{eq:weigthed_int_est_q}. As in (iii) we set $v_k(x)=v(2^k x)$ and obtain
\begin{equation}\label{eq:est_zeta_q}
    \int_{I_k} \zeta^{m-n} |\partial_x^j v|^p\, dx
    \leq
    C_{m,n}(2^{-k})^{m-n} |I_k|^{-jq+1} \int_{0}^{1/16} |\partial_x^j v_k|^q\, dx.
\end{equation}
Moreover, we can derive the estimate
\begin{equation} \label{eq:GN_Young_q}
    \begin{split}
        \int_{0}^{1/16} |\partial_x^j v_k|^q\, dx 
        &\leq
        \left(\int_{0}^{1/16} |\partial_x^3 v_k|^q\, dx\right)^{j/3} \left(\int_{0}^{1/16} |v_k|^q\, dx\right)^\frac{3-j}{3} 
        \\
        &\leq
        C
        \left(\int_{0}^{1/16} |\partial_x^3 v_k|^p\, dx\right)^\frac{qj}{3p} \left(\int_{0}^{1/16} |v_k|^p\, dx\right)^\frac{q(3-j)}{3} 
        \\
        &\leq
        \delta \int_{0}^{1/16} |\partial_x^3 v_k|^p\, dx + \tilde{C}_\delta \|v\|_{L_\infty(\R)}^\frac{qp(3-j)}{3p-qj}
        \\
        &\leq
        \delta |I_k|^{3p-1} \int_{I_k} |\partial_x^3 v_k|^p\, dx + \tilde{C}_\delta \|v\|_{L_\infty(\R)}^\frac{qp(3-j)}{3p-qj},
    \end{split}
\end{equation}
where $\tilde{C}_\delta = C \delta^{-\frac{qj}{3p-qj}} > 0$. Here we used H\"older's inequality and then Young's inequality with exponents $\frac{3p}{qj}$ and $\frac{3p}{3p-qj}$. Combining \eqref{eq:GN_Young_q} and \eqref{eq:est_zeta_q} leads to the inequality 
\begin{equation*}
    \int_{I_k} \zeta^{m-n} |\partial_x^j v|^p\, dx
    \leq
    C_{m,n}\ \delta\ \int_{I_k} \zeta^{m-n+(3-j)p} |\partial_x^3 v|^p\, dx
    +
    C_{m,n}\ \tilde{C}_\delta\ (2^{-k})^{m-n-jp+1} \|v\|_{L_\infty(\R)}^\frac{qp(3-j)}{3p-qj}.
\end{equation*}
This completes the proof.
\end{proof}

\bigskip

\subsection{Regularity estimates}
With the weighted interpolation estimates of Proposition \ref{prop:weighted_Sobolev} we are able to prove an interior regularity estimate for general nonlinear fourth-order degenerate parabolic equations of the form \eqref{eq:general_parabolic} under certain structural conditions (cf. \eqref{eq:conditions_Phi}) on the nonlinearity. 


\begin{theorem}[Interior regularity estimate] \label{thm:reg_estimate_general}
Let $v \in L_\infty\bigl([1/4,1]\times [a,b]\bigr)$ with $\|v\|_{L_\infty([1/4,1]\times [a,b])} \leq 1$ be a weak solution to the fourth-order degenerate parabolic problem \eqref{eq:general_parabolic}
in the sense of Definition \ref{def:sol_general_parabolic}, where $S\in L_1\bigl((1/4,1);H^2([a,b])\bigr)$ and the function $\Phi$ satisfies \eqref{eq:conditions_Phi}.
Then, for all $x_0 \in [a,b]$ and all $\kappa > 0$ such that $a < x_0-2\kappa < x_0 + 2\kappa < b$, we have the estimate
\begin{equation*}
    \int_{1/2}^{1} \int_{x_0-\kappa}^{x_0+\kappa} |\partial_x^3 v|^{\alpha+1}\, dx\, d\tau
    \leq
    C + C \|S\|_{L_1((1/4,1);H^2([a,b]))}
\end{equation*}
with a positive constant $C > 0$.
\end{theorem}


\begin{remark}
We use the interior regularity result in Theorem \ref{thm:reg_estimate_general} in order to study the long-time behaviour of solutions to the fluid dynamical problem \eqref{eq:PDE_alpha>1}, cf. Theorem \ref{thm:global-ex_stability} (i) below. In particular, the estimate provides control of the Fourier modes $h_{\pm 1}$ and therewith control of the motion of the circle's center.
Since (by testing the differential equation with $e^{\pm i \theta)}$) controlling the Fourier modes $h_{\pm 1}$ corresponds to controlling integrals of $|\partial_\theta^3 h|^\alpha$, this requires to derive a similar estimate as in Theorem \ref{thm:reg_estimate_general} with exponent $\alpha$ instead of $\alpha+1$, cf. Theorem \ref{thm:improved_estimate_alpha} below. This is achieved by an iterative argument that combines Theorem \ref{thm:reg_estimate_general} with the energy-dissipation formula.
\end{remark}

\begin{remark}
Due to the absence of boundary conditions in the fluid dynamical problem \eqref{eq:PDE_alpha>1}, the interior regularity estimate derived in Theorem \ref{thm:reg_estimate_general} could in fact be replaced by a global estimate. However, we decide to keep the general interior estimate in Theorem \ref{thm:reg_estimate_general} because it is of general independent interest.
\end{remark}


\begin{proof}
Without loss of generality we assume that $\kappa = 1/8$ and that $a, b$ and $x_0$ are such that $a < x_0 - 1/4 < x_0 + 1/4 < b$.

\noindent(i) Let $\xi \in C\bigl([0,1]\times [a,b]\bigr)$ with $\text{supp}(\xi) = [1/4,1]\times[x_0-1/4,x_0+1/4]$ be defined such that
\begin{equation*}
    \xi(\tau,\theta) = \chi(\tau) \zeta(\theta), \quad \tau \in [0,1],\ x \in [a,b],
\end{equation*}
with  
\begin{equation*}
    \chi \in C^\infty\bigl([0,1]\bigr), \quad \text{supp}(\chi) = [1/4,1]
    \quad \text{and} \quad
    \chi(\tau) \equiv 1, \quad \tau \in [1/2,1],
\end{equation*}
and $\zeta$ as in Proposition \ref{prop:weighted_Sobolev}.
Moreover, let $m > 3(\alpha+1)$.
\begin{center}
	\begin{tikzpicture}[domain=0:10, scale=0.75] 
		\draw[very thick,<->] (16,0) 
		node[right] {$x$} -- (0,0) -- (0,5) node[above] {$\tau$};
		\draw (0,4)
		node[left] {\tiny{$1$}} -- (14,4) -- (14,0) node[below] {\tiny{$2\pi$}};
		\draw (6.25,4) -- (6.25,2.5) -- (7.75,2.5) -- (7.75,4);
		\draw (5.75,4) -- (5.75,1.25) -- (8.25,1.25) -- (8.25,4);
		\node[below] at (7,0) {\tiny{$x_0$}};
		\node[below] at (5.75,0) {\tiny{$x_0-1/4$}};
		\node[below] at (8.25,0) {\tiny{$x_0+1/4$}};
		\node[left] at (0,2.5) {\tiny{$1/2$}};
		\node[left] at (0,1.25) {\tiny{$1/4$}};
		\node at (7,3.25) {\tiny{$\xi \equiv 1$}};
	\end{tikzpicture} 
\end{center}

\noindent(ii) In order to prove the theorem, we test the partial differential equation \eqref{eq:general_parabolic} with $\partial_x(\xi^m \partial_x v)$. 
This corresponds to calculating the derivative
\begin{equation*}
    \begin{split}
        &
        \frac{d}{d\tau} \Bigl(\frac{1}{2} \int_a^b \xi^m |\partial_x v|^2\, dx\Bigr)
        =
        \int_a^b \xi^m \partial_x v \partial_x \partial_\tau v\, dx
        +
        \frac{m}{2} \int_a^b \xi^{m-1} \partial_\tau \xi |\partial_x v|^2\, dx
        \\
        &=
        -\int_a^b \partial_\tau v \partial_x(\xi^m \partial_x v)\, dx
        +
        \frac{m}{2} \int_a^b \xi^{m-1} \partial_\tau \xi |\partial_x v|^2\, dx
        \\
        &=
        -\int_a^b \Phi\bigl(v,\partial_x v,\partial_x^2 v, \partial_x^3 v\bigr)\, \partial_x^2(\xi^m \partial_x v)\, dx
        -
        \int_a^b S(\tau,x)\, \partial_x(\xi^m \partial_x v)\, dx
        +
        \frac{m}{2} \int_a^b \xi^{m-1} \partial_\tau \xi |\partial_x v|^2\, dx.
    \end{split}
\end{equation*}
Since $m > 2$ by assumption, we may calculate the derivatives $\partial_x^2(\xi^m) 
=m (m-1) \xi^{m-2} |\partial_x \xi|^2 + m \xi^{m-1} \partial_x^2 \xi$ 
and find that
\begin{equation*}
    \begin{split}
        \frac{d}{d\tau} \Bigl(\frac{1}{2} & \int_a^b \xi^m |\partial_x v|^2\, dx\Bigr)
        \\
        &=
        \frac{m}{2} \int_a^b \xi^{m-1} \partial_\tau \xi |\partial_x v|^2\, dx
        -
        \int_a^b S(\tau,x)\, \partial_x(\xi^m \partial_x v)\, dx
        \\
        &\quad
        -
        \int_a^b \xi^m \Phi\bigl(v,\partial_x v,\partial_x^2 v, \partial_x^3 v\bigr)\, \partial_x^3 v\, dx \\
        &\quad
        - 2m
        \int_a^b
        \xi^{m-1} \partial_x \xi\, \Phi\bigl(v,\partial_x v,\partial_x^2 v, \partial_x^3 v\bigr)\, \partial_x^2 v\, dx 
        \\
        &\quad
        -m \int_a^b
        \left((m-1) \xi^{m-2} |\partial_x \xi|^2 + \xi^{m-1} \partial_x^2 \xi\right)\,\Phi\bigl(v,\partial_x v,\partial_x^2 v, \partial_x^3 v\bigr)\, \partial_x v\, dx.
    \end{split}
\end{equation*}
Integrating with respect to time and using that $\xi$ is compactly supported in time on $[1/4,1]$, we obtain
\begin{equation*}
    \int_{1/4}^{1}\frac{d}{d\tau} \left(\frac{1}{2} \int_a^b \xi^m |\partial_x v|^2\, dx\right)\, d\tau 
    =
    \frac{1}{2} \int_a^b \xi^m(1,x)\, |\partial_x v(1,x)|^2\, dx 
    \geq 0
\end{equation*}
and hence
\begin{equation} \label{eq:est_diss_reg}
    \begin{split}
        \int_{1/4}^{1} & \int_a^b \xi^m \Phi\bigl(v,\partial_x v,\partial_x^2 v, \partial_x^3 v\bigr)\, \partial_x^3 v\, dx\, d\tau
        \\
        &\leq
        \frac{m}{2}
        \int_{1/4}^{1} \int_a^b \xi^{m-1} \partial_\tau \xi |\partial_x v|^2\, dx\, d\tau
        -
        \int_{1/4}^{1} \int_a^b S(\tau,x)\, \partial_x(\xi^m \partial_x v)\, dx\, d\tau
        \\
        &\quad
        -m \int_{1/4}^{1} \int_a^b
        \left((m-1) \xi^{m-2} |\partial_x \xi|^2 + \xi^{m-1} \partial_x^2 \xi\right)\,\Phi\bigl(v,\partial_x v,\partial_x^2 v, \partial_x^3 v\bigr)\, \partial_x v\, dx\,
        d\tau
        \\
        &\quad
        - 2m \int_{1/4}^{1}
        \int_a^b
        \xi^{m-1} \partial_x \xi\, \Phi\bigl(v,\partial_x v,\partial_x^2 v, \partial_x^3 v\bigr)\, \partial_x^2 v\, dx\, d\tau.
    \end{split}
\end{equation}

\noindent(iii) The next issue is to estimate the terms on the right-hand side of \eqref{eq:est_diss_reg} in a way that allows us to absorb parts of it in the dissipative term on the left-hand side and to keep only terms with lower-order derivatives on the right-hand side. For this purpose, observe that the first term  on the right-hand side of \eqref{eq:est_diss_reg} satisfies
\begin{equation*}
        \frac{m}{2} \int_{1/4}^{1} \int_a^b \xi^{m-1} \partial_\tau \xi\, |\partial_x v|^2\, dx\, d\tau
        \leq
        C(m,\partial_\tau \xi)
        \int_{1/4}^{1} \int_a^b \xi^{m-1} |\partial_x v|^2\, dx\, d\tau.
\end{equation*}
Applying integration by parts in space twice 
on the second term on the right-hand side of \eqref{eq:est_diss_reg} yields
\begin{equation*}
    \begin{split}
        &
        \int_{1/4}^{1} \int_a^b S(\tau,x)\, \partial_x(\xi^m \partial_x v)\, dx\, d\tau
        \\
        &\quad 
        \leq
        \int_{1/4}^{1} \int_a^b
        |\xi^m \partial_x^2 S(\tau,x) + m \xi^{m-1} \partial_x \xi \partial_x S(\tau,x)|\cdot |v|\, dx\, d\tau 
        \\
        &\quad 
        \leq
        \left(
        \int_{1/4}^{1} \int_a^b
        |\xi^m \partial_x^2 S(\tau,x) + m \xi^{m-1} \partial_x \xi \partial_x S(\tau,x)|\, dx\, d\tau\right)
        \|v\|_{L_\infty([1/4,1]\times [a,b])}
        \\
        &\quad 
        \leq
        C(m,\xi,\partial_x \xi) \|S\|_{L_1((1/4,1);H^2([a,b])} \|v\|_{L_\infty([1/4,1]\times [a,b])}.
    \end{split}
\end{equation*}
Since we have a sum with different powers of $\xi$ in the third term on the right-hand side of \eqref{eq:est_diss_reg}, we consider both summands separately. For the first summand we proceed as follows. Using the upper bound $\eqref{eq:conditions_Phi}_1$ on $|\Phi|$, we find that
\begin{equation*}
    \begin{split}
        &
        \int_{1/4}^{1} \int_a^b
        \xi^{m-2} |\partial_x \xi|^2\, \Phi\bigl(v,\partial_x v,\partial_x^2,\partial_x^3 v\bigr) \partial_x v\, dx\,d\tau
        \\
        &\quad 
        \leq
        C(\partial_x \xi) \int_{1/4}^{1} \int_a^b \xi^{m-2}  |\partial_x^3 v|^\alpha\, |\partial_x v|\, dx\,d\tau
        +
        C(\partial_x \xi) \int_{1/4}^{1} \int_a^b \xi^{m-2}  \sum_{k=0}^2 |\partial_x^k v|^\alpha\, |\partial_x v|\, dx\,d\tau.
    \end{split}
\end{equation*}
Furthermore, applying first H\"older's and then Young's inequality (both with exponents $p=(\alpha+1)/\alpha$ and $q=\alpha+1$), we obtain that for all $\eps_0 > 0$ there exists a constant $C_{\eps_0} > 0$ such that the estimate
\begin{equation*} 
    \begin{split}
        &
        \int_{1/4}^{1} \int_a^b \xi^{m-2}  |\partial_x^3 v|^\alpha\, |\partial_x v|\, dx\,d\tau
        \\
        &\quad
        \leq
        C(m,\partial_x \xi)
        \left(\int_{1/4}^{1} \int_a^b \xi^m  |\partial_x^3 v|^{\alpha+1}\, dx\,d\tau\right)^\frac{\alpha}{\alpha+1}
        \left(\int_{1/4}^{1} \int_a^b \xi^{m-2(\alpha+1)} |\partial_x v|^{\alpha+1}\, dx\,d\tau\right)^\frac{1}{\alpha+1}
        \\
        &\quad
        \leq
        \eps_0 \int_{1/4}^{1} \int_a^b \xi^m  |\partial_x^3 v|^{\alpha+1}\, dx\,d\tau 
        + 
        C_{\eps_0} \int_{1/4}^{1} \int_a^b \xi^{m-2(\alpha+1)} |\partial_x v|^{\alpha+1}\, dx\,d\tau
    \end{split}
\end{equation*}
holds true. 
For the integral containing the lower-order derivatives, we obtain in the same way
\begin{equation*}
    \int_{1/4}^{1} \int_a^b \xi^{m-2}  \sum_{k=0}^2 |\partial_x^k v|^\alpha\, |\partial_x v|\, dx\,d\tau
    \leq
    \sum_{k=0}^2 \bar{C}_k \int_{1/4}^{1} \int_a^b \xi^{m-2}\, |\partial_x^k v|^{\alpha+1}\, dx\, d\tau.
\end{equation*}
Thus, we have
\begin{equation*}
    \begin{split}
        &
        \int_{1/4}^{1} \int_a^b
        \xi^{m-2} |\partial_x \xi|^2\, \Phi\bigl(v,\partial_x v,\partial_x^2,\partial_x^3 v\bigr) \partial_x v\, dx\,d\tau
        \\
        &\quad
        \leq
        \eps_0 \int_{1/4}^{1} \int_a^b \xi^m  |\partial_x^3 v|^{\alpha+1}\, dx\,d\tau 
        + 
        C_{\eps_0} \int_{1/4}^{1} \int_a^b \xi^{m-2(\alpha+1)} |\partial_x v|^{\alpha+1}\, dx\,d\tau
        \\
        &\quad\quad
        +
        \sum_{k=0}^2 \bar{C}_k \int_{1/4}^{1} \int_a^b \xi^{m-2}\, |\partial_x^k v|^{\alpha+1}\, dx\, d\tau.
    \end{split}
\end{equation*}
Proceeding similarly with the second summand yields for all $\eps_1 > 0$ the existence of a constant $C_{\eps_1} > 0$ such that
\begin{equation*}
    \begin{split}
        & 
        \int_{1/4}^{1} \int_a^b
        \xi^{m-1}\, |\partial_x^2 \xi|^2\, \Phi\bigl(v,\partial_x v,\partial_x^2 v,\partial_x^3 v\bigr) \partial_x v\, dx\,d\tau
        \\
        &\quad 
        \leq
        \eps_1 \int_{1/4}^{1} \int_a^b \xi^m\, |\partial_x^3 v|^{\alpha+1}\, dx\,d\tau 
        + C_{\eps_1} \int_{1/4}^{1} \int_a^b \xi^{m-(\alpha+1)}\, |\partial_x v|^{\alpha+1}\, dx\,d\tau
        \\
        &\quad\quad
        +
        \sum_{k=0}^2 \bar{C}_k \int_{1/4}^{1} \int_a^b \xi^{m-1}\, |\partial_x^k v|^{\alpha+1}\, dx\, d\tau.
    \end{split}
\end{equation*}
Analogously, we obtain the following for the fourth integral on the right-hand side of \eqref{eq:est_diss_reg}. For all $\eps_2 > 0$ there exists a constant $C_{\eps_2} > 0$ such that
\begin{equation*}
    \begin{split}
        & \int_{1/4}^{1} \int_a^b \xi^{m-1} \partial_x \xi\, \Phi\bigl(v,\partial_x v,\partial_x^2 v,\partial_x^3 v\bigr)\,\partial_x^2 v\, dx\, d\tau
        \\
        &\quad
        \leq
        \eps_2 \int_{1/4}^{1} \int_a^b \xi^m  |\partial_x^3 v|^{\alpha+1}\, dx\,d\tau
        + 
        C_{\eps_2} \int_{1/4}^{1} \int_a^b \xi^{m-(\alpha+1)} |\partial_x^2 v|^{\alpha+1}\, dx\,d\tau
        \\
        &\quad\quad
        +
        \sum_{k=0}^2 \bar{C}_k \int_{1/4}^{1} \int_a^b \xi^{m-1}\, |\partial_x^k v|^{\alpha+1}\, dx\, d\tau.
    \end{split}
\end{equation*}
We now use both the upper bound $\eqref{eq:conditions_Phi}_1$ on $|\Phi(\cdot)|$ and the lower bound $\eqref{eq:conditions_Phi}_2$ on $\Phi(\cdot)\partial_x^3 v$ to obtain
\begin{equation*}
    \begin{split}
        &
        \int_{1/4}^{1} \int_a^b \xi^m \Phi\bigl(v,\partial_x v,\partial_x^2 v,\partial_x^3 v\bigr)\, \partial_x^3 v\, dx\,d\tau
        \\
        &\quad
        \geq
        \int_{1/4}^{1} \int_a^b \xi^m  |\partial_x^3 v|^{\alpha+1}\,  dx\,d\tau
        -
        \sum_{k=0}^2 C_k \int_{1/4}^{1} \int_a^b \xi^m |\Phi(v,\partial_x v,\partial_x^2 v,\partial_x^3 v)|\cdot |\partial_x^k v|\, dx\,d\tau
        \\
        &\quad
        \geq 
        \int_{1/4}^{1} \int_a^b \xi^m  |\partial_x^3 v|^{\alpha+1}\,  dx\,d\tau
        -
        \sum_{k=0}^3 \sum_{j=0}^2 C_k \int_{1/4}^{1} \int_a^b \xi^m |\partial_x^k v|^\alpha |\partial_x^j v|\, dx\,d\tau.
    \end{split}
\end{equation*}
Recalling that $\xi \in C\bigl([0,1]\times [a,b]\bigr)$ with $0 \leq \xi \leq 1$ we find that
$\xi^{m-1} \leq \xi^{m-2} \leq \xi^{m-(\alpha+1)} \leq \xi^{m-2(\alpha+1)}$. Thus, choosing $\eps_i$ small enough for $i \in \{0,1,2\}$, we arrive at the estimate
\begin{equation*}
    \begin{split}
        &
        \int_{1/4}^{1} \int_a^b \xi^m\,  |\partial_x^3 v|^{\alpha+1}\, dx\,d\tau
        \\
        &\quad 
        \leq
        C \int_{1/4}^{1} \int_a^b
        \xi^{m-1} |\partial_x v|^2\, dx\, d\tau
        +
        C 
        \|S\|_{L_1([1/4,1];H^2([a,b]))}
        \|v\|_{L_\infty([1/4,1]\times [a,b])}
        \\
        &\quad\quad
        +
        C \int_{1/4}^{1} \int_a^b \xi^{m-2(\alpha+1)}
        |\partial_x v|^{\alpha+1}\, dx\, d\tau
        +
        C \int_{1/4}^{1} \int_a^b
        \xi^{m-(\alpha+1)} \Bigl(|\partial_x v|^{\alpha+1} + |\partial_x^2 v|^{\alpha+1}\Bigr)\, dx\,d\tau
        \\
        &\quad\quad
        +
        \sum_{k=0}^2 \bar{C}_k \int_{1/4}^{1} \int_a^b \Bigl(\xi^{m-2} + \xi^{m-1}\Bigr)\, |\partial_x^k v|^{\alpha+1} 
        + \sum_{k=0}^3 \sum_{j=0}^2 C_k \int_{1/4}^{1} \int_a^b \xi^m |\partial_x^k v|^\alpha |\partial_x^j v|\, dx\,d\tau
        \\
        &\leq
        C \int_{1/4}^{1} \int_a^b
        \xi^{m-1} |\partial_x v|^2\, dx\, d\tau
        +
        C 
        \|S\|_{L_1([1/4,1];H^2([a,b]))}
        \|v\|_{L_\infty([1/4,1]\times [a,b])}
        \\
        &\quad\quad
        +
        C \int_{1/4}^{1} \int_a^b \xi^{m-2(\alpha+1)}
        |\partial_x v|^{\alpha+1}\, dx\, d\tau
        +
        C \int_{1/4}^{1} \int_a^b
        \xi^{m-(\alpha+1)} |\partial_x^2 v|^{\alpha+1}\, dx\,d\tau
        \\
        &\quad\quad
        +
        C \|v\|_{L_\infty([1/4,1]\times [a,b])}^{\alpha+1} 
        + 
        \sum_{k=0}^3 \sum_{j=0}^2 C_k \int_{1/4}^{1} \int_a^b \xi^m |\partial_x^k v|^\alpha |\partial_x^j v|\, dx\,d\tau,
    \end{split}
\end{equation*}
where $C = C(m,\partial_x \xi,\partial_x^2 \xi,\eps_i) > 0$. 

Next, using Young's inequality for $k=3$ and all $j \in \{0,1,2\}$ in the last term on the right-hand side, we observe that for all $\eps_4 > 0$, there exists $C_{\eps_4} > 0$ such that for all 
\begin{equation*}
    \begin{split}
        \int_{1/4}^{1} \int_a^b & \xi^m |\partial_x^3 v|^\alpha |\partial_x^j v|\, dx\, d\tau
        \\
        &\leq
        \eps_4 \int_{1/4}^{1} \int_a^b \xi^m |\partial_x^3 v|^{\alpha+1}\, dx\,d\tau
        +
        C_{\eps_4}\int_{1/4}^{1} \int_a^b \xi^m |\partial_x^j v|^{\alpha+1}\, dx\,d\tau
        \\
        &\leq
        \eps_4 \int_{1/4}^{1} \int_a^b \xi^m |\partial_x^3 v|^{\alpha+1}\, dx\,d\tau
        +
        C_{\eps_4}\int_{1/4}^{1} \int_a^b \xi^{m-2(\alpha+1)} |\partial_x^j v|^{\alpha+1}\, dx\,d\tau.
    \end{split}
\end{equation*}
For the remaining lower-order terms we have, as above,
\begin{equation*}
    \begin{split}
        &
        \sum_{k=0}^2 \sum_{j=0}^2 C_k \int_{1/4}^{1} \int_a^b \xi^m |\partial_x^k v|^\alpha |\partial_x^j v|\, dx\,d\tau
        \\
        &\quad 
        \leq
        \sum_{k=0}^2 \bar{C}_k \int_{1/4}^{1} \int_a^b \xi^m\, |\partial_x^k v|^{\alpha+1}\, dx\, d\tau   
        \\
        &\quad
        \leq
        C \|v\|_{L_\infty([1/4,1]\times [a,b])}^{\alpha+1}
        +
        C \int_{1/4}^{1} \int_a^b \xi^{m-2(\alpha+1)}
        |\partial_x v|^{\alpha+1}\, dx\, d\tau
        +
        C \int_{1/4}^{1} \int_a^b
        \xi^{m-(\alpha+1)} |\partial_x^2 v|^{\alpha+1}\, dx\,d\tau.
    \end{split}
\end{equation*}
Consequently, we arrive at the estimate
\begin{equation*}
    \begin{split}
        &
        \int_{1/4}^{1} \int_a^b \xi^m |\partial_x^3 v|^{\alpha+1}\, dx\,d\tau
        \\
        &\quad 
        \leq
        C \int_{1/4}^{1} \int_a^b
        \xi^{m-1} |\partial_x v|^2\, dx\, d\tau
        +
        C 
        \|S\|_{L_1([1/4,1];H^2([a,b]))}
        \|v\|_{L_\infty([1/4,1]\times [a,b])}
        \\
        &\quad\quad
        +
        C \int_{1/4}^{1} \int_a^b \xi^{m-2(\alpha+1)}
        |\partial_x v|^{\alpha+1}\, dx\, d\tau
        +
        C \int_{1/4}^{1} \int_a^b
        \xi^{m-(\alpha+1)} |\partial_x^2 v|^{\alpha+1}\, dx\,d\tau
        \\
        &\quad\quad
        +
        C \|v\|_{L_\infty([1/4,1]\times [a,b])}^{\alpha+1}.
    \end{split}
\end{equation*}

\noindent(iv) Now we apply the weighted interpolation estimates from Proposition \ref{prop:weighted_Sobolev} as follows. First, choosing $p=\alpha+1, n = 2(\alpha+1)$ and $j=1$, we obtain
\begin{equation*}
    \begin{split}
        \int_{1/4}^{1}\int_a^b \xi^{m-2(\alpha+1)} |\partial_x v|^{\alpha+1}\, dx\, d\tau 
        &=
        \int_{1/4}^{1} \chi^{m-2(\alpha+1)} \int_a^b \zeta^{m-2(\alpha+1)} |\partial_x v|^{\alpha+1}\, dx\, d\tau 
        \\
        &\leq
        \int_{1/4}^{1} \chi^{m-2(\alpha+1)} \left(
        \delta(\tau) \int_a^b \zeta^m |\partial_x^3 v|^{\alpha+1}\, dx
        + 
        C \delta(\tau)^{-\frac{1}{2}} \|v\|_{L_\infty([a,b])}^{\alpha+1}
        \right)\, d\tau.
    \end{split}
\end{equation*}
Choosing $\delta(\tau) = \delta_0 \chi(\tau)^{2(\alpha+1)}$ for some fixed $\delta_0 > 0$ and each $\tau \in [1/4,1]$ we find that
\begin{equation} \label{eq:weight_est_1}
    \int_{1/4}^{1} \int_a^b \xi^{m-2(\alpha+1)} |\partial_x v|^{\alpha+1}\, dx\, d\tau 
    \leq
    \delta_0
    \int_{1/4}^1 \int_a^b \xi^m |\partial_x^3 v|^{\alpha+1}\, dx\, d\tau + 
    C_{\delta_0,\chi} \|v\|_{L_\infty((1/4,1)\times [a,b])}^{\alpha+1}
\end{equation}
for $m > 3(\alpha+1)$ such that $\chi^{m-3(\alpha+1)}$ is integrable.

Similarly, we apply Proposition \ref{prop:weighted_Sobolev} with $p=n=\alpha+1$ and $j=1$ to derive the estimate
\begin{equation*}
    \begin{split}
        \int_{1/4}^{1} v \xi^{m-(\alpha+1)} |\partial_x v|^{\alpha+1}\, dx\, d\tau 
        &\leq
        \int_{1/4}^{1} \chi^{m-(\alpha+1)} \left(
        \delta(\tau) \int_a^b \zeta^m |\partial_x^3 v|^{\alpha+1}\, dx
        + 
        C \delta(\tau)^{-\frac{1}{2}} \|v\|_{L_\infty([a,b])}^{\alpha+1}
        \right)\, d\tau
    \end{split}
\end{equation*}
and then, choosing $\delta(\tau) = \delta_0 \chi(\tau)^{\alpha+1}$,
\begin{equation} \label{eq:weight_est_2}
    \int_{1/4}^{1} \int_a^b \xi^{m-(\alpha+1)} |\partial_x v|^{\alpha+1}\, dx\, d\tau 
    \leq
    \delta_0
    \int_{1/4}^1 \int_a^b \xi^m |\partial_x^3 v|^{\alpha+1}\, dx\, d\tau + 
    C_{\delta_0,\chi} \|v\|_{L_\infty([1/4,1]\times [a,b])}^{\alpha+1}
\end{equation}
for $m > \tfrac{\alpha+1}{2}$ such that $\chi^{m-3(\alpha+1)}$ is integrable.

Finally, applying the second inequality of Proposition \ref{prop:weighted_Sobolev} with $p=\alpha+1, q = 2$ and $n = j = 1$, we find that
\begin{equation*}
    \begin{split}
        \int_{1/4}^{1} \int_a^b \xi^{m-1} |\partial_x v|^{2}\, dx\, d\tau 
        &\leq
        \int_{1/4}^{1} \chi^{m-1} \left(
        \delta(\tau) \int_a^b \zeta^m |\partial_x^3 v|^{\alpha+1}\, dx
        + 
        C \delta(\tau)^{-\frac{2}{3\alpha+1}} \|v\|_{L_\infty([a,b])}^\frac{4(\alpha+1)}{3\alpha+1}
        \right)\, d\tau
    \end{split}
\end{equation*}
and then, choosing $\delta(\tau) = \delta_0 \chi(\tau)$,
\begin{equation} \label{eq:weight_est_3}
    \int_{1/4}^{1} \int_a^b \xi^{m-1} |\partial_x v|^{2}\, dx\, d\tau 
    \leq
    \delta_0
    \int_{1/4}^1 \int_a^b \xi^m |\partial_x^3 v|^{\alpha+1}\, dx\, d\tau + 
    C_{\delta_0,\chi} \|v\|_{L_\infty((1/4,1)\times [a,b])}^{\alpha+1}
\end{equation}
for $m > \tfrac{3(\alpha+1)}{3\alpha+1}$ such that $\chi^{m-\tfrac{3(\alpha+1)}{3\alpha+1}}$ is integrable.

Combining \eqref{eq:weight_est_1}--\eqref{eq:weight_est_3}, we observe that there exists a constant $C > 0$ such that
\begin{equation*}
    \begin{split}
        \int_{1/4}^{1} \int_a^b \xi^m |\partial_x^3 v|^{\alpha+1}\, dx\,d\tau
        &\leq
        C \left(\|v\|_{L_\infty([1/4,1]\times [a,b])}^{\alpha+1} 
        +
        \|v\|_{L_\infty([1/4,1]\times [a,b])}^\frac{4(\alpha+1)}{3\alpha+1}\right)
        \\
        &\quad 
        +
        C \|S\|_{L_1([1/4,1]\times [a,b])} \|v\|_{L_\infty([1/4,1]\times [a,b])}.
    \end{split}
\end{equation*}
Since $\|v\|_{L_\infty([1/4,1]\times [a,b])} \leq 1$ by assumption, we obtain
\begin{equation*}
    \int_{1/4}^{1} \int_a^b \xi^m |\partial_x^3 v|^{\alpha+1}\, dx\,d\tau
    \leq
    C +
    C \|S\|_{L_1([1/4,1]\times [a,b])}.
\end{equation*}
and by definition of $\xi$ we finally arrive at the desired estimate
\begin{equation*}
    \int_{1/2}^{1} \int_{x_0-\kappa}^{x_0+\kappa}  |\partial_x^3 v|^{\alpha+1}\, dx\,d\tau
    \leq
    C +
    C \|S\|_{L_1([1/4,1]\times [a,b])}.
\end{equation*}
\end{proof}

\bigskip

\section{Optimal decay estimates for the derivatives of the Fourier modes $h_{\pm 1}$ of  solutions to \eqref{eq:PDE_regularised}} \label{ssec:Fourier_modes}

The content of this subsection is twofold. First, we introduce some suitable rescaling that allows us a simple application of the general interior regularity estimate, derived in Theorem \ref{thm:reg_estimate_general}, to the regularised problem \eqref{eq:PDE_regularised}. This provides us with an $L_1\bigl((0,t_\ast);L_{\alpha+1}(S^1)\bigr)$-bound for the third derivative of the remainder $\phi$ of the local solution
\begin{equation*}
    h(t,\theta) 
    =
    \bar{h}_0 + h_{\pm 1}(t) e^{\pm i\theta} + \phi(t,\theta), 
    \quad t \in [0,t_\ast],\ \theta \in S^1.
\end{equation*}
to \eqref{eq:PDE_regularised}. However, recall that we have seen in Section \ref{sec:conv_to_circle} that, if this solution corresponds to an initial value $h_0 \in C^\infty(S^1)$ and an $\eps_0>0$ such that
\begin{equation*}
    \frac{1}{2\pi}\int_{S^1} h_0\, d\theta = \bar{h}_0
    \quad \text{and} \quad
    \|h_0 - \bar{h}_0\|_{H^1(S^1)} \leq \eps,
    \quad 
    \eps \in (0,\eps_0),
\end{equation*}
then $\|\phi(t)\|_{C([\bar{t}/4,\bar{t}];H^1(S^1))} \leq C \Lambda_\eps(\bar{t})$ for all $0 \leq \bar{t} \leq t_\ast$, where $t_\ast$ denotes, as before, the maximal time of existence until which the solution remains bounded from below and from above by $\tfrac{1}{2}\bar{h}_0 \leq h(t,\theta) \leq 2 \bar{h}_0$. This tells us that the solution remains close to a circle in the $H^1(S^1)$-norm for small times and it converges at rate $(\bar{t})^{-\frac{1}{\alpha-1}}$ for times $\bar{t} \geq \eps^{-\frac{\alpha-1}{2}}$. The second part of this section is devoted to the issue of controlling the position of the circle's center. This corresponds to deriving estimates for the Fourier modes $h_{\pm 1}(t)$ and hence for the $L_1\bigl((0,t_\ast);L_{\alpha}(S^1)\bigr)$-norm of the third derivatives of $\phi$. Note that arriving from the $L_{\alpha+1}(S^1)$-estimate to the $L_\alpha(S^1)$-estimate provides an improved inequality for small third derivatives. The proof is based on an iterative scheme that, in the end, provides us with an estimate of the form
\begin{equation*}
    \int_{0}^{\bar{t}} |h_{\pm 1}^\prime(t)|\, dt
    \leq 
    C \eps\left(1 + \log\bigl(\eps^{1-\alpha}\bigr)\right),
    \quad
    0 \leq \bar{t} \leq t_\ast.
\end{equation*}
It turns out that with the estimates we can arrive to any finite time $T$ as long as the regularisation parameter satisfies $0 < \sigma \leq \bar{\sigma}$ for some sufficiently small $\bar{\sigma}$.

\bigskip

\noindent\textbf{\textsc{Rescaling. }} 
Let $t_\ast > 0$ and $\eps \in (0,\eps_0)$ be as above. The evolution equation $\eqref{eq:PDE_regularised}_1$ may then be rewritten as an equation for $\phi$ as follows
\begin{equation} \label{eq:PDE_phi}
    \partial_t \phi + \partial_\theta\left(|\bar{h} + h_{\pm 1} e^{\pm i\theta} + \phi|^{\alpha+2} \psi_\sigma\bigl(\partial_\theta \phi + \partial_\theta^3 \phi\bigr)\right) + h^\prime_{\pm 1} e^{\pm i\theta} 
    = 0, 
    \quad t \in (0,t_\ast),\ \theta \in S^1.
\end{equation}
Moreover, we introduce the notation
\begin{equation*}
    \delta 
    =
    \|\phi\|_{L_\infty([\bar{t}/4,\bar{t}]\times S^1)} 
\end{equation*}
for the $L_\infty$-norm of $\phi$.
As a consequence of Theorem \ref{thm:power-law_decay} we have that
\begin{equation}\label{eq:delta}
    \delta
    \leq
    C \Lambda_\eps(\bar{t}),
    \quad 
    0 \leq \bar{t} \leq t_\ast.
\end{equation}
In the following we assume that $\delta > 0$, otherwise there is nothing to prove.
Since we are interested in the behaviour of solutions for large times, we rescale $\phi$ and the time variable $t$ via the transform
\begin{equation} \label{eq:transform_delta}
    \phi = \delta v 
    \quad \text{and} \quad
    t = \bar{t} \tau. 
\end{equation}
That is, we have
\begin{equation*}
    \phi(t,\theta) = \delta v\bigl(\tau/\bar{t},\theta),
    \quad 
    0\leq t \leq t_\ast,\ \theta \in S^1, 
    \quad \text{and} \quad
    \|v\|_{L_\infty([1/4,1]\times S^1)} = 1.
\end{equation*}
In these new variables the equation \eqref{eq:PDE_phi} reads
\begin{equation} \label{eq:PDE_rescaled}
    \partial_\tau v + \partial_\theta\bigl(|\bar{h} + h_{\pm 1}(\bar{t}\tau) e^{\pm i\theta} + \delta v|^{\alpha+2} \psi_\sigma\bigl(\partial_\theta v + \partial_\theta^3 v\bigr)\bigr) + \tfrac{1}{\delta} h^\prime_{\pm 1}(\bar{t} \tau) e^{\pm i\theta} = 0, \quad \tau \in (1/4,1),\ \theta \in S^1.
\end{equation}

\bigskip

\noindent\textbf{\textsc{Regularity estimates. }}
A direct application of Theorem \ref{thm:reg_estimate_general} to the rescaled version \eqref{eq:PDE_rescaled} of \eqref{eq:PDE_regularised} yields the following result.
Recall that we identify functions $v$ defined on $S^1$ by functions that are defined on the whole real line $\R$ that are periodic with period $2\pi$. In particular, we consider $v(\theta)$ being well-defined for $\theta \not\in [0,2\pi]$.


\begin{proposition}\label{prop:local_estimates}
Let $v$ be a weak solution to \eqref{eq:PDE_rescaled} with $v \in L_\infty\bigl([1/4,1]\times S^1\bigr)$ such that $\|v\|_{L_\infty([1/4,1]\times S^1)} \leq 1$. Let $[\theta_0-1/4,\theta_0+1/4] \subset S^1$. Then, $v$ satisfies the estimate
\begin{equation*}
    \int_{1/2}^{1} \int_{\theta_0-1/8}^{\theta_0+1/8} |\partial_\theta^3 v|^{\alpha+1}\, d\theta\, d\tau
    \leq
    C + \frac{C}{\delta} \int_{1/4}^1 |h_{\pm 1}^\prime(\bar{t}\tau)|\, d\tau
\end{equation*}
with a positive constant $C > 0$.
\end{proposition}


\begin{proof}
We apply Theorem \ref{thm:reg_estimate_general} with
\begin{equation*}
    S = \frac{1}{\delta} h_{\pm 1}^\prime(\bar{t}\tau) e^{\pm i \theta}
    \quad \text{and} \quad
    \Phi(v,\partial_\theta v,\partial_\theta^2 v,\partial_\theta^3 v)
    =
    \psi_\sigma\bigl(\partial_\theta v + \partial_\theta^3 v\bigr)
    =
    \bigl(|\partial_\theta v + \partial_\theta^3 v|^2 + \sigma^2\bigr)^\frac{\alpha-1}{2} (\partial_\theta v + \partial_\theta^3 v),
\end{equation*}
and use that, up to the time $t_\ast > 0$, we have that $\tfrac{\bar{h}_0}{2} \leq h(t,\theta) \leq 2\bar{h}_0$ such that the term $h^{\alpha+2}$ is always estimated from below and above by a positive constant that does not depend on $\sigma$.
\end{proof}


\begin{remark}
Theorem \ref{thm:reg_estimate_general} does not only apply to the regularised problem \eqref{eq:PDE_regularised}.
Although it is not needed to analyse the long-time asymptotics of solutions to the original problem \eqref{eq:PDE_alpha>1}, it might be of general interest to mention that the regularity estimate of Theorem \ref{thm:reg_estimate_general} also applies to the non-regularised doubly nonlinear problem
\begin{equation*}
    \begin{cases}
        \partial_t v 
        + 
        \partial_\theta\bigl(v^{\alpha+2} |\partial_\theta^3 v|^{\alpha-1} \partial_\theta^3 v\bigr)
        =
        0, \quad t > 0,\, \theta \in S^1,
        \\
        v(0,\theta) = v_0(\theta), \quad \theta \in S^1,
    \end{cases}
\end{equation*}
with periodic boundary conditions,
as long as $v$ remains bounded from below and above.
\end{remark}


\begin{remark}\label{rem:estimate_v_S^1}
Covering $S^1$ by finitely many intervals of the form $[\theta_0-1/4,\theta_0+1/4]$, we obtain the estimate
\begin{equation} \label{eq:est_D3_scaled}
    \int_{1/2}^1 \int_{S^1} |\partial_\theta^3 v|^{\alpha+1} d\theta\, d\tau \leq C + \frac{C}{\delta} \int_{1/4}^1 |h_{\pm 1}(\bar{t}\tau)|\, d\tau. 
\end{equation}
\end{remark}


We finally go back to the original variables $\phi$ and $t$ for which the estimate \eqref{eq:est_D3_scaled} reads as follows.


\begin{corollary} \label{cor:estimate_D3}
Let $\sigma \in (0,1)$ and let $h$ be a local weak solution to \eqref{eq:PDE_regularised}, corresponding to an $\eps_0 > 0$ and an initial value $h_0 \in C^\infty(S^1)$ satisfying 
\begin{equation*}
    \frac{1}{2\pi}\int_{S^1} h_0\, d\theta = \bar{h}_0
    \quad \text{and} \quad
    \|h_0 - \bar{h}_0\|_{H^1(S^1)} \leq \eps,
    \quad \eps \in (0,\eps_0).
\end{equation*}
Then we have
\begin{equation*}
    \int_{\bar{t}/2}^{\bar{t}} \int_{S^1} |\partial_\theta^3 \phi|^{\alpha+1}\, d\theta\, dt
    \leq
    C \|\phi\|_{L_\infty([\bar{t}/4,\bar{t}]\times S^1)}^2
    +
    \|\phi\|_{L_\infty([\bar{t}/4,\bar{t}]\times S^1)}^{\alpha} \int_{\bar{t}/4}^{\bar{t}} |h_{\pm 1}^\prime(t)|\, dt
\end{equation*}
for all $0 \leq \bar{t} \leq t_\ast$.
\end{corollary}


\begin{proof}
In Proposition \ref{prop:local_estimates} and Remark \ref{rem:estimate_v_S^1} we have proved that
\begin{equation*}
    \int_{1/2}^1 \int_{S^1} |\partial_\theta^3 v|^{\alpha+1} d\theta\, d\tau \leq C + \frac{C}{\delta} \int_{1/4}^1 |h_{\pm 1}(\bar{t}\tau)|\, dt.
\end{equation*}
Going back to the original variables, cf. \eqref{eq:transform_delta}, we find that
\begin{equation*}
    \int_{\bar{t}/2}^{\bar{t}} \int_{S^1} |\partial_\theta^3 \phi|^{\alpha+1} \frac{1}{\delta^{\alpha+1}}\, d\theta\, \frac{1}{\bar{t}} dt
    \leq
    C + \frac{C}{\delta} \int_{\bar{t}/4}^{\bar{t}} |h_{\pm 1}(t)|\, \frac{1}{\bar{t}} d\tau
\end{equation*}
and hence
\begin{equation}
    \begin{split}
        \int_{\bar{t}/2}^{\bar{t}} \int_{S^1} |\partial_\theta^3 \phi|^{\alpha+1}\, d\theta\, dt
        &\leq
        C \bar{t} \delta^{\alpha+1} + \frac{C}{\delta} \bar{t} \delta^{\alpha+1} \frac{1}{\bar{t}}   \int_{\bar{t}/4}^{\bar{t}} |h_{\pm 1}^\prime(t)|\, dt
        \\
        &\leq
        C \delta^2 + C \delta^{\alpha}  \int_{\bar{t}/4}^{\bar{t}} |h_{\pm 1}^\prime(t)|\, dt
        \\
        &\leq
        C \|\phi\|_{L_\infty([\bar{t}/4,\bar{t}]\times S^1)}^2 + C \|\phi\|_{L_\infty([\bar{t}/4,\bar{t}]\times S^1)}^\alpha \int_{\bar{t}/4}^{\bar{t}} |h_{\pm 1}^\prime(t)|\, dt,
    \end{split}
\end{equation}
where we use that  $\|\phi\|_{L_\infty([\bar{t}/4,\bar{t}]\times S^1)} = \delta \leq C \Lambda_\eps(\bar{t})$ and hence $\bar{t} \leq C \delta^{-(\alpha-1)}$.
\end{proof}


As a direct consequence of Corollary \ref{cor:estimate_D3} and the decay estimate for $\|\phi\|_{L_\infty([\bar{t}/4,\bar{t}]\times S^1)}$ in terms of $\Lambda_\eps(\bar{t})$ we obtain the following.


\begin{corollary} \label{cor:estimate_D3_time}
Let $h$ be a local weak solution to \eqref{eq:PDE_regularised}, corresponding to some $\eps_0 > 0$ and an initial value $h_0 \in C^\infty(S^1)$ satisfying 
\begin{equation*}
    \frac{1}{2\pi}\int_{S^1} h_0\, d\theta = \bar{h}_0
    \quad \text{and} \quad
    \|h_0 - \bar{h}_0\|_{H^1(S^1)} \leq \eps,
    \quad \eps \in (0,\eps_0).
\end{equation*}
Then, for all $0 \leq \bar{t} \leq t_\ast$, we have
\begin{equation*}
    \int_{\bar{t}/2}^{\bar{t}} \int_{S^1} |\partial_\theta^3 \phi|^{\alpha+1}\, d\theta\, dt
    \leq
    C \left( 
    \bigl(\Lambda_\eps(\bar{t})\bigr)^2
    +
    \bigl(\Lambda_\eps(\bar{t})\bigr)^{\alpha}
    \int_{\bar{t}/4}^{\bar{t}} |h_{\pm 1}^\prime(t)|\, dt
    \right).
\end{equation*}
\end{corollary}


By an iterative argument, the estimates obtained in Corollary \ref{cor:estimate_D3}, respectively Corollary \ref{cor:estimate_D3_time}, may be improved (for `small' third derivatives) to an estimate with the power $\alpha$ instead of $\alpha+1$ on the left-hand side. This in turn leads to an estimate for the change of the Fourier coefficients $h_{\pm 1}$ in time. In order to simplify the proof of this result, we first state some auxiliary technical estimate for powers of the function $\Lambda_\eps$ which is needed frequently.


\begin{lemma} \label{lem:estimate_Lambda}
Given $\eps \in (0,1)$, for any constant $C \leq 1$ in the definition of $\Lambda_\eps$, cf. \eqref{eq:def_Lambda}, we have the inequality
\begin{equation*}
    \bar{t}
    \leq
    \bar{C}
    \bigl(\Lambda_\eps(\bar{t})\bigr)^{1-\alpha},
    \quad
    0 \leq \bar{t} \leq t_\ast,
\end{equation*}
for a positive constant $\bar{C} > 0$.
\end{lemma}


\begin{proof}
We may estimate $0 \leq \bar{t} \leq t_\ast$ against $\Lambda_\eps(\bar{t})$ as follows. There exists a constant $C > 0$ such that
\begin{equation*}
    \bar{t}
    =
    \eps^{1-\alpha}
    \left(\eps^{\alpha-1} \bar{t}\right) 
    \leq
    \bar{C} \left(\eps^{-1} 
    \left(1 + C \eps^{\alpha-1} \bar{t}\right)^\frac{1}{\alpha-1}
    \right)^{\alpha-1}
    =
    \bar{C}
    \bigl(\Lambda_\eps(\bar{t})\bigr)^{1-\alpha}.
\end{equation*}
\end{proof}


\begin{lemma}\label{lem:estimate_flux_dissipation}
Let $\sigma \in (0,1)$ be fixed and let $h$ be a local weak solution to \eqref{eq:PDE_regularised}, corresponding to some $\eps_0 > 0$ and an initial value $h_0 \in C^\infty(S^1)$ satisfying 
\begin{equation*}
    \frac{1}{2\pi}\int_{S^1} h_0\, d\theta = \bar{h}_0
    \quad \text{and} \quad
    \|h_0 - \bar{h}_0\|_{H^1(S^1)} \leq \eps,
    \quad \eps \in (0,\eps_0).
\end{equation*}
Then we have
\begin{equation*}
    \int_{\bar{t}/4}^{\bar{t}} \int_{S^1} h^{\alpha+2} |\psi_\sigma(\partial_\theta h + \partial_\theta^3 h)|\, d\theta\, dt
    \leq
    C (\bar{t})^\frac{1}{\alpha+1} 
    \left(\bar{t}\, \sigma^{\alpha+1} + D_{t_\ast}^\sigma[h]\right)^\frac{\alpha}{\alpha+1},
    \quad
    0 \leq \bar{t} \leq t_\ast.
\end{equation*}
\end{lemma}


\begin{proof}
Applying H\"older's inequality with exponents $p=\alpha+1$ and $q=(\alpha+1)/\alpha$ to the left-hand side of the desired inequality yields
\begin{equation*}
    \begin{split}
        &
        \int_{\bar{t}/4}^{\bar{t}} \int_{S^1} h^{\alpha+2} |\psi_\sigma(\partial_\theta h + \partial_\theta^3 h)|\, d\theta\, dt
        \\
        &\quad
        \leq
        C\left(\int_{\bar{t}/4}^{\bar{t}} \int_{S^1} 1\, d\theta\, dt\right)^\frac{1}{\alpha+1}
        \left(\int_{\bar{t}/4}^{\bar{t}} \int_{S^1} 
        \bigl(|\partial_\theta h + \partial_\theta^3 h|^2 + \sigma^2\bigr)^{\frac{\alpha-1}{2}\frac{\alpha+1}{\alpha}} |\partial_\theta h + \partial_\theta^3 h|^\frac{\alpha+1}{\alpha}\, d\theta\, dt\right)^\frac{\alpha}{\alpha+1}
        \\
        &\quad 
        \leq 
        C (\bar{t})^\frac{1}{\alpha+1}
        \left(\int_{\bar{t}/4}^{\bar{t}} \int_{S^1} 
        \bigl(|\partial_\theta h + \partial_\theta^3 h|^2 + \sigma^2\bigr)^\frac{\alpha-1}{2}
        \bigl(|\partial_\theta h + \partial_\theta^3 h|^2 + \sigma^2\bigr)^{\frac{1}{\alpha}\frac{\alpha-1}{2}}|\partial_\theta h + \partial_\theta^3 h|^\frac{\alpha+1}{\alpha}\, d\theta\, dt\right)^\frac{\alpha}{\alpha+1}.
    \end{split}
\end{equation*}
On the set $I_{\leq \sigma}=\left\{(t,\theta) \in [\bar{t}/4,\bar{t}]\times S^1;\ |\partial_\theta h + \partial_\theta^3 h| \leq \sigma\right\}$ we obtain
\begin{equation*}
    \int_{I_{\leq \sigma}} \bigl(|\partial_\theta h + \partial_\theta^3 h|^2 + \sigma^2\bigr)^\frac{\alpha-1}{2}
    \bigl(|\partial_\theta h + \partial_\theta^3 h|^2 + \sigma^2\bigr)^{\frac{1}{\alpha}\frac{\alpha-1}{2}}|\partial_\theta h + \partial_\theta^3 h|^\frac{\alpha+1}{\alpha}\, d\theta\, dt
    \leq
    C \bar{t} \sigma^{\alpha+1},
    \quad
    0 \leq \bar{t} \leq t_\ast,
\end{equation*}
and on $I_{> \sigma}=\left\{(t,\theta) \in [\bar{t}/4,\bar{t}]\times S^1;\ |\partial_\theta h + \partial_\theta^3 h| > \sigma\right\}$
\begin{equation*}
    \begin{split}
        &
        \int_{I_{> \sigma}} \bigl(|\partial_\theta h + \partial_\theta^3 h|^2 + \sigma^2\bigr)^\frac{\alpha-1}{2}
        \bigl(|\partial_\theta h + \partial_\theta^3 h|^2 + \sigma^2\bigr)^{\frac{1}{\alpha}\frac{\alpha-1}{2}}|\partial_\theta h + \partial_\theta^3 h|^\frac{\alpha+1}{\alpha}\, d\theta\, dt
        \\
        &\quad
        \leq
        \int_{I_{> \sigma}} \bigl(|\partial_\theta h + \partial_\theta^3 h|^2 + \sigma^2\bigr)^\frac{\alpha-1}{2}
        |\partial_\theta h + \partial_\theta^3 h|^2\, d\theta\, dt
        \\
        &\quad 
        \leq
        D_{t_\ast}^\sigma[h],
        \quad 0 \leq \bar{t} \leq t_\ast.
    \end{split}
\end{equation*}
This completes the proof.
\end{proof}


We are finally in a position to prove an estimate for the $L_1([\bar{t}/2,\bar{t}];L_\alpha(S^1))$-norm of $\partial_\theta^3 \phi$ in terms of $\Lambda_\eps(\bar{t})$. Note that at this point the smallness of the regularisation parameter $\sigma$ becomes crucial.


\begin{theorem}[Interior regularity estimate II] \label{thm:improved_estimate_alpha}
Let $h$ be a local weak solution to \eqref{eq:PDE_regularised} for a fixed $\sigma >0$, corresponding to some $\eps_0 > 0$ and an initial value $h_0 \in C^\infty(S^1)$ satisfying 
\begin{equation*}
    \frac{1}{2\pi}\int_{S^1} h_0\, d\theta = \bar{h}_0
    \quad \text{and} \quad
    \|h_0 - \bar{h}_0\|_{H^1(S^1)} \leq \eps,
    \quad \eps \in (0,\eps_0).
\end{equation*}
Then the following holds true. For every given finite  $T > 0$ there exists $\bar{\sigma} = \bar{\sigma}(\eps_0,T)$ such that, if $0 < \sigma \leq \bar{\sigma}$, then
\begin{equation*}
    \int_{\bar{t}/2}^{\bar{t}} \int_{S^1} |\partial_\theta^3 \phi|^{\alpha}\, d\theta\, dt
    \leq
    C 
    \Lambda_\eps(\bar{t}),
    \quad 
    0\leq \bar{t} \leq \min\{t_\ast,T\}.
\end{equation*}
\end{theorem}


The proof of this theorem is mainly based on an iterative scheme the goal of which is to get in each iteration step an improved estimate for the Fourier coefficients $h_{\pm 1}$. The underlying inequality is based on estimates of the dissipation in terms of powers of $\Lambda_\eps$ and $h_{\pm 1}$.


\begin{proof}
\noindent (i) In the first step we derive an inequality for the dissipation in terms of powers of $\Lambda_\eps$ and $h^\prime_{\pm 1}$.
Similar as in the proof of Lemma \ref{lem:estimate_flux_dissipation} we find that
\begin{equation*}
    \begin{split}
        &
        \int_{\bar{t}/2}^{\bar{t}} \int_{S^1} h^{\alpha+2} \bigl(|\partial_\theta \phi + \partial_\theta^3 \phi|^2 + \sigma^2\bigr)^\frac{\alpha-1}{2} |\partial_\theta \phi + \partial_\theta^3 \phi|^2\, d\theta\, dt
        \\
        &\quad 
        \leq
        C \int_{\bar{t}/2}^{\bar{t}} \int_{S^1} |\partial_\theta \phi + \partial_\theta^3 \phi|^{\alpha+1}\, d\theta\, dt
        +
        \bar{C}\,\bar{t}\, \sigma^{\alpha+1},
        \quad
        0 \leq \bar{t} \leq t_\ast.
    \end{split}
\end{equation*}
We estimate the integral on the right-hand side. A standard interpolation argument (cf. Proposition \ref{prop:weighted_Sobolev} for a similar argument) yields 
\begin{equation*}
    \int_{S^1} |\partial_\theta \phi(t)|^{\alpha+1}\, d\theta
    \leq
    C \int_{S^1} |\partial_\theta^3 \phi(t)|^{\alpha+1}\, d\theta
    +
    C \|\phi\|_{L_\infty([\bar{t}/4,\bar{t}]\times S^1)}^{\alpha+1}
\end{equation*}
for almost every $t \in [\bar{t}/4,\bar{t}]$ and a generic constant $C > 0$. Consequently, we find that
\begin{equation} \label{eq:iteration_0}
    \begin{split}
        \int_{\bar{t}/2}^{\bar{t}} \int_{S^1} |\partial_\theta \phi + \partial_\theta^3 \phi|^{\alpha+1}\, d\theta\, dt
        &\leq
        C \left(\int_{\bar{t}/2}^{\bar{t}} \int_{S^1} |\partial_\theta \phi|^{\alpha+1}\, d\theta\, dt 
        + 
        \int_{\bar{t}/2}^{\bar{t}} \int_{S^1}
        |\partial_\theta^3 \phi|^{\alpha+1}\, d\theta\, dt\right)
        \\
        &\leq
        C
        \left(\int_{\bar{t}/2}^{\bar{t}} \int_{S^1} |\partial_\theta^3 \phi|^{\alpha+1}\, d\theta\, dt 
        + 
        \bar{t}\, \|\phi\|_{L_\infty([\bar{t}/4,\bar{t}]\times S^1)}^{\alpha+1}\right).    
    \end{split}
\end{equation}
In view of Corollary \ref{cor:estimate_D3_time} and Theorem \ref{thm:power-law_decay} we deduce from \eqref{eq:iteration_0} that
\begin{equation} \label{eq:aux_iteration}
    \int_{\bar{t}/2}^{\bar{t}} \int_{S^1} |\partial_\theta \phi + \partial_\theta^3 \phi|^{\alpha+1}\, d\theta\, dt
    \leq
    C 
    \left(
    \bigl(\Lambda_\eps(\bar{t})\bigr)^2
    +
    \bigl( \Lambda_\eps(\bar{t})\bigr)^{\alpha}
    \int_{\bar{t}/4}^{\bar{t}} |h_{\pm 1}^\prime(t)|\, dt
    +
    \bar{t}\, \bigl( \Lambda_\eps(\bar{t})\bigr)^{\alpha+1}
    \right)
\end{equation}
for all $0\leq \bar{t} \leq t_\ast$.
Next, Lemma \ref{lem:estimate_Lambda} implies that
\begin{equation*}
    \bar{t} \bigl(\Lambda_\eps(\bar{t})\bigr)^{\alpha+1}
    \leq 
    C 
    \bigl(\Lambda_\eps(\bar{t})\bigr)^2, 
    \quad 
    0\leq \bar{t} \leq t_\ast,
\end{equation*}
which allows us to absorb the last term on the right-hand side of \eqref{eq:aux_iteration} in the first one. This leads to the estimate
\begin{equation*} 
    \int_{\bar{t}/2}^{\bar{t}} \int_{S^1} |\partial_\theta \phi + \partial_\theta^3 \phi|^{\alpha+1}\, d\theta\, dt
    \leq
    C 
    \left(
    \bigl(\Lambda_\eps(\bar{t})\bigr)^2
    +
    \bigl( \Lambda_\eps(\bar{t})\bigr)^{\alpha}
    \int_{\bar{t}/4}^{\bar{t}} |h_{\pm 1}^\prime(t)|\, dt\right), 
    \quad 
    0\leq \bar{t} \leq t_\ast.
\end{equation*}
and consequently, we obtain the estimate
\begin{equation} \label{eq:iteration_dissipation}
    \begin{split}
        &
        \int_{\bar{t}/2}^{\bar{t}} \int_{S^1} h^{\alpha+2} \bigl(|\partial_\theta \phi + \partial_\theta^3 \phi|^2 + \sigma^2\bigr)^\frac{\alpha-1}{2} |\partial_\theta \phi + \partial_\theta^3 \phi|^2\, d\theta\, dt
        \\
        &\quad
        \leq
        C 
        \left(
        \bigl(\Lambda_\eps(\bar{t})\bigr)^2
        +
        \bigl( \Lambda_\eps(\bar{t})\bigr)^{\alpha}
        \int_{\bar{t}/4}^{\bar{t}} |h_{\pm 1}^\prime(t)|\, dt\right)
        +
        \bar{C}\, \bar{t}\, \sigma^{\alpha+1}
    \end{split}
\end{equation}
for all $0\leq \bar{t} \leq t_\ast$.
\smallskip

\noindent (ii) In the second step, we use H\"older's inequality in order to derive an estimate for the integral $\int_{\bar{t}/4}^{\bar{t}} |h^\prime_{\pm 1}(t)|\, dt$  in terms of powers of $\bar{t}$ and powers of the dissipation integral $\int_{\bar{t}/2}^{\bar{t}} \int_{S^1} \bigl(|\partial_\theta \phi + \partial_\theta^3 \phi|^2 + \sigma^2\bigr)^\frac{\alpha-1}{2} |\partial_\theta \phi + \partial_\theta^3 \phi|^2\, d\theta\, dt$. Then, using the estimate \eqref{eq:iteration_dissipation} for the dissipation, derived in step (i), we obtain the  basic inequality for the iteration scheme.

More precisely, by the definition of the Fourier coefficients $h_{\pm 1}$ and recalling that $\tfrac{1}{2} \bar{h}_0 \leq h(t,\theta) \leq 2 \bar{h}_0,\ 0 \leq t \leq t_\ast,\, \theta \in S^1$, we deduce from Lemma \ref{lem:estimate_flux_dissipation} that there exists a constant $C > 0$
\begin{equation*} 
    \begin{split}
        \int_{\bar{t}/4}^{\bar{t}} |h_{\pm 1}^\prime(t)|\, dt
        &\leq
        C \int_{\bar{t}/4}^{\bar{t}} \int_{S^1} h^{\alpha+2} |\psi_\sigma(\partial_\theta \phi + \partial_\theta^3 \phi)|\, d\theta\,dt
        \\
        &\leq
        C (\bar{t})^\frac{1}{\alpha+1} 
        \left(\bar{t}\, \sigma^{\alpha+1} + D_{t_\ast}^\sigma[h]\right)^\frac{\alpha}{\alpha+1}
    \end{split}
\end{equation*}
for all $0 \leq \bar{t} \leq t_\ast$. Using Lemma \ref{lem:energy_dissipation} and defining $\sigma_0(\eps_0,T)=T^{-\frac{\alpha}{(\alpha+1)^2}}$, we obtain for all $0 < \bar{t} \leq \min\{T,t_\ast\}$ and $\sigma \leq \sigma_0$ that $\bar{t} \sigma^{\alpha+1} \leq (\bar{t})^\frac{1}{\alpha+1}$. Thus, for all $\sigma \leq \sigma_0$ there exists a constant $C_0 > 0$ such that
\begin{equation}\label{eq:step_0_idea}
    \int_{\bar{t}/4}^{\bar{t}} |h_{\pm 1}^\prime(t)|\, dt
    \leq
    C_0\, (\bar{t})^\frac{1}{\alpha+1},
    \quad
    0 \leq \bar{t} \leq \min\{T,t_\ast\}.
\end{equation}
Next, we split the time interval of the dissipation term as follows
\begin{equation*} 
    \begin{split}
        &
        \int_{\bar{t}/4}^{\bar{t}} \int_{S^1} \bigl(|\partial_\theta \phi + \partial_\theta^3 \phi|^2 + \sigma^2\bigr)^\frac{\alpha-1}{2} |\partial_\theta \phi + \partial_\theta^3 \phi|^2\, d\theta\, dt
        \\
        &
        =
        \int_{\bar{t}/2}^{\bar{t}} \int_{S^1} \bigl(|\partial_\theta \phi + \partial_\theta^3 \phi|^2 + \sigma^2\bigr)^\frac{\alpha-1}{2} |\partial_\theta \phi + \partial_\theta^3 \phi|^2\, d\theta\, dt 
        \\
        &\quad
        + 
        \int_{\bar{t}/4}^{\bar{t}/2} \int_{S^1} \bigl(|\partial_\theta \phi + \partial_\theta^3 \phi|^2 + \sigma^2\bigr)^\frac{\alpha-1}{2} |\partial_\theta \phi + \partial_\theta^3 \phi|^2\, d\theta\, dt,
    \end{split}
\end{equation*}
where $0\leq \bar{t} \leq t_\ast$.
Using \eqref{eq:iteration_dissipation}, we estimate both terms on the right-hand side similarly. Indeed, for the first integral we obtain
\begin{equation*}
    \int_{\bar{t}/2}^{\bar{t}} \int_{S^1} \bigl(|\partial_\theta \phi + \partial_\theta^3 \phi|^2 + \sigma^2\bigr)^\frac{\alpha-1}{2} |\partial_\theta \phi + \partial_\theta^3 \phi|^2\, d\theta\, dt
    \leq
    C 
    \left(
    \bigl(\Lambda_\eps(\bar{t})\bigr)^2
    +
    \bigl( \Lambda_\eps(\bar{t})\bigr)^{\alpha}
    \int_{\bar{t}/4}^{\bar{t}} |h_{\pm 1}^\prime(t)|\, dt
    \right)
    +
    \bar{C}\, \bar{t}\, \sigma^{\alpha+1}.
\end{equation*}
Now we define $\sigma_1(\eps_0,T) = \Lambda_\eps(T)$. Then, using Lemma \ref{lem:estimate_Lambda} and the monotonicity $\Lambda_\eps(T) \leq \Lambda_\eps(\bar{t})$, for all $0 < \bar{t} \leq \min\{T,t_\ast\}$ and $\sigma \leq \sigma_1$ we have $\bar{t} \sigma^{\alpha+1} \leq C \bigl(\Lambda_\eps(\bar{t})\bigr)^2$, yields
\begin{equation*}
    \int_{\bar{t}/2}^{\bar{t}} \int_{S^1} \bigl(|\partial_\theta \phi + \partial_\theta^3 \phi|^2 + \sigma^2\bigr)^\frac{\alpha-1}{2} |\partial_\theta \phi + \partial_\theta^3 \phi|^2\, d\theta\, dt
    \leq
    C 
    \left(
    \bigl(\Lambda_\eps(\bar{t})\bigr)^2
    +
    \bigl( \Lambda_\eps(\bar{t})\bigr)^{\alpha}
    \int_{\bar{t}/4}^{\bar{t}} |h_{\pm 1}^\prime(t)|\, dt
    \right).
\end{equation*}
In the same way we deduce for all $0 < \bar{t} \leq \min\{T,t_\ast\}$ the estimate
\begin{equation*}
    \int_{\bar{t}/4}^{\bar{t}/2} \int_{S^1} \bigl(|\partial_\theta \phi + \partial_\theta^3 \phi|^2 + \sigma^2\bigr)^\frac{\alpha-1}{2} |\partial_\theta \phi + \partial_\theta^3 \phi|^2\, d\theta\, dt
    \leq
    C 
    \left(
    \bigl(\Lambda_\eps(\bar{t})\bigr)^2
    +
    \bigl(\Lambda_\eps(\bar{t})\bigr)^\alpha
    \int_{\bar{t}/8}^{\bar{t}/4} |h_{\pm 1}^\prime(t)|\, dt\right)
\end{equation*}
for the second integral.
Moreover, proceeding as in the proof of Lemma \ref{lem:estimate_flux_dissipation} and using that $(a+b)^\theta \leq C(a^\theta + b^\theta)$ for $a,b \in \R_+$ and $\theta \in (0,1)$, we find that
\begin{equation} \label{eq:iteration_general}
    \begin{split}
        \int_{\bar{t}/4}^{\bar{t}} |h_{\pm 1}^\prime(t)|\, dt
        &\leq
        C (\bar{t})^\frac{1}{\alpha+1}
        \left(\int_{\bar{t}/4}^{\bar{t}} \int_{S^1}
        \bigl(|\partial_\theta \phi + \partial_\theta^3 \phi|^2 + \sigma^2\bigr)^\frac{\alpha-1}{2} |\partial_\theta \phi + \partial_\theta^3 \phi|^2\, d\theta\, dt 
        + 
        \bar{t}\, \sigma^{\alpha+1}
        \right)^\frac{\alpha}{\alpha+1}
        \\
        &\leq
        C (\bar{t})^\frac{1}{\alpha+1}
        \left(
        \bigl(\Lambda_\eps(\bar{t})\bigr)^2
        +
        \bigl(\Lambda_\eps(\bar{t})\bigr)^\alpha
        \int_{\bar{t}/8}^{\bar{t}} |h_{\pm 1}^\prime(t)|\, dt
        \right)^\frac{\alpha}{\alpha+1}
        \\
        &\leq
        C (\bar{t})^\frac{1}{\alpha+1}
        \left(
        \bigl(\Lambda_\eps(\bar{t})\bigr)^\frac{2\alpha}{\alpha+1}
        +
        \bigl(\Lambda_\eps(\bar{t})\bigr)^\frac{\alpha^2}{\alpha+1} \left(\int_{\bar{t}/8}^{\bar{t}} |h_{\pm 1}^\prime(t)|\, dt\right)^\frac{\alpha}{\alpha+1}
        \right)
    \end{split}
\end{equation}
for all $0 \leq \bar{t} \leq \min\{T,t_\ast\}$. We choose $\bar{\sigma} = \min\{\sigma_0,\sigma_1\}$.

\smallskip
\noindent(iii) Iteration. Based on \eqref{eq:iteration_general}, we set up an iteration scheme as follows. First, we rewrite \eqref{eq:step_0_idea} as
\begin{equation*}
    \int_{\bar{t}/4}^{\bar{t}} |h_{\pm 1}^\prime(t)|\, dt
    \leq
    C_0\ (\bar{t})^\frac{1}{\alpha+1} \bigl(\Lambda_\eps(\bar{t})\bigr)^{\beta_0}
    \quad
    \text{with} 
    \quad 
    \beta_0 = 0,
    \quad 0 \leq \bar{t} \leq \min\{T,t_\ast\}.
\end{equation*}
Moreover, using Lemma \ref{lem:estimate_Lambda}, we find that
\begin{equation} \label{eq:iteration_scheme}
    \begin{dcases}
        \int_{\bar{t}/4}^{\bar{t}} |h_{\pm 1}^\prime(t)|\, dt
        \leq
        C_{n+1} (\bar{t})^\frac{1}{\alpha+1} \bigl(\Lambda_\eps(\bar{t})\bigr)^{\beta_{n+1}},
        \quad 
        0 \leq \bar{t} \leq \min\{T,t_\ast\},
        & 
        \\
        \beta_0 = 0,\ 
        \beta_{n+1} = \min\left\{\frac{2\alpha}{\alpha+1},\frac{\alpha(\alpha^2 + 1)}{(\alpha+1)^2} + \beta_n \frac{\alpha}{\alpha+1}\right\},\ n \in \N, &
    \end{dcases}
\end{equation}
where the sequence $(C_n)_{n \in \N_0}$ of constants is increasing. However, we can prove that the iteration scheme \eqref{eq:iteration_scheme} yields $\beta_n = \frac{2\alpha}{\alpha+1}$ after a finite number of steps. To see this,
observe first that, if $\beta_n = \frac{2\alpha}{\alpha+1}$, then we have
\begin{equation*}
    \frac{\alpha(\alpha^2 + 1)}{(\alpha+1)^2} + \beta_n \frac{\alpha}{\alpha+1} 
    =
    \frac{\alpha(\alpha^2 + 1)}{(\alpha+1)^2} + \frac{2\alpha^2}{(\alpha+1)^2}
    =
    \alpha
    > \frac{2\alpha}{\alpha+1} 
    \quad
    \text{and thus} 
    \quad
    \beta_{n+1} = \frac{2\alpha}{\alpha+1}.
\end{equation*}
Note that we used the assumption $\alpha > 1$ here.
Moreover, the linear problem
\begin{equation} \label{eq:linear_problem_beta}
    \begin{cases}
        \beta_{n+1} = \frac{\alpha(\alpha^2 + 1)}{(\alpha+1)^2} + \beta_n \frac{\alpha}{\alpha+1}, & 
        \\
        \beta_0 = 0
    \end{cases}
\end{equation}
has the unique solution
\begin{equation*}
    \beta_n 
    =
    \frac{\alpha(\alpha^2 + 1)}{\alpha+1} \left(1 - \left(\frac{\alpha}{\alpha+1}\right)^n\right), 
    \quad n \in \N,
\end{equation*}
satisfying
\begin{equation*}
    \beta_n 
    \longrightarrow 
    \beta_\infty =
    \frac{\alpha(\alpha^2 + 1)}{\alpha+1}
    \quad
    \text{as }
    n \to \infty
    \quad \text{and} \quad \beta_\infty > \frac{2\alpha}{\alpha+1}, 
    \quad \text{for all } \alpha > 1.
\end{equation*}
Consequently, the iteration scheme \eqref{eq:iteration_scheme} leads us to the estimates
\begin{equation*}
    \int_{\bar{t}/2}^{\bar{t}} |h^\prime_{\pm 1}(t)|\, dt
    \leq
    C \int_{\bar{t}/2}^{\bar{t}} \int_{S^1} |\partial_\theta \phi + \partial_\theta^3 \phi|^\alpha\, d\theta\, dt 
    \leq
    C (\bar{t})^\frac{1}{\alpha+1} \bigl(\Lambda_\eps(\bar{t})\bigr)^\frac{2\alpha}{\alpha+1},
    \quad
    0\leq \bar{t} \leq \min\{T,t_\ast\}.
\end{equation*}
Using again Lemma \ref{lem:estimate_Lambda} and  $\frac{1-\alpha}{\alpha+1}+\frac{2\alpha}{\alpha+1}=1$, we finally arrive at the inequality
\begin{equation*}
    \int_{\bar{t}/2}^{\bar{t}} |h^\prime_{\pm 1}(t)|\, dt
    \leq
    C \int_{\bar{t}/2}^{\bar{t}} \int_{S^1} |\partial_\theta \phi + \partial_\theta^3 \phi|^\alpha\, d\theta\, dt 
    \leq
    C  \Lambda_\eps(\bar{t}),
    \quad 0\leq \bar{t} \leq \min\{T,t_\ast\}.
\end{equation*}
This completes the proof.
\end{proof}


As an immediate consequence of the previous proof we obtain the following estimate for the difference of the Fourier coefficients $h_{\pm 1}$ at two instants in time of order $\bar{t}$.


\begin{corollary} \label{cor:estimate_Fourier_time}
Let $h$ be a local weak solution to \eqref{eq:PDE_regularised} for a fixed $\sigma >0$, corresponding to some $\eps_0 > 0$ and an initial value $h_0 \in C^\infty(S^1)$ satisfying 
\begin{equation*}
    \frac{1}{2\pi}\int_{S^1} h_0\, d\theta = \bar{h}_0
    \quad \text{and} \quad
    \|h_0 - \bar{h}_0\|_{H^1(S^1)} \leq \eps,
    \quad \eps \in (0,\eps_0).
\end{equation*}
Then the following holds true. For all $T > 0$ there exists $\bar{\sigma} = \bar{\sigma}(\eps_0,T)$ such that, if $0 < \sigma \leq \bar{\sigma}$, then
the Fourier coefficients $h_{\pm 1}$ satisfy the inequality
\begin{equation*}
    \int_{\bar{t}/2}^{\bar{t}} |h^\prime_{\pm 1}(t)|\, dt 
    \leq 
    C \Lambda_\eps(\bar{t}),
    \quad 0 \leq \bar{t} \leq \min\{t_\ast,T\}.
\end{equation*}
\end{corollary}


From the previous results we deduce, for $0 < \sigma \leq \bar{\sigma}$ small enough, the following estimate for the Fourier coefficients $h_{\pm 1}$ as long as the solution is bounded away from zero and bounded from above in the sense that $\tfrac{1}{2} \bar{h}_0 \leq h(t,\theta) \leq 2\bar{h}_0$.


\begin{proposition} \label{prop:estimate_Fourier}
Let $h$ be a local weak solution to \eqref{eq:PDE_regularised} for a fixed $\sigma >0$, corresponding to some $\eps_0 > 0$ and an initial value $h_0 \in C^\infty(S^1)$ satisfying 
\begin{equation*}
    \frac{1}{2\pi}\int_{S^1} h_0\, d\theta = \bar{h}_0
    \quad \text{and} \quad
    \|h_0 - \bar{h}_0\|_{H^1(S^1)} \leq \eps,
    \quad \eps \in (0,\eps_0).
\end{equation*}
Then the following holds true. For each fixed $0 < T < \infty$ there exists $\bar{\sigma} = \bar{\sigma}(\eps_0,T)$ such that, if $0 < \sigma \leq \bar{\sigma}$, then
the Fourier coefficients $h_{\pm 1}$ satisfy the estimate 
\begin{equation*}
    \int_0^{\min\{t_\ast,T\}} |h^\prime_{\pm 1}(t)|\, dt 
    \leq
    C \eps \left(1 + \log\bigl(\eps^{1-\alpha}\bigr)\right).
\end{equation*}
\end{proposition}


\begin{proof}
In view of Corollary \ref{cor:estimate_Fourier_time}, we split the time integral into two integrals
\begin{equation} \label{eq:time_dyadics_1}
    \int_0^{\min\{t_\ast,T\}} |h^\prime_{\pm 1}(t)|\, dt
    =
    \int_0^{\eps^{1-\alpha}}
    |h^\prime_{\pm 1}(t)|\, dt
    +
    \int_{\eps^{1-\alpha}}^{\min\{t_\ast,T\}} |h^\prime_{\pm 1}(t)|\, dt
\end{equation}
and estimate each of them separately. In order to estimate the first integral, we define $I_n = \bigl[2^{-(n+1)} \eps^{1-\alpha},2^{-n} \eps^{1-\alpha}\bigr)$ for $n \in \N_0,\ n \leq N_\eps$, where $N_\eps \in \N$ is such that $2^{-N_\eps} \eps^{1-\alpha} \simeq 1$ is of order one, i.e. 
\begin{equation*}
    N_\eps = \frac{\log\bigl(\eps^{1-\alpha}\bigr)}{\log(2)} \leq C \log\bigl(\eps^{1-\alpha}\bigr).
\end{equation*}
Then, using Corollary \ref{cor:estimate_Fourier_time} for times $t$ smaller that $\eps^{1-\alpha}$, we obtain
\begin{equation} \label{eq:time_dyadics_2}
    \int_0^{\eps^{1-\alpha}}
    |h^\prime_{\pm 1}(t)|\, dt
    \leq
    \sum_{n=0}^{N_\eps} \int_{I_n} |h^\prime_{\pm 1}(t)|\, dt
    \leq
    \sum_{n=0}^{N_\eps} C \eps
    =
    C N_\eps \eps 
    \leq
    C \log\bigl(\eps^{1-\alpha}\bigr) \eps. 
\end{equation}
For the second integral, we define $I_n = \bigl[2^n \eps^{1-\alpha},2^{n+1} \eps^{1-\alpha}\bigr),\ n \in \N_0$, and apply Corollary \ref{cor:estimate_Fourier_time} for times $t$ larger than $\eps^{1-\alpha}$ in order to deduce that
\begin{equation} \label{eq:time_dyadics_3}
    \int_{\eps^{-\frac{\alpha-1}{2}}}^{\min\{t_\ast,T\}} |h^\prime_{\pm 1}(t)|\, dt
    \leq
    \sum_{n=0}^\infty \int_{I_n} |h^\prime_{\pm 1}(t)|\, dt
    \leq
    \sum_{n=0}^\infty C \bigl(2^n \eps^{1-\alpha}\bigr)^{- \frac{1}{\alpha-1}}
    =
    C \eps \sum_{n=0}^\infty 2^{-\frac{n}{\alpha-1}}
    \leq 
    C \eps,
\end{equation}
since the sum $\sum_{n=0}^\infty 2^{-\frac{n}{\alpha-1}}$ is finite. Combining \eqref{eq:time_dyadics_2} and \eqref{eq:time_dyadics_3}, we find that the integral in \eqref{eq:time_dyadics_1} is bounded in the sense that
\begin{equation*}
    \int_0^{\min\{t_\ast,T\}} |h^\prime_{\pm 1}(t)|\, dt 
    \leq
    C \eps \left(1 + \log\bigl(\eps^{1-\alpha}\bigr)\right).
\end{equation*}
This completes the proof.
\end{proof}


It is worthwhile to mention that Proposition \ref{prop:estimate_Fourier} provides an estimate for the difference of the values of the first Fourier coefficients $h_{\pm 1}$ at times $t\leq t_\ast$ and their value at time $t=0$ for $0 < \sigma \leq \bar{\sigma}$ small enough. This in turn allows us to control the position of the circle's center. More precisely, Proposition \ref{prop:estimate_Fourier} guarantees that the center of the circle does not move `too far' away from the common center of the cylinders if it is not `too far' away from it initially.


As a last step of this section we show that all the estimates proved above are valid up to any fixed finite time $0 < T < \infty$, i.e. we have $\min\{t_\ast,T\}=T$, if we only choose $\bar{\sigma} > 0$ small enough.


\begin{lemma}\label{lem:T_larger_t_ast}
There exists an $\eps_0 > 0$, independent of $\sigma$, such that for all $T > 0$ and all $\eps \in (0,\eps_0)$ there exists a $\bar{\sigma} = \bar{\sigma}(T;\eps)$ such that if $0 < \sigma \leq \bar{\sigma}$, then $t_\ast = t_\ast(\eps,\sigma) \geq T$.
\end{lemma}


\begin{proof}
In view of Theorem \ref{thm:existence_regularised} we know that there exists a positive time $T > 0$ and a positive weak solution 
\begin{equation*}
    h \in 
    L_{\alpha+1}\bigl((0,T);W^3_{\alpha+1}(S^1)\bigr) \cap C\bigl([0,T];H^1(S^1)\bigr),
    \quad
    \partial_t h \in L_\frac{\alpha+1}{\alpha}\bigl((0,T);(W^1_{\alpha+1}(S^1))'\bigr),
\end{equation*}
to the regularised problem \eqref{eq:PDE_regularised} that satisfies
\begin{equation*}
    0 < \tfrac{1}{2} \bar{h}_0 \leq h(t,\theta) \leq 2\bar{h}_0, 
    \quad 
    t \in [0,t_\ast],\, \theta \in S^1.
\end{equation*}
We easily find that, for all $0 \leq t \leq T$,
\begin{equation*}
    \|h(t,\cdot) - \bar{h}_0\|_{H^1(S^1)}
    =
    \|h_{\pm 1}(t) e^{\pm i \cdot} + \phi(t,\cdot)\|_{H^1(S^1)}
    \leq
    C
    |h_{\pm 1}(t)| + \|\phi(t,\cdot)\|_{H^1(S^1)}.
\end{equation*}
In view of Theorem \ref{thm:power-law_decay} and Proposition \ref{prop:estimate_Fourier} this yields
\begin{equation*}
    \|h(t,\cdot) - \bar{h}_0\|_{H^1(S^1)}
    \leq
    C \Lambda_\eps(\bar{t})
    +
    C \eps \left(1 + \log\bigl(\eps^{1-\alpha}\bigr)\right),
    \quad 0 \leq \bar{t} \leq \min\{t_\ast,T\}
\end{equation*}
for all $0 < \sigma \leq \bar{\sigma}(T,\eps)$ and all $0 < \eps \leq \eps_0$.
Thus, choosing $\eps_0$ small enough, we find in particular that
\begin{equation}
    \|h(t,\cdot) - \bar{h}_0\|_{L_\infty(S^1)}
    \leq 
    \frac{1}{4} \bar{h}_0,
    \quad 
    0 \leq t \leq \min\{t_\ast,T\}.
\end{equation}
Now, assume (for contradiction) that $t_\ast < T$. Then
\begin{equation*}
    \frac{3}{4} \bar{h}_0 
    \leq
    h(t,\theta)
    \leq
    \frac{5}{4} \bar{h}_0, 
    \quad
    0 \leq t \leq t_\ast.
\end{equation*}
This contradicts the definition of $t_\ast$, since by continuity of the solution, we can extend it to larger times in a way that it satisfies the bounds $\frac{1}{2} \bar{h}_0 \leq h(t,\theta) \leq 2 \bar{h}_0$.
\end{proof}


As a consequence of the previous results we find that there exists a positive weak solution $h^\sigma$ to the regularised problem \eqref{eq:PDE_regularised} up to any arbitrary but finite positive time $0 < T < \infty$ if we only choose $0 < \sigma \leq \bar{\sigma}(T,\eps_0)$ small enough. The precise statement reads as follows.


\begin{theorem} \label{thm:existence_sigma}
There exists an $\eps_0 > 0$ such that for all $0 < \eps < \eps_0$, all initial values $h_0^\sigma \in C^\infty(S^1)$ satisfying 
\begin{equation*}
    \frac{1}{2\pi}\int_{S^1} h_0^\sigma\, d\theta = \bar{h}_0
    \quad \text{and} \quad
    \|h_0^\sigma - \bar{h}_0^\sigma\|_{H^1(S^1)} \leq \eps,
    \quad \eps \in (0,\eps_0),
\end{equation*}
and every fixed $T > 0$, there exists $\bar{\sigma} = \bar{\sigma}(T,\eps) > 0$ such that, as long as $0 < \sigma \leq \bar{\sigma}$, the regularised problem \eqref{eq:PDE_regularised} possesses a positive weak solution $h^\sigma$ in the sense that
\begin{equation*}
    h^\sigma \in 
    C\bigl([0,T];H^1(S^1)\bigr)
    \cap 
    L_{\alpha+1}\bigl((0,T);W^3_{\alpha+1}(S^1)\bigr)
    \quad \text{and} \quad 
    \partial_t h^\sigma \in L_\frac{\alpha+1}{\alpha}\bigl((0,T);(W^1_{\alpha+1}(S^1))'\bigr)
\end{equation*}
satisfies the equation 
\begin{equation*}
    \int_0^T \langle \partial_t h^\sigma(t),\varphi(t)\rangle_{W^1_{\alpha+1}(S^1)}\, dt
	=
	\int_0^T \int_{S^1} |h^\sigma|^{\alpha+2} \psi_\sigma(\partial_\theta h^\sigma +\partial_\theta^3 h^\sigma) \partial_\theta \varphi\, d\theta\, dt
\end{equation*}
for all test functions $\varphi \in L_{\alpha+1}\bigl((0,T);W^1_{\alpha+1}(S^1)\bigr)$.
Moreover, the solution has the following properties. 
\begin{itemize}
    \item[(i)] (Positivity, boundedness) The solution is bounded away from zero and bounded above. More precisely,
    $$
    \frac{1}{2} \bar{h}_0^\sigma \leq h^\sigma(t,\theta) \leq 2\bar{h}_0^\sigma, 
    \quad
    t \in [0,T],\, \theta \in S^1.
    $$
    \item[(ii)] (Conservation of mass) The mass of the fluid film is conserved in the sense that
    $$
    \left\|h^\sigma(t)\right\|_{L_1(S^1)} = \left\|h_0^\sigma\right\|_{L_1(S^1)}, \quad t \in [0,T].
    $$
    \item[(iii)] (Energy dissipation) The solution satisfies the energy-dissipation equality
    \begin{equation*}
        E[h^\sigma](t) + \int_0^t \int_{S^1} |h^\sigma|^{\alpha+2} \bigl(|\partial_\theta h^\sigma +\partial_\theta^3 h^\sigma|^2 + \sigma^2\bigr)^\frac{\alpha-1}{2} |\partial_\theta h^\sigma +\partial_\theta^3 h^\sigma|^2\, d\theta\, ds 
        = E[h_0^\sigma], 
        \quad 
        t \in [0,T].
    \end{equation*}
\end{itemize}
\end{theorem}

\bigskip

\section{The limit $\sigma \to 0^+$. Solutions to the original problem \eqref{eq:PDE_alpha>1}.} \label{sec:sigma_to_zero}


In this section we prove global existence of positive weak solutions to the the original problem \eqref{eq:PDE_alpha>1} for flow-behaviour exponents $\alpha > 1$ and initial values $h_0 \in H^1(S^1)$ the shape of which is close to a circle centered at the origin. We use the energy-dissipation equality in order to derive suitable uniform (in $\sigma$) estimates for the solutions $h^\sigma$ to the regularised problem \eqref{eq:PDE_regularised}. The resulting sequence $(h^\sigma)_\sigma$ of solutions admits an accumulation point which turns out to be a solution to the original problem \eqref{eq:PDE_alpha>1}. Note that from now on we use the notation $h^\sigma$ and $h$ for solutions to the regularised problem \eqref{eq:PDE_regularised} and to the original problem \eqref{eq:PDE_alpha>1}, respectively. More precisely, given any fixed $0 < T < \infty$ and $0 < \sigma < 1$, we denote solutions to the regularised problem \eqref{eq:PDE_regularised} by
\begin{equation*}
    h^\sigma(t,\theta) 
    = 
    \bar{h}_0 + h^\sigma_{\pm 1}(t) e^{\pm i \theta} 
    + 
    \phi^\sigma(t,\theta), 
    \quad
    0 < t < T,\, \theta \in S^1.
\end{equation*}
The main result of this section is the following.


\begin{theorem}[Global existence for the original problem \eqref{eq:PDE_alpha>1}]\label{thm:global_ex_original}
There exists an $\eps_0 > 0$ such that for all $0 < \eps \leq \eps_0$, all initial values $h_0 \in H^1(S^1)$ satisfying
\begin{equation*}
    \frac{1}{2 \pi} \int_{S^1} h_0\, d\theta = \bar{h}_0
    \quad \text{and} \quad 
    \|h_0 - \bar{h}_0\|_{H^1(S^1)} \leq \eps, 
    \quad
    \eps \in (0,\eps_0),
\end{equation*}
and for every fixed $T > 0$ the original problem \eqref{eq:PDE_alpha>1} possesses a positive weak solution $h$ in the sense that
\begin{equation*}
    h \in 
    C\bigl([0,T];H^1(S^1)\bigr)
    \cap 
    L_{\alpha+1}\bigl((0,T);W^3_{\alpha+1}(S^1)\bigr)
    \quad \text{and} \quad 
    \partial_t h \in L_\frac{\alpha+1}{\alpha}\bigl((0,T);(W^1_{\alpha+1}(S^1))'\bigr)
\end{equation*}
satisfies the weak formulation
\begin{equation*}
    \int_0^T \langle \partial_t h(t),\varphi(t)\rangle_{W^1_{\alpha+1}(S^1)}\, dt
	=
	\int_0^T \int_{S^1} h^{\alpha+2} \psi(\partial_\theta h+\partial_\theta^3 h) \partial_\theta \varphi\, d\theta\, dt
\end{equation*}
for all test functions $\varphi \in L_{\alpha+1}\bigl((0,T);W^1_{\alpha+1}(S^1)\bigr)$.
Moreover, the solution has the following properties 
\begin{itemize}
    \item[(i)] (Positivity, boundedness) The solution is bounded away from zero and bounded above. More precisely,
    $$
    \frac{1}{2} \bar{h}_0 \leq h(t,\theta) \leq 2\bar{h}_0, 
    \quad
    t \in [0,T],\, \theta \in S^1.
    $$
    \item[(ii)] (Conservation of mass) The mass of the fluid is conserved in the sense that
    $$
    \left\|h(t)\right\|_{L_1(S^1)} = \left\|h_0\right\|_{L_1(S^1)}, \quad t \in [0,T].
    $$
    \item[(iii)] (Energy dissipation) The solution satisfies the energy-dissipation identity
    \begin{equation*}
        E[h](t) + \int_0^t \int_{S^1} h^{\alpha+2} |\partial_\theta h+\partial_\theta^3 h|^{\alpha+1}\, d\theta\, ds 
        = E[h_0]
    \end{equation*}
    for almost every $t \in [0,T]$.
\end{itemize}
\end{theorem}


It is worthwhile to emphasise again that the solution obtained in this theorem is globally defined. In order to be able to prove Theorem \ref{thm:global_ex_original}, we first derive suitable uniform (in $\sigma$) a-priori estimates.


\begin{lemma}[Uniform bounds] \label{lem:uniform_bounds}
Let $h^\sigma$ be a local weak solution to \eqref{eq:PDE_regularised} for a fixed $\sigma \in (0,1)$, corresponding to some $\eps_0 > 0$ and an initial value $h_0^\sigma \in C^\infty(S^1)$ satisfying 
\begin{equation*}
    \frac{1}{2\pi}\int_{S^1} h_0^\sigma\, d\theta = \bar{h}_0^\sigma,
    \quad
    \|h_0^\sigma - \bar{h}_0^\sigma\|_{H^1(S^1)} \leq \eps,
    \quad \eps \in (0,\eps_0),
    \quad \text{and} \quad
    h_0^\sigma \longrightarrow h_0 \quad \text{in } H^1(S^1).
\end{equation*}
Then the following holds true. For all $0 < T < \infty$ there exists $\bar{\sigma} = \bar{\sigma}(\eps_0,T)$ such that, for $0 < \sigma \leq \bar{\sigma}$, the family
$(h^\sigma)_\sigma$ has the following properties.
\begin{itemize}
	\item[(i)] $(h^\sigma)_\sigma$ is uniformly bounded in $L_\infty\bigl((0,T);H^1(S^1)\bigr)$; 
	\item[(ii)] $\bigl(\left|h^\sigma\right|^{\alpha+2} \psi_\sigma\bigl(\partial_\theta h^\sigma + \partial_\theta^3 h^\sigma\bigr)\bigr)_\sigma$ is uniformly bounded in $L_\frac{\alpha+1}{\alpha}\bigl([0,T]\times S^1\bigr)$;
	\item[(iii)] $(\partial_t h^\sigma)_\sigma$ is uniformly bounded in $L_\frac{\alpha+1}{\alpha}\bigl((0,T);(W^1_{\alpha+1}(S^1))'\bigr)$;
	\item[(iv)] $(\partial_\theta h^\sigma +\partial_\theta^3 h^\sigma)_\sigma$ is uniformly bounded in $L_{\alpha+1}\bigl([0,T]\times S^1\bigr)$;
	\item[(v)] $(h^\sigma)_\sigma$ is uniformly bounded in $L_{\alpha+1}\bigl((0,T);W^3_{\alpha+1}(S^1)\bigr)$;
	\item[(vi)] $(\partial_t(\partial_\theta h^\sigma))_\sigma$ is uniformly bounded in $L_\frac{\alpha+1}{\alpha}\bigl((0,T);\bigl(W^1_{\alpha+1,0}(S^1)\cap W^2_{\alpha+1}(S^1)\bigr)'\bigr)$.
\end{itemize}
\end{lemma}


\begin{proof}
\noindent (i) Let $T > 0$ and $\sigma \in (0,1)$ be given. We write
\begin{equation*}
    h^\sigma(t,\theta) 
    = 
    \bar{h}_0 + h^\sigma_{\pm 1}(t) e^{\pm i \theta} 
    + 
    \phi^\sigma(t,\theta), 
    \quad
    0 < t < T,\, \theta \in S^1,
\end{equation*}
for a solution $h^\sigma$ to the regularised problem \eqref{eq:PDE_regularised}. Using the decay estimate \eqref{eq:decay_est} derived in Theorem \ref{thm:power-law_decay} and the bound for the Fourier coefficients $h^\sigma_{\pm 1}$, derived in Proposition \ref{prop:estimate_Fourier} (together with Lemma \ref{lem:T_larger_t_ast}) we find that there exists a $\bar{\sigma} > 0$ such that for all $0 < \sigma < \bar{\sigma}$  the estimate
\begin{equation*}
\begin{split}
    \|h^\sigma\|_{L_\infty((0,T);H^1(S^1))}
    &=
    \|\bar{h}_0 + h^\sigma_{\pm 1} e^{\pm i \theta} 
    + 
    \phi^\sigma\|_{L_\infty((0,T);H^1(S^1))}
    \\
    &\leq
    C_0\, |\bar{h}_0|
    +
    C_{\pm 1}\, \|h^\sigma_{\pm 1}\|_{L_\infty((0,T))}
    +
    \|\phi^\sigma\|_{L_\infty((0,T);H^1(S^1))}
    \\
    &\leq
    C_0 + C_{\pm 1} \eps \left(1 + \log\bigl(\eps^{1-\alpha}\bigr)\right)
    +
    C_2 \Lambda_\eps(T)
    \\
    &\leq 
    C_T
\end{split}
\end{equation*}
holds true with positive constants $C_0, C_{\pm 1}, C_2, C_T > 0$.

\noindent (ii) As in Lemma \ref{lem:estimate_flux_dissipation} we can prove that there exists a positive constant $C >0$, independent of $\sigma$, such that
\begin{equation*}
    \left\|\left|h^\sigma\right|^{\alpha+2} \psi_\sigma\bigl(\partial_\theta h^\sigma + \partial_\theta^3 h^\sigma\bigr)\right\|_{L_\frac{\alpha+1}{\alpha}([0,T]\times S^1)}^\frac{\alpha}{\alpha+1}
    \leq
    C T^\frac{1}{\alpha+1} \left(T \sigma^{\alpha+1} + D^\sigma_{T}[h^\sigma]\right)^\frac{\alpha}{\alpha+1}
    \leq
    C_T.
\end{equation*}

\noindent (iii) Similarly as in step (ii) we may use the weak formulation of \eqref{eq:PDE_regularised} and H\"older's inequality to obtain
\begin{align*}
    &
    \left|\int_0^T \langle \partial_t h^\sigma(t), \varphi(t) \rangle_{W^1_{\alpha+1}(S^1)}\, dt
    \right|
    \leq
    \int_0^T \int_{S^1} \left|h^\sigma\right|^{\alpha+2} \left|\psi_\sigma\bigl(\partial_\theta h^\sigma + \partial_\theta^3 h^\sigma\bigr)\right|\, \left|\partial_\theta \varphi\right|\, d\theta\, dt
    \\
    &\quad 
    \leq
    C
    \left(\int_0^T \int_{S^1} \left|h^\sigma\right|^{\alpha+2} \left|\partial_\theta \varphi\right|^{\alpha+1}\, d\theta\, dt
    \right)^\frac{1}{\alpha+1}
    \cdot
    \\
    &\quad\quad \cdot
    \left(\int_0^T \int_{S^1} \left|h^\sigma\right|^{\alpha+2} \bigl(|\partial_\theta h^\sigma + \partial_\theta^3 h^\sigma|^2 + \sigma^2\bigr)^{\frac{\alpha-1}{2}\frac{\alpha+1}{\alpha}} |\partial_\theta h^\sigma + \partial_\theta^3 h^\sigma|^\frac{\alpha+1}{\alpha}\, d\theta\, dt
    \right)^\frac{\alpha}{\alpha+1}
    \\
    &\quad 
    \leq
    C T^\frac{1}{\alpha+1} \left(T \sigma^{\alpha+1} + D_T^\sigma[h^\sigma]\right)^\frac{\alpha}{\alpha+1}
    \leq
    C,
\end{align*}
with a positive constant $C>0$ that does not depend on $\sigma$. 

\noindent (iv) We prove that $(\partial_\theta h^\sigma + \partial_\theta^3 h^\sigma)$ is uniformly bounded in $L_{\alpha+1}\bigl([0,T]\times S^1\bigr)$. To this end, recall that $h^\sigma$ is, for each $\sigma \leq \bar{\sigma}$, bounded away from zero for times $0< t \leq T$. This implies
\begin{align*}
    \int_0^{T} \int_{S^1} |\partial_\theta h^\sigma + \partial_\theta^3 h^\sigma|^{\alpha+1}\, d\theta\, dt
    &\leq
    C 
    \int_0^{T} \int_{S^1} \left|h^\sigma\right|^{\alpha+2} \psi_\sigma\bigl(\partial_\theta h^\sigma + \partial_\theta^3 h^\sigma\bigr)\,|\partial_\theta h^\sigma + \partial_\theta^3 h^\sigma|^2\, d\theta\, dt
    \\
    &=
    C\, D_{T}^\eps[h^\sigma].
\end{align*}
Using again the dissipation estimate in Lemma \ref{lem:energy_dissipation}, we obtain the desired bound
\begin{equation*}
    \int_0^{T} \int_{S^1} |\partial_\theta h^\sigma + \partial_\theta^3 h^\sigma|^{\alpha+1}\, d\theta\, dt
    \leq
    C_{h_0}.
\end{equation*}

\noindent (v) This part of the proof is almost identical to \cite[Lemma 3.9 (iv)]{LPTV} but we adapt it to our setting for the sake of completeness. We prove the estimate
\begin{equation*}
    \int_0^{T}\int_{S^1} |\partial_\theta^3 h^\sigma|^{\alpha+1}d\theta\, dt 
    \leq 
    \int_0^{T} \int_{S^1} |\partial_\theta h^\sigma + \partial_\theta^3 h^\sigma|^{\alpha+1} d\theta\, dt + C(T)\, \|h^\sigma\|_{H^1(S^1)}^{\alpha+1}.
\end{equation*}    
To this end, 
we define $V_0\equiv \text{span}\{\cos(\theta),\sin(\theta)\}\subset L_{2}(S^1)$ and $V_1$ as the orthogonal complement of $V_0\oplus \text{span}\{1\}$ in $L_{2}(S^1)$.
Given $h^\sigma\in H^3(S^1)$, we may decompose $h^\sigma$ as 
\begin{equation}\label{descomposicion}
    h^\sigma 
    = 
    a_0 + u + v
    \quad
    \text{with}
    \quad 
    a_0\in\R,\, u\in V_0\cap H^3(S^1)
    \quad \text{and} \quad  
    v\in V_1\cap H^3(S^1)
\end{equation}
and write
\begin{equation*}
    (\partial_\theta^3 v)_n
    =
    -in^{3}v_n
    =
    m(n)i(n-n^{3})v_n,
\end{equation*}
where $m(n)=\frac{-in^{3}}{i(n-n^{3})}=\frac{n^{2}}{n^{2}-1}$ for $n\neq 0,\pm 1$. Since $m(n)$ is bounded, we may use the Littlewood--Paley Theory (c.f. \cite[Chapter 4]{stein}) to obtain 
\begin{equation}\label{littlewood}
    \|\partial_\theta^3 v\|_{L_{\alpha+1}(S^1)}^{\alpha+1}
    \leq 
    K \|\partial_\theta v + \partial_\theta^3v\|_{L_{\alpha+1}(S^1)}^{\alpha+1}
\end{equation}
with a positive constant $K>0$ that does not depend on $v$.
Therefore, using \eqref{littlewood} and the fact that $u=a_1(t)\cos\theta+a_{-1}(t)\sin\theta$, we can write
\begin{equation}
\begin{split}
    \int_0^T\int_{S^1}
    |\partial_\theta^3 h^\sigma|^{\alpha+1}d\theta\, dt
    & \leq 
    C\left(\int_0^T\int_{S^1} |\partial_\theta^3 u|^{\alpha+1} d\theta\, dt
    +
    \int_0^T\int_{S^1} |\partial_\theta^3 v|^{\alpha+1} d\theta\, dt\right)\\\label{eq86}
    &\leq 
    C\left(\int_0^T |a_1(t)|^{\alpha+1}
    +
    |a_{-1}(t)|^{\alpha+1}\, dt\right) \\ &\quad 
    + CK \left(\int_0^T\int_{S^1} |\partial_\theta v+\partial_\theta^3  v|^{\alpha+1}\, d\theta\, dt\right).  
\end{split}
\end{equation}
Indeed, for the first term on the right-hand side of \eqref{eq86} we use the structure of $u$ and Young's inequality for convolutions to derive the pointwise estimate 
\begin{equation*}
\begin{split}
    |a_1(t)|^{\alpha+1} + |a_{-1}(t)|^{\alpha+1}
    \leq 
    C\left(|a_1(t)|^{2} + |a_{-1}(t)|^{2}\right)^{\frac{\alpha+1}{2}} 
    = 
    C\, \|u(t)\|_{L_2(S^1)}^{\alpha+1} 
    \leq 
    C\|h^\sigma(t)\|_{H^1(S^1)}^{\alpha+1}.
\end{split}
\end{equation*}
For the second integral on the right-hand side of \eqref{eq86} we use that $\partial_\theta u+\partial_\theta^3 u = 0$ and we obtain that
\begin{equation*}
    \int_0^T\int_{S^1}|\partial_\theta v+\partial_\theta^3 v|^{\alpha+1}\, d\theta\, dt
    =
    \int_0^T\int_{S^1} |\partial_\theta h^\sigma+\partial_\theta^3  h^\sigma|^{\alpha+1}\, d\theta\, dt.
\end{equation*}
Thus, we can conclude that
\begin{equation*}
\begin{split}
    \int_0^T\int_{S^1} |\partial_\theta^3  h^\sigma|^{\alpha+1}d\theta\, dt 
    &\leq 
    C \int_0^T \|h^\sigma(t)\|_{H^1(S^1)}^{\alpha+1}\, dt
    + C_K \int_0^T\int_{S^1} |\partial_\theta h^\sigma+\partial_\theta^3  h^\sigma|^{\alpha+1}\, d\theta\, dt
    \leq C(T,h_0)
\end{split}
\end{equation*}
and we have proved the desired result.

(vi) This follows similarly as in (iii).

\end{proof}


\begin{lemma}[Convergence of approximations] \label{lem:convergence}
Let $h^\sigma$ be a local weak solution to \eqref{eq:PDE_regularised} for a fixed $\sigma \in (0,1)$, corresponding to some $\eps_0 > 0$ and an initial value $h_0^\sigma \in C^\infty(S^1)$ satisfying 
\begin{equation*}
    \frac{1}{2\pi}\int_{S^1} h_0^\sigma\, d\theta = \bar{h}_0^\sigma,
    \quad
    \|h_0^\sigma - \bar{h}_0^\sigma\|_{H^1(S^1)} \leq \eps,
    \quad \eps \in (0,\eps_0),
    \quad \text{and} \quad
    h_0^\sigma \longrightarrow h_0 \quad \text{in } H^1(S^1).
\end{equation*}
Then the following holds true. For all $T > 0$ there exists a subsequence $(h^\sigma)_\sigma$ (not relabelled) such that, as $\sigma \searrow 0$, we have convergence in the following sense.
\begin{itemize}
	\item[(i)] $h^\sigma \to h$ strongly in $C\bigl([0,T];C^{\rho}(S^1)\bigr)$; 
	\item[(ii)] $\left|h^\sigma\right|^{\alpha+2} \psi_\sigma\bigl(\partial_\theta h^\sigma + \partial_\theta^3 h^\sigma\bigr) \rightharpoonup \chi$ weakly in $L_\frac{\alpha+1}{\alpha}\bigl([0,T]\times S^1\bigr)$ for some limit function $\chi$;
	\item[(iii)] $\partial_t h^\sigma \rightharpoonup \partial_t h$ weakly in 		$L_\frac{\alpha+1}{\alpha}\bigl((0,T); (W^1_{\alpha+1}(S^1))'\bigr)$;
	\item[(iv)] $\partial_\theta h^\sigma + \partial_\theta^3 h^\sigma \rightharpoonup \partial_\theta h + \partial_\theta^3 h$ weakly in $L_{\alpha+1}\bigl([0,T]\times S^1\bigr)$;
	\item[(v)] $\partial_t(\partial_\theta h^\sigma) \rightharpoonup \partial_t\partial_\theta h$ weakly in $L_\frac{\alpha+1}{\alpha}\bigl((0,T);\bigl(W^1_{\alpha+1,0}(S^1)\cap W^2_{\alpha+1}(S^1)\bigr)'\bigr)$.
\end{itemize}
\end{lemma}


\begin{proof}
(i) In the previous Lemma \ref{lem:uniform_bounds} (i), (iii) we have proved that 
\begin{equation*}
	\begin{cases}
		(h^\sigma)_\sigma \text{ is uniformly bounded in } L_\infty\bigl((0,T);H^1(S^1)\bigr) 
		& \\
		(\partial_t h^\sigma)_\sigma \text{ is uniformly bounded in } L_\frac{\alpha+1}{\alpha}\bigl((0,T);(W^1_{\alpha+1}(S^1))'\bigr).&
    \end{cases}
\end{equation*}
Moreover, thanks to the Rellich-Kondrachov theorem, cf. for instance in \cite[Thm. 6.3]{adams_fournier}, we know that
\begin{equation*}
	H^1(S^1) \xhookrightarrow[]{c} C^{\rho}(S^1) \hookrightarrow (W^1_{\alpha+1}(S^1))', \quad \rho \in [0,1/2),
\end{equation*}
where $\xhookrightarrow[]{c}$ indicates compactness of the embedding.
This allows us to invoke \cite[Cor. 4]{S:1987} in order to conclude that the sequence 
\begin{equation*}
	(h^\sigma)_\sigma \text{ is relatively compact in } C\bigl([0,T];C^\rho(S^1)\bigr) 
\end{equation*}
with $\rho \in [0,1/2)$ as above.
	    
(ii) This is an immediate consequence of Lemma \ref{lem:uniform_bounds} (ii).
	
(iii) Thanks to Lemma \ref{lem:uniform_bounds} (iii), we may extract a subsequence $(\partial_t h^\sigma)_\sigma$ such that
\begin{equation*}
	\partial_t h^\sigma \rightharpoonup v 
	\quad \text{weakly in } L_\frac{\alpha+1}{\alpha}\bigl((0,T);(W^1_{\alpha+1}(S^1))'\bigr) \hookrightarrow \Dcal'\bigl((0,T);(W^1_{\alpha+1}(S^1))'\bigr)
\end{equation*}
for some limit function $v \in L_\frac{\alpha+1}{\alpha}\bigl((0,T);(W^1_{\alpha+1}(S^1))'\bigr)$.
Since we know in addition that
\begin{equation*}
    h^\sigma \longrightarrow h \quad \text{in }  C\bigl([0,T];C^\rho(S^1)\bigr) \hookrightarrow \Dcal'\bigl((0,T);(W^1_{\alpha+1}(S^1))'\bigr),
\end{equation*}
$\rho \in [0,1/2)$, we conclude that 
\begin{equation*}
	\partial_t h^\sigma \longrightarrow 
	\partial_t h
	\quad \text{in }
	\Dcal'\bigl((0,T);(W^1_{\alpha+1}(S^1))'\bigr)\bigr),
\end{equation*}
and consequently, $v = \partial_t h \in L_\frac{\alpha+1}{\alpha}\bigl((0,T);(W^1_{\alpha+1}(S^1))'\bigr)$.
	
(iv) The strong convergence $h^\sigma \to h$ in $C\bigl([0,T];C^\rho(S^1)\bigr),\, \rho \in [0,1/2),$ in Lemma \ref{lem:convergence} (i) in particular implies uniform convergence 
\begin{equation}\label{eq:conv_1}
    h^\sigma 
    \longrightarrow 
    h 
    \quad
    \text{in } 
    C\bigl([0,T]\times S^1\bigr).
\end{equation}
Moreover, Lemma \ref{lem:uniform_bounds} (v) guarantees the existence of some $\hat{h} \in L_{\alpha+1}\bigl((0,T);W^3_{\alpha+1}(S^1)\bigr)$ such that
\begin{equation} \label{eq:conv_2}
    h^\sigma \rightharpoonup \hat{h}
	\quad
	\text{in } L_{\alpha+1}\bigl((0,T);W^3_{\alpha+1}(S^1)\bigr).
\end{equation}
In virtue of the uniqueness of the limit function, \eqref{eq:conv_1} and \eqref{eq:conv_2} imply
\begin{equation*}
     h^\sigma \rightharpoonup h
    \quad
    \text{in } 
    L_{\alpha+1}\bigl((0,T);W^3_{\alpha+1}(S^1)\bigr).
\end{equation*}
Thanks to the weak lower semicontinuity of the norm and Lemma \ref{lem:uniform_bounds} (iv), (v), we finally obtain
\begin{equation} \label{eq:lsc}
	\begin{cases}
	    \left\|\partial_\theta h+\partial_\theta^3 h\right\|_{L_{\alpha+1}((0,T)\times S^1)}
	    \leq
	    \liminf_{\sigma \to 0} \left\|\partial_\theta h^\sigma + \partial_\theta^3 h^\sigma\right\|_{L_{\alpha+1}((0,T)\times S^1)}
	    \leq 
	    C
	    & \\
	    \left\|h\right\|_{L_{\alpha+1}((0,T);W^3_{\alpha+1}(S^1))}
	    \leq
	    \liminf_{\sigma \to 0} \left\| h^\sigma\right\|_{L_{\alpha+1}((0,T);W^3_{\alpha+1}(S^1))}
	    \leq 
	    C &
	\end{cases}
\end{equation}
for some positive generic constant $C > 0$ that does not depend on $\sigma$. 

(v) This follows similarly as in (iii) and the proof is complete.
\end{proof}


It remains to prove the convergence of the nonlinear flux term
$\bigl(\left|h^\sigma\right|^{\alpha+2} \psi_\sigma\bigl(\partial_\theta h^\sigma + \partial_\theta^3 h^\sigma\bigr)\bigr) \rightharpoonup \bigl(\left|h\right|^{\alpha+2} \psi\bigl(\partial_\theta h + \partial_\theta^3 h\bigr)\bigr)$ in $L_\frac{\alpha+1}{\alpha}([0,T]\times S^1)$. This is the content of the next lemma the proof of which is based on Minty's trick.


\begin{lemma}\label{lem:limit_flux}
Given $\sigma \in (0,1)$, let $h^\sigma$ be the positive solution to \eqref{eq:PDE_regularised} on $[0,T]$. Then there exists a subsequence $(h^\sigma)_\sigma$ (not relabelled) such that
\begin{equation*}
	\left|h^\sigma\right|^{\alpha+2} \psi_\sigma\bigl(\partial_\theta h^\sigma + \partial^3_\theta h^\sigma\bigr) 
	\xrightharpoonup[\phantom{wea}]{}
	\left|h\right|^{\alpha+2} \psi\bigl(\partial_\theta h + \partial_\theta^3 h\bigr) 
	\quad
	\text{weakly in }
	L_\frac{\alpha+1}{\alpha}\bigl([0,T]\times S^1\bigr)
\end{equation*}
as $\sigma \searrow 0$.
\end{lemma}


\begin{proof}
Note that for convenience we pass to a subsequence where necessary without explicitly mentioning.
	
(i) 
Thanks to Lemma \ref{lem:convergence} (ii), we know that $\left|h^\sigma\right|^{\alpha+2} \psi_\sigma\bigl(\partial_\theta h^\sigma + \partial^3_\theta h^\sigma\bigr)$ is weakly sequentially compact, i.e. there is an element $\chi \in L_\frac{\alpha+1}{\alpha}\bigl((0,T)\times S^1)\bigr)$ such that
\begin{equation*}
    \left|h^\sigma\right|^{\alpha+2} \psi_\sigma\bigl(\partial_\theta h^\sigma + \partial^3_\theta h^\sigma\bigr) 
    \xrightharpoonup[\phantom{wea}]{}
    \chi
    \quad
	\text{weakly in } 
	L_\frac{\alpha+1}{\alpha}\bigl((0,T)\times S^1)\bigr).
\end{equation*}
It remains to identify the limit flux $\chi$.
	
(ii) The next step is to prove that $h$ is bounded in $C\bigl([0,T];H^1(S^1)\bigr)$. From Lemma \ref{lem:convergence} (i) we already know that
\begin{equation*}
    h \in C\bigl([0,T];C^\rho(S^1)\bigr) 
    \xhookrightarrow[\phantom{wea}]{} 
    C\bigl([0,T];L_2(S^1)\bigr).
\end{equation*}
Moreover,
\begin{equation*}
    \partial_\theta h \in L_{\alpha+1}\bigl((0,T);W^1_{\alpha+1,0}(S^1) \cap W^2_{\alpha+1}(S^1)\bigr) \quad \text{and} \quad
	\partial_t\partial_\theta h\in L_\frac{\alpha+1}{\alpha}\bigl((0,T);\bigl(W^1_{\alpha+1,0}(S^1)\cap W^2_{\alpha+1}(S^1)\bigr)'\bigr)
\end{equation*}
thanks to Lemma \ref{lem:convergence} (iv) and (v) and lower semicontinuity of the norm. Using \cite[Remark 3.4]{Bernis:1988}, this implies that $\partial_\theta h \in C\bigl([0,T];L_2(S^1)\bigr)$. Consequently, $h \in C\bigl([0,T];H^1(S^1)\bigr)$.
	
(iii) In view of the previous steps we may choose $\varphi = h + \partial_\theta^2 h \in L_{\alpha+1}\bigl((0,T);W^1_{\alpha+1}(S^1)\bigr)$ as a test function in the equation \eqref{eq:PDE_regularised} for $h^\sigma$. This yields, 
\begin{equation*}
    \int_0^T \int_{S^1} \partial_t h^\sigma (h + \partial_\theta^2 h)\, d\theta\, dt + \int_0^T \int_{S^1} \left|h^\sigma\right|^{\alpha+2} \psi_\sigma\left(\partial_\theta h^\sigma+ \partial_\theta^3 h^\sigma\right) \bigl(\partial_\theta h + \partial_\theta^3 h\bigr)\, d\theta\, dt = 0.
\end{equation*}
As $\sigma \searrow 0$, the first term satisfies
\begin{equation*}
    \int_0^T \int_{S^1} \partial_t h^\sigma\, (h + \partial_\theta^2 h)\, d\theta\, dt \longrightarrow 
	\int_0^T \int_{S^1} \partial_t h\, (h + \partial_\theta^2 h)\, d\theta\, dt
	=
	E[h](T) - E[h](0),
\end{equation*}
where we have used that the limit function satisfies $h \in C\bigl([0,T];H^1(S^1)\bigr)$.
Moreover, since $(\partial_\theta h+\partial_\theta^3  h)$ is bounded in $L_{\alpha+1}\bigl((0,T)\times S^1\bigr)$ by \eqref{eq:lsc}, we may infer from Lemma \ref{lem:convergence} (ii) that
\begin{equation*}
    \int_0^T \int_{S^1} \left|h^\sigma\right|^{\alpha+2} \psi_\sigma\bigl(\partial_\theta h^\sigma + \partial_\theta^3 h^\sigma\bigr) \bigl(\partial_\theta h + \partial_\theta^3 h\bigr)\, d\theta\, dt
	\longrightarrow 
	\int_0^T \int_{S^1} \chi \bigl(\partial_\theta h + \partial_\theta^3 h\bigr)\, d\theta\, dt,
\end{equation*}
and consequently, we obtain the identity
\begin{equation*}
	E[h](t) + \left\langle \chi | \partial_\theta h + \partial_\theta^3 h \right\rangle
	=
	E[h_0]
\end{equation*}
for almost every $t \in [0,T]$.
	
(iv) Monotonicity and identification of the limit flux $\chi$ by Minty's trick. Observe that the operator 
\begin{equation*}
	\begin{cases}
	    \psi_\sigma \colon L_{\alpha+1}\bigl((0,T)\times S^1\bigr) \longrightarrow L_{\frac{\alpha+1}{\alpha}}\bigl((0,T)\times S^1\bigr), & \\
    	\psi_\sigma(u) = \bigl(|u|^2 + \sigma^2\bigr)^\frac{\alpha-1}{2} u &
    \end{cases}
\end{equation*}
is monotone, i.e. we have
\begin{equation*}
	\langle \psi_\sigma(u) - \psi_\sigma(v),u-v\rangle_{L_{\alpha+1}}
	=
	\int_0^T \int_{S^1} \bigl(\psi_\sigma(u) - \psi_\sigma(v)\bigr) (u-v)\, 	d\theta\, dt
	> 0
\end{equation*} 
for all $u, v \in L_{\alpha+1}\bigl([0,T]\times S^1\bigr)$ with $u\neq v$, where we use the notation $\langle u | v\rangle$ to describe the dual pairing $\langle u |v \rangle_{L_{\alpha+1}((0,T)\times S^1)}$ between $u \in L_\frac{\alpha+1}{\alpha}\bigl((0,T)\times S^1\bigr)$ and $v \in L_{\alpha+1}\bigl((0,T)\times S^1\bigr)$.

Let now $\phi \in W^3_{\alpha+1}\bigl((0,T)\times S^1\bigr)$. Thanks to the monotonicity of $\psi_\sigma$ we have
\begin{align*}
	0
	&\leq
	\left\langle \left|h^\sigma\right|^{\alpha+2} \psi_\sigma\bigl(\partial_\theta h^\sigma + \partial_\theta^3 h^\sigma\bigr) - \left|h^\sigma\right|^{\alpha+2} \psi_\sigma \bigl(\partial_\theta \phi + \partial_\theta^3 \phi\bigr) | (\partial_\theta + \partial_\theta^3) (h^\sigma-\phi) 
	\right\rangle
	\\
	&=	
	\left\langle \left|h^\sigma\right|^{\alpha+2} \psi_\sigma\bigl(\partial_\theta h^\sigma + \partial_\theta^3 h^\sigma\bigr) | \partial_\theta h^\sigma + \partial_\theta^3 h^\sigma \right\rangle
	-
	\left\langle \left|h^\sigma\right|^{\alpha+2} \psi_\sigma\bigl(\partial_\theta h^\sigma + \partial_\theta^3 h^\sigma\bigr) | \partial_\theta \phi + \partial_\theta^3 \phi \right\rangle
	\\
	&\quad
	-
	\left\langle \left|h^\sigma\right|^{\alpha+2} \psi_\sigma \bigl(\partial_\theta \phi + \partial_\theta^3 \phi\bigr) | \partial_\theta h^\sigma + \partial_\theta^3 h^\sigma \right\rangle
	+
	\left\langle \left|h^\sigma\right|^{\alpha+2} \psi_\sigma \bigl(\partial_\theta \phi + \partial_\theta^3 \phi\bigr) | \partial_\theta \phi + \partial_\theta^3 \phi \right\rangle.
\end{align*}
We consider the four dual pairings on the right-hand side separately.
	
First, we rewrite the energy-dissipation formula for the problem \eqref{eq:PDE_regularised} as 
\begin{equation*}
	\left\langle \left|h^\sigma\right|^{\alpha+2} \psi_\sigma\bigl(\partial_\theta h^\sigma + \partial_\theta^3 h^\sigma\bigr) | \partial_\theta h^\sigma + \partial_\theta^3 h^\sigma
	\right\rangle
	= 
	E[h_0^\sigma] - E[h^\sigma](t),
	\quad
	\text{for a.e. } t \in [0,T].
\end{equation*}
Thanks to Lemma \ref{lem:convergence} (i) we know that $h^\sigma(t) \to h(t)$ in $H^1(S^1)$ for almost every $t \in [0,T]$. Note that $h_0^\sigma \to h_0$ in $H^1(S^1)$ by assumption. Thus, in the limit $\sigma \searrow 0$ we find that
\begin{equation} \label{eq:Minty_1}
	\left\langle \left|h^\sigma\right|^{\alpha+2} \psi_\sigma\bigl(\partial_\theta h^\sigma + \partial_\theta^3 h^\sigma\bigr) | \partial_\theta h^\sigma + \partial_\theta^3 h^\sigma
	\right\rangle
	\longrightarrow
	E[h_0] - E[h](t)
\end{equation}
for almost every $t \in [0,T]$. Moreover, Lemma \ref{lem:convergence} (ii) yields
\begin{equation}\label{eq:Minty_2}
	\left\langle \left|h^\sigma\right|^{\alpha+2} \psi_\sigma\bigl(\partial_\theta h^\sigma + \partial_\theta^3 h^\sigma\bigr) | \partial_\theta \phi + \partial_\theta^3 \phi \right\rangle
	\longrightarrow
	\left\langle \chi | \partial_\theta \phi + \partial_\theta^3 \phi \right\rangle,
	\quad
	\text{as } \sigma \searrow 0,
\end{equation}
for the second dual pairing.
For the third pairing we invoke Lemma \ref{lem:convergence} (i) and Lemma \ref{lem:convergence} (iv) to obtain
\begin{equation*}
	\begin{cases}
		h^\sigma \longrightarrow h 
		\quad \text{strongly in } C\bigl([0,T]\times S^1\bigr) &
		\\
		\partial_\theta h^\sigma + \partial_\theta^3 h^\sigma
		\xrightharpoonup[\phantom{wea}]{}
		\partial_\theta h + \partial_\theta^3 h
		\quad \text{weakly in } L_{\alpha+1}\bigl([0,T]\times S^1\bigr). &
	\end{cases}
\end{equation*}
This implies
\begin{equation}\label{eq:Minty_3}
	\left\langle \left|h^\sigma\right|^{\alpha+2} \psi_\sigma \bigl(\partial_\theta \phi + \partial_\theta^3 \phi\bigr) | \partial_\theta h^\sigma + \partial_\theta^3 h^\sigma \right\rangle
	\longrightarrow
	\left\langle \left|h\right|^{\alpha+2} \psi \bigl(\partial_\theta \phi + \partial_\theta^3 \phi\bigr) | \partial_\theta h + \partial_\theta^3 h \right\rangle,
    \quad
    \text{as } \sigma \searrow 0.
\end{equation}
Clearly, the fourth pairing satisfies
\begin{equation}\label{eq:Minty_4}
    \left\langle \left|h^\sigma\right|^{\alpha+2} \psi_\sigma \bigl(\partial_\theta \phi + \partial_\theta^3 \phi\bigr) | \partial_\theta \phi + \partial_\theta^3 \phi \right\rangle  
    \longrightarrow
    \left\langle \left|h\right|^{\alpha+2} \psi \bigl(\partial_\theta \phi + \partial_\theta^3 \phi\bigr) | \partial_\theta \phi + \partial_\theta^3 \phi \right\rangle,
    \quad
    \text{as } \sigma \searrow 0.
\end{equation}
Combining \eqref{eq:Minty_1}--\eqref{eq:Minty_4} leads to the inequality
\begin{equation*}
	0 
	\leq
	E[h_0] - E[h](t)
	-
	\left\langle \chi | \partial_\theta \phi + \partial_\theta^3 \phi \right\rangle
	-
	\left\langle \left|h\right|^{\alpha+2} \psi\bigl(\partial_\theta \phi + \partial_\theta^3 \phi\bigr) | \partial_\theta (h - \phi) + \partial_\theta^3 (h - \phi)  \right\rangle,
\end{equation*}
and using the identity 
\begin{equation*}
	E[h](t) + \left\langle \chi | \partial_\theta h + \partial_\theta^3 h \right\rangle
	=
	E[h_0],
\end{equation*}
proved in step (iii), for almost every $t \in [0,T]$, we discover that
\begin{equation*}
	0 
	\leq
	\left\langle \chi - \left|h\right|^{\alpha+2} \psi\bigl(\partial_\theta \phi + \partial_\theta^3 \phi\bigr) | \partial_\theta (h - \phi) + \partial_\theta^3 (h - \phi) \right\rangle.
\end{equation*}
By choosing $\phi = h - \lambda v$ for some arbitrary $v \in W^3_{\alpha+1}\bigl((0,T)\times S^1\bigr)$ and $\lambda > 0$, we obtain the inequality
\begin{equation*}
	\left\langle \chi - \left|h\right|^{\alpha+2} \psi\bigl((\partial_\theta + \partial_\theta^3) (h - \lambda v)\bigr) | \partial_\theta v + \partial_\theta^3 v \right\rangle
	\geq 0
\end{equation*}
and hence, in the limit $\lambda \searrow 0$:
\begin{equation*}
	\left\langle \chi - \left|h\right|^{\alpha+2} \psi\bigl(\partial_\theta h + \partial_\theta^3 h\bigr) | \partial_\theta v + \partial_\theta^3 v \right\rangle
	\geq 0,
	\quad
	v \in W^3_{\alpha+1}\bigl((0,T)\times S^1\bigr),
\end{equation*}
for almost every $t \in [0,T]$. Choosing $\phi = h + \lambda v$, we discover that
\begin{equation*}
	\left\langle \chi - \left|h\right|^{\alpha+2} \psi\bigl(\partial_\theta h + \partial_\theta^3 h\bigr) | \partial_\theta v + \partial_\theta^3 v \right\rangle
	\leq 0,
	\quad
	v \in W^3_{\alpha+1}\bigl((0,T)\times S^1\bigr).
\end{equation*}
Consequently, we have proved that
\begin{equation*}
	\left\langle \chi - \left|h\right|^{\alpha+2} \psi\bigl(\partial_\theta h + \partial_\theta^3 h\bigr) | \partial_\theta v + \partial_\theta^3 v \right\rangle
	= 0,
	\quad
	v \in W^3_{\alpha+1}\bigl((0,T)\times S^1\bigr),
\end{equation*}
which allows us to identify
\begin{equation*}
	\chi 
	=
	\left|h\right|^{\alpha+2} \psi\bigl(\partial_\theta h + \partial_\theta^3 h\bigr)
	\in
	L_\frac{\alpha+1}{\alpha}\bigl([0,T]\times S^1\bigr).
\end{equation*}
This completes the proof.
\end{proof}
\medskip

With the previous convergence results at hand, we are now able to prove global existence of positive weak solutions to \eqref{eq:PDE_alpha>1}. 

\medskip

\begin{proof}[\textbf{Proof of Theorem \ref{thm:global_ex_original}}]
\noindent(i) We first prove that we can define the limit function $h$ for all times $0 \leq t < \infty$. To this end, we use a standard diagonal argument. We define a sequence of times $(T_n)_{n \in \N}$ with $T_n = n$ for all $n \in \N$. Moreover, we denote by 
\begin{equation*}
    h_{k,n} = h^{\sigma_{k,n}},
    \quad
    \text{where}
    \quad
    \sigma_{k,n} \to 0
    \quad \text{as} \quad 
    k \to 0
    \quad \text{for all}
    \quad
    n \in \N,
\end{equation*}
the sequence $h^{\sigma_{k,n}} \to h_n,\ k \to \infty$, of functions the convergence of which has been obtained in Lemma \ref{lem:convergence} for $T=T_n$. Note that we can assume that $(h_{k,n+1})_k$ is a subsequence  of $(h_{k,n})_k$ for all $n \in \N$. Therefore, we have $h_{n+1} = h_n$ for any time $0 \leq t \leq T_n$ and all $n \in \N$. This allows us to construct a diagonal sequence $(h_{k_n,n})_{n\in \N}$ such that
\begin{equation*}
    h_{k_j,j} \longrightarrow h,
    \quad \text{as} \quad j \to \infty,
\end{equation*}
for a limit function $h$ which is defined for all times $0\leq t < \infty$ and satisfies $h=h_n$ for $0\leq t \leq T_n$. Since we have strong convergence 
\begin{equation}
    h_{k_j,j} \longrightarrow h_n
    \quad
    \text{in }
    C\bigl([0,T_n]\times S^1\bigr)
\end{equation}
for every $n \in \N$, it follows from Theorem \ref{thm:existence_sigma} (i) that $h = h_n$ satisfies
\begin{equation*}
    \frac{1}{2} \bar{h}_0 \leq h(t,\theta) \leq 2 \bar{h}_0,
    \quad
    0 \leq t < \infty,\, \theta \in S^1.
\end{equation*}

Consequently, it suffices to verify the properties of the limit function, stated in the Theorem, for any arbitrary finite interval $[0,T]$. This is done in the following steps.

\noindent (ii) 
The regularity properties 
\begin{equation*}
    h \in C\bigl([0,T];H^1(S^1)\bigr) \cap L_{\alpha+1}\bigl((0,T);W^3_{\alpha+1}(S^1)\bigr)\quad\text{and}\quad \partial_t h \in L_\frac{\alpha+1}{\alpha}\bigl((0,T);(W^1_{\alpha+1}(S^1))'\bigr)
\end{equation*}
of $h$ are guaranteed thanks to Lemma \ref{lem:convergence} (iii), Lemma \ref{lem:convergence} (iv) and step (ii) of the proof of Lemma \ref{lem:limit_flux}.

\noindent (iii) Next, we prove that $h$ satisfies the weak integral formulation. Indeed, for solutions to the regularised problem \eqref{eq:PDE_regularised} we have that
\begin{equation*}\label{eq:weak_solution_eps}
    \int_0^T \langle \partial_t h^\sigma(t),\varphi(t)\rangle_{W^1_{\alpha+1}(S^1)}\, dt
	=
	\int_0^T \int_{S^1}  \left|h^\sigma\right|^{\alpha+2} \psi_\sigma\left(\partial_\theta h^\sigma + \partial_\theta^3 h^\sigma\right)\, \partial_\theta \phi\,
	d\theta\, dt
\end{equation*}
for all test functions $\phi \in L_{\alpha+1}\bigl((0,T);W^1_{\alpha+1}(S^1)\bigr)$. 
On the one hand, since $\partial_\theta \phi \in L_{\alpha+1}\bigl((0,T)\times S^1)$, we may use Lemma \ref{lem:limit_flux} to obtain the convergence
\begin{equation*}
    \int_0^T \langle \partial_t h^\sigma(t),\varphi(t)\rangle_{W^1_{\alpha+1}(S^1)}\, dt 
    \longrightarrow 
    \int_0^T \int_{S^1} h^{\alpha+2} \psi\bigl(\partial_\theta h + \partial_\theta^3 h\bigr)\, \partial_\theta \phi(t)\, d\theta\, dt.
\end{equation*}
On the other hand, Lemma \ref{lem:convergence} (iii) implies that
\begin{equation*}
    \int_0^T \langle \partial_t h^\sigma(t),\varphi(t)\rangle_{W^1_{\alpha+1}(S^1)}\, dt 
    \longrightarrow 
    \int_0^T \langle \partial_t h(t),\varphi(t)\rangle_{W^1_{\alpha+1}(S^1)}\, dt.
\end{equation*}
Combining both, we observe that $h$ complies with the desired integral identity
\begin{equation*}
    \int_0^T \langle \partial_t h(t),\varphi(t)\rangle_{W^1_{\alpha+1}(S^1)}\, dt = \int_0^T \int_{S^1} h^{\alpha+2} \psi\bigl(\partial_\theta h + \partial_\theta^3 h\bigr)\, \partial_\theta \phi\, d\theta\, dt,
    \quad
    \phi \in L_{\alpha+1}\bigl((0,T);W^1_{\alpha+1}(S^1)\bigr).
\end{equation*}

\noindent (iv) Since we choose $h_0^\sigma$ such that $h^\sigma_0 \to h_0$ in $H^1(S^1)$ the initial condition is satisfied in the limit.

\noindent (v) That a solution conserves its mass follows from the convergence $h^\sigma \to h \in C\bigl([0,T];C^\rho(S^1)\bigr)$, cf. Lemma \ref{lem:convergence} (i), and from the conservation of mass property 
\begin{equation*}
    \int_{S^1} h^\sigma(t,\theta)\, d\theta = \int_{S^1} h_0^\sigma(\theta)\, d\theta, \quad t \in [0,T],
\end{equation*}
for the approximations $h^\sigma$.

\noindent (vi) The proof of Lemma \ref{lem:limit_flux} has also shown that solutions to the original problem \eqref{eq:PDE_alpha>1} satisfy the energy-dissipation identity for almost every time $t \in [0,T]$.
\end{proof}

\bigskip

\section{Long-time asymptotics of solutions to the original problem \eqref{eq:PDE_alpha>1} -- Polynomial stability of steady states} \label{sec:polynomial_stability}

In this section we prove that the original problem \eqref{eq:PDE_alpha>1} possesses, for positive initial data that is close to a circle in $H^1(S^1)$, at least one globally in time defined weak solution which is polynomially stable in the following sense. Given an initial value the shape of which is close to a circle centered at the origin (i.e. the common center of the cylinders) in the $H^1(S^1)$-norm, the shape of the interface will stay close to a circle for all positive times and, moreover, it converges to a circle (not necessarily centered at the origin) with a $\bigl(1/t^{\frac{1}{\alpha-1}}\bigr)$-polynomial decay of the remainder $\phi$ as $t\to \infty$.

The main issue in the proof is to observe that the energy decay estimates of Section \ref{sec:energy_estimates} and the estimates for the Fourier modes $h_{\pm 1}$ are conserved in the limit $\sigma \to 0^+$. This is proved in the following lemma.


\begin{lemma} \label{lem:bounds_in_limit}
For any given $0 < T < \infty$ and an initial value $h_0 \in H^1(S^1)$ satisfying
\begin{equation*}
    \frac{1}{2\pi} \int_{S^1} h_0\, d\theta = \bar{h}_0
    \quad \text{and} \quad
    \|h_0 - \bar{h}_0\|_{H^1(S^1)} \leq \eps,
    \quad
    \eps \in (0,\eps_0),
\end{equation*}
let $h$ be a positive weak solution to \eqref{eq:PDE_alpha>1} on $(0,T)$. Then $h$ satisfies the following. 
\begin{itemize}
    \item[(i)] The remainder $\phi(t,\theta) = h - \bar{h}_0 - h_{\pm 1}(t) e^{\pm i\theta}$ satisfies the decay estimate
    \begin{equation*}
        \left\|\phi\right\|_{C([\bar{t}/4,\bar{t}];H^1(S^1))}
        \leq
        C\Lambda_\eps(\bar{t})
        =
        \frac{\eps}{\bigl(1 + \eps^{\alpha-1} \bar{t}\bigr)^\frac{1}{\alpha-1}},
        \quad
        0 < \bar{t} \leq T.
    \end{equation*}
    \item[(ii)] The derivatives of the Fourier coefficients $h_{\pm 1}$ satisfy the $L_1((0,T))$-bound
    \begin{equation*}
        \int_0^T |\partial_t h_{\pm 1}(t)|\, dt
        \leq
        C \eps
        \left(1 + \log\bigl(\eps^{1-\alpha}\bigr)\right).
    \end{equation*}
\end{itemize}
\end{lemma}


\begin{proof}
In this proof we consider as a subsequence $(h^\sigma)_\sigma$ the subsequence $(h^{\sigma_{k_j,j}})_j$ introduced in the proof of Theorem \ref{thm:global_ex_original}.

\noindent (i) 
Note that in the proof of Lemma \ref{lem:limit_flux} we find that the energy-dissipation equality is preserved in the limit $\sigma \to 0$. Since
the proofs of Lemma \ref{lem:E-E_min_J} and Corollary \ref{cor:energy_decay} for solutions $h^\sigma$ to the regularised problem \eqref{eq:PDE_regularised} do only rely on the energy-dissipation inequality for the corresponding solution, we can directly carry them over to solutions $h$ of the original problem \eqref{eq:PDE_alpha>1}, as obtained in the previous section. Consequently, for any fixed $T > 0$, we have the estimate
\begin{equation*}
    \|\phi\|_{C([\bar{t}/4,\bar{t}];H^1(S^1))}
    \leq 
    C \Lambda_\eps(\bar{t}),
    \quad 
    0 \leq \bar{t} \leq T.
\end{equation*}

\noindent (ii) Due to Lemma \ref{lem:convergence} (iii) we have the convergence 
\begin{equation}
    \partial_\theta h^\sigma
    \rightharpoonup
    \partial_t h
    \quad \text{weakly in } L_\frac{\alpha+1}{\alpha}\bigl((0,T);\bigl(W^1_{\alpha+1}(S^1)\bigr)^\prime\bigr).
\end{equation}
Thus, for all test functions $\xi \in L_{\alpha+1}((0,T))$ and  $\zeta_{\pm 1}(t,\theta) = \xi(t) e^{\pm i \theta}, \zeta_{\pm 1} \in L_{\alpha+1}\bigl((0,T);W^1_{\alpha+1}(S^1)\bigr)$ we find that
\begin{equation*}
    \int_0^T \partial_t h^\sigma_{\pm 1}(t)\, \xi(t)\, dt
    =
    \int_0^T \langle \partial_t h^\sigma | e^{\pm \theta} \rangle\, \xi(t)\, dt
    \longrightarrow 
    \int_0^T \langle \partial_t h | e^{\pm \theta} \rangle\, \xi(t)\, dt
    =
    \int_0^T \partial_t h_{\pm 1}(t)\, \xi(t)\, dt,
    \quad \text{as } \sigma \to 0,
\end{equation*}
where $\langle \cdot|\cdot\rangle$ denotes the dual pairing in $\bigl(W^1_{\alpha+1}(S^1)\bigr)^\prime \times W^1_{\alpha+1}(S^1)$. This means that, for all $0 < T < \infty$, we have convergence
\begin{equation*}
    \partial_t h^\sigma_{\pm 1}
    \rightharpoonup
    \partial_t h_{\pm 1}
    \quad \text{weakly in } L_\frac{\alpha+1}{\alpha}((0,T)) \hookrightarrow L_1((0,T)),\quad \text{ as } \sigma \to 0,
\end{equation*}
for all $0 < T < \infty$.
Using weak lower-semicontinuity of the norm and Proposition \ref{prop:estimate_Fourier}, together with Lemma \ref{lem:T_larger_t_ast}, we end up with the $L_1((0,T))$-bound 
\begin{equation*}
    \|\partial_t h_{\pm 1}\|_{L_1((0,T))}
    \leq
    \liminf_{\sigma \to 0}
    \|\partial_t h^\sigma_{\pm 1}\|_{L_1((0,T))}
    \leq
    C \eps \left(1 + \log\bigl(\eps^{1-\alpha}\bigr)\right)
\end{equation*}
for the derivatives of Fourier coefficients $h_{\pm 1}$.
\end{proof}


The precise statement of the stability result reads as follows.


\begin{theorem}\label{thm:global-ex_stability}
Let $\alpha > 1$ be fixed.
There exists an $\eps_0 > 0$ such that for all $\eps \in (0,\eps_0)$ and all initial values $h_0\in H^1(S^1)$ with 
\begin{equation*}
    \frac{1}{2\pi}\int_{S^1} h_0\, d\theta = \bar{h}_0
    \quad \text{and} \quad
    \|h_0 - \bar{h}_0\|_{H^1(S^1)} \leq \eps,
\end{equation*}
the problem \eqref{eq:PDE_alpha>1} possesses a globally in time defined weak solution
\begin{equation*}
    h \in C\bigl([0,\infty);H^1(S^1)\bigr)
    \cap
    L_{\alpha+1}\bigl((0,\infty);W^3_{\alpha+1}(S^1)\bigr)
\end{equation*}
in the sense of Definition \ref{def:weak_solution}.
Moreover, these solutions have the following properties.
\begin{itemize}
    \item[(i)] There exist $\xi_{\pm 1} \in \R$ and hence a function $\bar{h}_\infty(\cdot) =  \bar{h}_0 + \xi_{\pm 1} e^{\pm i\cdot} \in H^1(S^1)$ such that 
    \begin{equation*}
        h(t,\cdot) 
        \longrightarrow 
        \bar{h}_\infty(\cdot)
        \quad 
        \text{in } H^1(S^1), \quad \text{as } t \to \infty.
    \end{equation*}
    Moreover, the numbers $\xi_{\pm 1}\in \R$ are given by $\xi_{\pm 1}=\lim_{t\to \infty} h_{\pm 1}(t)$ and they satisfy the estimate
    \begin{equation*}
        |\xi_{\pm 1}| 
        \leq 
        C_1 \eps \bigl(1 + \log\bigl(\eps^{1-\alpha}\bigr)\bigr).
    \end{equation*}
    \item[(ii)] The convergence happens in the following way:
    \begin{equation*}
        \|h(t,\cdot) - \bar{h}_\infty(\cdot)\|_{H^1(S^1)}
        \leq
        \frac{\eps}{\bigl(1 + C_2 \eps^{\alpha-1} t\bigr)^\frac{1}{\alpha-1}}, 
        \quad
        t > 0,
    \end{equation*}
    where $C_2 > 0$ is a positive constant.
\end{itemize}
\end{theorem}


\begin{proof}
\noindent (i) Global existence has already been proved in Theorem \ref{thm:global_ex_original} in the previous section.

\noindent (ii) It remains to prove the existence of a function $\bar{h}_\infty \in H^1(S^1)$ such that $h(t,\cdot) \to \bar{h}_\infty$ in $H^1(S^1)$ as $t \to \infty$.
To this end, we write as above 
\begin{equation*}
   h(t,\theta) = \bar{h}_0 + h_{\pm 1}(t) e^{\pm i \theta} + \phi(t,\theta), \quad t > 0,\ \theta \in S^1. 
\end{equation*}
Thanks to Lemma \ref{lem:uniform_bounds} (i), we know that
\begin{equation} \label{eq:Conv_1}
    \left\|\phi(t,\theta)\right\|_{H^1(S^1)}
    \leq
    C\Lambda_\eps(t) \longrightarrow 0, 
    \quad \text{as } t \to \infty.
\end{equation}
Moreover, in view of Lemma \ref{lem:uniform_bounds} (ii) we have
\begin{equation*}
    \int_0^{\infty} |h^\prime_{\pm 1}(t)|\, dt
    \leq
    C \eps \left(1 + \log\bigl(\eps^{1-\alpha}\bigr)\right).
\end{equation*}
Since $|h_{\pm 1}(0)| \leq \|h_0\|_{H^1} \leq C$, this implies in particular that $|h_{\pm 1}(t)| \leq C$ for all $t \geq 0$. Thus, there exist a sequence $t_k \to \infty$ and two elements $\xi_{\pm 1} \in \R$ such that 
\begin{equation*}
    h_{\pm 1}(t_k) \longrightarrow \xi_{\pm 1}, 
    \quad \text{as } k \to \infty.
\end{equation*}
Moreover, as $t \to \infty$, we find that
\begin{equation} \label{eq:Conv_2}
    |\xi_{\pm 1} - h_{\pm 1}(t)|
    = 
    \lim_{k \to \infty}
    |h_{\pm 1}(t_k) - h_{\pm 1}(t)|
    =
    \lim_{k \to \infty}
    \left|\int_t^{t_k} |h^\prime_{\pm 1}(s)|\, ds\right|
    \leq
    \left|\int_t^{\infty} |h^\prime_{\pm 1}(s)|\, ds\right|
    \longrightarrow 
    0.
\end{equation}
Combining \eqref{eq:Conv_1} and \eqref{eq:Conv_2} and defining $\bar{h}_\infty(\cdot) = \bar{h}_0 + \xi_{\pm 1} e^{\pm i \cdot} \in H^1(S^1)$, we find that
\begin{equation*}
    h(t,\cdot) \longrightarrow \bar{h}_\infty(\cdot)
    \quad \text{in } H^1(S^1),
\end{equation*}
as $t \to \infty$. This completes the proof.
\end{proof}


\bibliographystyle{plain}

\end{document}